\documentclass[a4paper,11pt]{amsart}
\usepackage[T1]{fontenc}
\usepackage[utf8]{inputenc}

\title[Renormalised energies and renormalisable maps]{Renormalised energies and renormalisable singular harmonic maps into a compact manifold on planar domains}

\subjclass[2010]{58E20 (49Q10, 53C22, 55S35)}

\keywords{Harmonic maps, renormalised energy, topological resolution of boundary data}

\author{Antonin Monteil}
\address{
Antonin Monteil \\
University of Bristol\\
School of Mathematics\\
Fry building\\
Woodland Road\\
Bristol BS8 1UG, United Kingdom 
}
\email{antonin.monteil@bristol.ac.uk}

\author{R\'emy Rodiac}
\address{
R\'emy Rodiac\\
Universit\'e Paris-Saclay, CNRS\\
Laboratoire de Math\'ematiques d'Orsay\\
91405, Orsay, France\\
}
\email{remy.rodiac@math.universite-paris-saclay.fr}

\author{Jean Van Schaftingen}
\address{
Jean Van Schaftingen \\
 Universit\'e catholique de Louvain\\
 Institut de Recherche en Math\'ematique et Physique\\
 Chemin du cyclotron 2, L7.01.02\\
1348 Louvain-la-Neuve, Belgium}
\email{jean.vanschaftingen@uclouvain.be}

\thanks{A. Monteil was a postdoctoral researcher (charg\'e de recherches) by the Fonds de la Recherche Scientifique--FNRS over the period 2016--2019; R. Rodiac and J. Van Schaftingen were supported by the Mandat d'Impulsion Scientifique F.4523.17, ``Topological singularities of Sobolev maps'' of the Fonds de la Recherche Scientifique--FNRS: R.Rodiac was partially supported by the ANR project  BLADE Jr. ANR-18-CE40-0023}
\usepackage[hmargin=2.8cm,vmargin=1.8cm]{geometry}

\usepackage{lmodern}
\usepackage{libertine}
\usepackage{microtype}

\usepackage{paralist}
\usepackage{booktabs}

\usepackage{amsfonts}
\usepackage{amssymb}
\usepackage{mathtools}
\usepackage{esint}
\usepackage{bm}
\usepackage{constants}
\usepackage{gensymb}
\usepackage[version=4]{mhchem}
\usepackage{mathrsfs}

\usepackage[hidelinks, pdftitle={Renormalised energies  and renormalisable singular harmonic maps into a compact manifold on planar domains}]{hyperref}

\usepackage[abbrev,backrefs]{amsrefs}

\usepackage[capitalize]{cleveref}

\numberwithin{equation}{section}
\newtheorem{theorem}{Theorem}[section]
\newtheorem{proposition}{Proposition}[section]
\newtheorem{lemma}[proposition]{Lemma}
\newtheorem{corollary}[proposition]{Corollary}
\theoremstyle{definition}
\newtheorem{definition}[proposition]{Definition}
\theoremstyle{remark}
\newtheorem{remark}{Remark}

\newcommand{\st}{\;:\;}

\newcommand{\abs}[1]{\lvert #1 \rvert}
\newcommand{\norm}[1]{\lVert #1 \rVert}

\newcommand{\bigabs}[1]{\bigl\lvert #1 \bigr\rvert}
\newcommand{\Bigabs}[1]{\Bigl\lvert #1 \Bigr\rvert}
\newcommand{\biggabs}[1]{\biggl\lvert #1 \biggr\rvert}

\newcommand{\Deriv}{D}

\newcommand{\Rset}{\mathbb{R}}
\newcommand{\Cset}{\mathbb{C}}
\newcommand{\Nset}{\mathbb{N}}
\newcommand{\Zset}{\mathbb{Z}}
\newcommand{\Sset}{\mathbb{S}}

\newcommand{\Hset}{\mathbb{H}}

\newcommand{\compose}{\,\circ\,}
\newcommand{\dif}{\,\mathrm{d}}
\newcommand{\VMO}{\mathrm{VMO}}
\newcommand{\manifold}[1]{\mathcal{#1}}

\newcommand{\defeq}{\coloneqq}

\newcommand{\equivnorm}[1]{\lambda(#1)}

\newcommand{\Esing}{\mathcal{E}^{\mathrm{sg}}}
\newcommand{\Eext}{\mathcal{E}^{\mathrm{ext}}}
\newcommand{\synhar}[2]{d_{\mathrm{synh}} (#1, #2)}

\newcommand{\Conf}[1]{\operatorname{Conf}_{#1}}

\let\Re=\relax

\DeclareMathOperator{\Re}{Re}

\DeclareMathOperator{\dist}{dist}
\DeclareMathOperator{\id}{id}

\DeclareMathOperator{\dive}{div}
\DeclareMathOperator{\tr}{tr}

\DeclareMathOperator{\syst}{sys}
\DeclareMathOperator{\diam}{diam}

\date{\today}

\begin{document}

\begin{abstract}
We define renormalised energies for maps that describe the first-order asymptotics of harmonic maps outside of singularities arising due to obstructions generated by the boundary data and the mutliple connectedness of the target manifold. The constructions generalise the definition by Bethuel, Brezis and H\'elein for the circle (\emph{Ginzburg--Landau vortices}, 1994). In general, the singularities are geometrical objects and the dependence on homotopic singularities can be studied through a new notion of synharmony. The renormalised energies are showed to be coercive and Lipschitz-continuous. The renormalised energies are associated to minimising renormalisable singular harmonic maps and minimising configurations of points can be characterised by the flux of the stress-energy tensor at the singularities. We compute the singular energy and the renormalised energy in several particular cases.
\end{abstract}

\maketitle

\section{Introduction}
 
Throughout the paper, \(\manifold{N}\) will denote a smooth compact connected Riemannian manifold and \(\Omega\) will denote a bounded connected Lipschitz open subset of \(\Rset^2\).
Given a measurable  map \(g:\partial\Omega\to\manifold{N}\), we say that \(u\) is a \emph{minimising harmonic map} if it is a  minimiser of the 
\emph{Dirichlet energy}
\[
\int_\Omega \frac{\abs{\Deriv u}^2}{2}
\]
among all maps \(u : \Omega \to \manifold{N}\) in the \emph{Sobolev space of mappings} with boundary value \(g\):
\[
W^{1,2}_{g} (\Omega,\manifold{N})\defeq \{u \in W^{1,2}(\Omega,\Rset^\nu)\st u \in \manifold{N} \text{ almost everywhere in \(\Omega\) and 
} \tr_{\partial \Omega} u = g\}.
 \label{eq_space}
\]
Here we have assumed that the manifold \(\manifold{N}\) is isometrically embedded into the Euclidean space \(\Rset^\nu\) for some \(\nu \in \Nset_*\)\footnote{Such an embedding always exist in view of the Nash embedding theorem \cite{Nash_1956}.}, \(\tr_{\partial \Omega}\) denotes the trace operator (see for example 
\cite{Evans_Gariepy_2015}*{\S 4.3}), and \(\abs{\cdot}\) is the Euclidean norm defined by \(\abs{M}^2=\tr(MM^\ast)\) whenever \(M\) is a real matrix (also known as Hilbert--Schmidt or Frobenius norm).

It is known since Morrey \cite{Morrey_1948} that any minimising harmonic map in a planar domain is smooth. However, when the target manifold \(\manifold{N}\) is not simply connected, some boundary data \(g\) cannot be extended continuously to \(\Omega\); for such boundary data, the set \(W^{1,2}_{g} (\Omega,\manifold{N})\) is \emph{empty} \cite{Bethuel_Demengel_1995}*{Theorem 2}.

Multiply-connected -- i.e. connected but not simply connected -- target manifolds arise naturally in several relevant contexts.
They can be sets of order parameters in condensed matter physics, for instance: the circle \(\manifold{N} = \Sset^1\) representing planar spins in superconductivity in the Chern-Simon-Higgs theory, the real projective space \(\manifold{N} = \Rset P^n\) in nematic liquid crystals,
\(\manifold{N} = SU(2)/Q\), where \(Q\) is the quaternion group, for biaxial molecules in nematic phase, and \( \manifold{N} = SU(2)\times SU(2)/H \), where \(H\) is a subgroup of \(SU(2) \times SU(2)\) isomorphic to four copies of \(\mathbb{S}^1\), for superfluid \ce{^3He} in dipole-free phase 
\citelist{\cite{Bauman_Park_Phillips_2012}\cite{Mermin_1979}}.
Such manifolds also appear in computer graphics and meshing algorithms, in which the quotient manifold \(SO (3)/O\) represents \emph{frame-fields} describing the attitude of cubes by a rotation in \(SO(3)\) up to an element of the octahedral group \(O\) describing the rotations preserving the cube \citelist{\cite{Beaufort_Lambrechts_Henrotte_Geuzaine_Remacle_2017}
\cite{Viertel_Osting_2017}}.
The absence, in general, of harmonic maps into these manifolds brings the question about constructing maps that are as much harmonic as possible.

In the case of the circle \(\manifold{N} = \Sset^1\), Bethuel, Brezis and H\'elein have provided maps that are as harmonic as possible in their seminal work on Ginzburg--Landau vortices \cite{Bethuel_Brezis_Helein_1994}. 
They considered the Ginzburg--Landau functional 
\begin{equation}
\label{GinzburgLandauFunctional}
 \int_{\Omega} \frac{\abs{\Deriv u}^2}{2} + \frac{(1 - \abs{u}^2)^2}{4 \varepsilon^2},
\end{equation}
which replaces the hard constraint that \(\abs{u} = 1\) by adding a penalisation term in the functional which forces to satisfy the constraint almost everywhere as \(\varepsilon \to 0\).
They proved that in the limit \(\varepsilon \to 0\), 
minimisers of the Ginzburg--Landau energy \eqref{GinzburgLandauFunctional} converge to a harmonic map into the unit circle \(\Sset^1\) outside a finite number \(\abs{\Deriv }\) of points, where \(\Deriv \in\Zset\) is the topological degree of the boundary datum \(g\), and that this \(\abs{\Deriv }\)-tuple of points minimises a \emph{renormalised energy}. This renormalised energy was defined as the Dirichlet integral of a real valued harmonic map with finitely many point singularities from which singular contributions have been removed \cite{Bethuel_Brezis_Helein_1994}*{(47) in \S I.4}; this harmonic map is characterised as a function whose Laplacian is a sum of Dirac masses at the singular points with Neumann boundary conditions originating in the boundary datum \(g\) \cite{Bethuel_Brezis_Helein_1994}*{(22) in \S I.3}.
This renormalised energy also describes the energies and governs the optimal position of singularities of maps which are harmonic outside shrinking disks \cite{Bethuel_Brezis_Helein_1994}*{Theorem I.7}
and the position of singularities of \(p\)--harmonic maps as \(p \nearrow 2\) \cite{Hardt_Lin_1994}.
The renormalised energy depends continuously on the position of points, and penalises singularities migrating towards the boundary \(\partial \Omega\) and singularities of same-sign degree migrating towards each other \cite{Bethuel_Brezis_Helein_1994}*{Theorem I.10}.
A renormalised energy for the problem of \(p\)--harmonic maps from \(\Omega \subset \Rset^{n + 1}\) into the sphere \(\Sset^{n}\) as \(p \nearrow n+1\) has been been defined by a core radius approach with shrinking disks \cite{Hardt_Lin_Wang_1997}. In the case of two-dimensional Riemannian manifolds, other types of Ginzburg--Landau relaxations giving rise to different renormalized energies between vortices are given in \cites{Ignat_Jerrard_2017,Ignat_Jerrard_2020}, where the vortices correspond to singularities of unit-length harmonic tangent fields due to a non-vanishing Euler-Poincaré characteristic of the vacuum manifold. 

\medskip

The goal of the present work is to define a renormalised energy for the problem of harmonic maps from \(\Omega \subset \Rset^2\) into a general target manifold \(\manifold{N}\). 
Such a wide framework will allow to cover a variety of manifolds that are physically and geometrically relevant.  

We define the renormalised energy of a list of singularities by a shrinking disks approach, as the limit of Dirichlet energies of minimising harmonic maps outside small disks centered at the singularities. More generally, we define the renormalised energy of Sobolev maps with singularities, which are not necessarily harmonic (\cref{def_renormalisable}). This flexibility allows to express the asymptotic expansion of the Ginzburg--Landau energies of low energy states, and not only minimisers, in terms of the renormalised energy \cite{Monteil_Rodiac_VanSchaftingen_GL}. In the case of the classical Ginzburg--Landau functional \eqref{GinzburgLandauFunctional}, this approach was clarified in \cite{goldmanGinzburgLandauModelTopologically2017}*{\S 2.5}.

The first step in the construction is to determine what boundary conditions on the shrinking disks are compatible with the boundary data. 
This is performed through the definition of \emph{topological resolution of boundary data} (\cref{def_resolution} and \cref{def_resolution_VMO}), which are invariant under homotopies in the framework of maps of vanishing mean oscillation (VMO) \citelist{\cite{Brezis_Nirenberg_1995}\cite{Brezis_Nirenberg_1996}}.
In contrast with the case of the circle \(\Sset^1\) whose fundamental group \(\pi_1 (\Sset^1) \simeq \Zset\) is abelian, these topological compatibility conditions cannot be described in general by simple algebraic relations; there is still a characterisation in terms of conjugacy classes, which correspond to free homotopy classes (see \cref{prop:condition_homotopy}).
The topological resolution brings then the notion of singular energy of boundary data (\cref{def_loose_equiv_norm}) and the notion of minimal topological resolution (\cref{minimal_resolution}).
In the abelian case, this description coincides with Canevari and Orlandi’s results that covers in general \(W^{1, k}\)--energies for a \((k - 2)\)--connected target manifold \(\manifold{N}\) \cite{Canevari_Orlandi}.

For the circle \(\Sset^1\), in order to define the renormalised energy by the shrinking disks approach, it is equivalent to either prescribe (parametrised oriented) geodesics at the boundary of each disk near singularities, or to prescribe the topological degree only \cite{Bethuel_Brezis_Helein_1994}*{Theorems I.7 and I.9}. Note that geodesics in the same homotopy class, i.e. of same degree, only differ by a rotation in this case. In a general manifold \(\manifold{N}\), the structure of geodesics in a given homotopy class can be more involved, and this brings us to define \emph{two} renormalised energies of configurations of points:
a \emph{topological renormalised energy}, in which we prescribe the homotopy class near the singularity and a \emph{geometrical renormalised energy} in which we prescribe the boundary datum to be a given geodesic near the singularity. 
The latter geometric energy is more involved, but turns out to be the relevant concept for the asymptotic analysis of Ginzburg--Landau energies \cite{Monteil_Rodiac_VanSchaftingen_GL}.

A first question is whether the geometric renormalised energy depends on prescribed geodesics -- also called \emph{geometric singularities} in the sequel -- \emph{that are homotopic}. We answer this question through a new notion of \emph{synharmony} between geodesics, which measures how much renormalised Dirichlet energy it takes to connect two given maps from \(\Sset^1\) into \(\manifold{N}\). In particular, the geometric and topological renormalised energies coincide with the classical definition for the circle \cite{Bethuel_Brezis_Helein_1994} (see also \cite{Rodiac_Ubillus} for the case of multiply-connected domains). More generally, we show that the geometric and topological renormalised energies coincide for discrete quotients of Lie groups due to the fact that homotopic closed geodesics are always synharmonic in this case (\cref{corollary_discreteSUdQuotient}).

The renormalised energies \emph{penalise} singularities \emph{approaching the boundary}, and singularities \emph{approaching each other} when the resulting combined singularity would have a higher singular contribution to the energy (\cref{proposition_coercivity}) (It can happen  by an arithmetic coincidence that several singularities combine into a single singularity with the same singular energy.) 
The proof relies on a \emph{lower bound on harmonic maps} (\cref{prop_lower_bound_compact}), that we prove in a form which is stronger than needed by the present work and which includes a weak \(L^2\) Marcinkiewicz estimate on the derivative --- or equivalently an estimate in the Lorentz space \(L^{2, \infty}\) --- in view of our further analysis of Ginzburg--Landau relaxations \cite{Monteil_Rodiac_VanSchaftingen_GL}.
We also show that the renormalised energies are locally Lipschitz-continuous functions of the position of the singularities (\cref{theorem_continuity_renormalised}).

The \emph{canonical harmonic map} in Bethuel, Brezis and H\'elein's analysis is a harmonic map outside the prescribed set of singularities, which can be used to recover the actual value of the renormalised energy \cite{Bethuel_Brezis_Helein_1994}*{\S I.8}.
The uniqueness of this canonical map is related to the nice structure of harmonic maps with singularities into \(\Sset^1\), maps that we call \emph{singular harmonic maps}.

In the case of a general compact manifold, we define notions of \emph{renormalisable maps, renormalisable singular harmonic maps, stationary renormalisable singular harmonic maps} and \emph{minimising renormalisable singular harmonic maps} into the manifold \(\manifold{N}\), and we study their connections (\S \ref{section_ren2}). 

In particular if the positions of the singularities minimise the geometric renomalised energy for a given set of prescribed geometric singularities, and if a renormalisable singular harmonic map achieves this renormalised energy, then it has the additional property that the flux of the stress-energy tensor around the singularities vanishes, or equivalently, that the residue of the Hopf differential vanishes at each singularity (see \cref{lemma_Stress_energy_tensor} and \cref{proposition_renormalisable_stationary}).

The topological renormalised energy is always achieved by a minimising singular harmonic map (\cref{proposition_renormalised_singular}), whereas for any position of the singularities, either the geometrical renormalised energy is achieved or it is achieved with other geometric singularities with a difference between the renormalised energies of the two sets of singularities being exactly the synharmonic distance (\cref{proposition_geometric_renormalised_singular}). 
In particular, these results give a characterisation of the topological renormalised energy as the minimum of the geometric renormalised energies with singularities in the given homotopy classes (see \cref{proposition_renormalised_prescribed} and \eqref{linkRenormalisedEnergies}).

In order to perform explicit computations, we identify the topological resolutions of all possible boundary data in several manifolds \(\manifold{N}\) of physical and geometrical interest in \cref{section_examples}.
We conclude this work by expressing  in terms of the renormalised energy into \(\Sset^1\) the renormalised energy for maps into any manifold for boundary data taking monotonically their value into a minimising geodesic with minimal singular energy (\cref{proposition_explicit_renormalised_energy}).

\section{Singular energy and renormalised energy of configuration of points}

\subsection{Topological resolution of the boundary datum}
Given \(k \in \Nset_\ast\), we denote by \(\Conf{k} \Omega\) the configuration space of \(k\) ordered points of \(\Omega\):
\[
 \Conf{k} \Omega
 = \{(a_1, \dotsc, a_k) \in \Omega^k \st 
 a_i = a_j
 \text{ if and only if } i = j\}.
\]
Given \((a_1, \dotsc, a_k)\in \Conf{k} \Omega\),
we define the quantity
\begin{multline}
\label{def_rho_barre}
  \Bar{\rho} (a_1, \dotsc, a_k) \defeq \inf \biggl(  \bigg\{\frac{\abs{a_i - a_j}}{2} \st i, j \in  \{1, \dotsc, k\} \text{ and } i \ne j  \bigg\}\\
  \cup \biggl\{\dist (a_i, \partial \Omega) \st i \in \{1, \dotsc, k\}\biggr\}\biggr),
\end{multline}
in such a way that if \(\rho \in (0, \Bar{\rho} (a_1, \dotsc, a_k))\), we have  \(\Bar{B}_{\rho} (a_i) \cap \Bar{B}_{\rho} (a_j) = \emptyset\) for each \(i, j \in  \{1, \dotsc, k\}\) such that \(i \ne j\) and
\(\Bar{B}_\rho (a_i) \subset \Omega\) for each \(i \in \{1, \dotsc, k\}\), and thus the set \(\Omega \setminus \bigcup_{i = 1}^k \Bar{B}_\rho (a_i)\) is connected and has
a Lipschitz boundary \(\partial (\Omega \setminus \bigcup_{i = 1}^k \Bar{B}_\rho (a_i)) = \partial \Omega \cup \bigcup_{i = 1}^k \partial B_\rho (a_i) \).

\begin{definition}
\label{def_resolution}
Given a Lipschitz bounded domain \(\Omega \subset \Rset^2\) and \(k\in\Nset_\ast\), we say that \((\gamma_1, \dotsc, \gamma_k)\in \mathcal{C}(\Sset^1, \manifold{N})^k\) is a \emph{topological resolution} of  \(g \in \mathcal{C}(\partial \Omega, \manifold{N})\) whenever there exist \((a_1, \dotsc, a_k)\in \Conf{k} \Omega\), a radius \(\rho \in (0, \Bar{\rho} (a_1, \dotsc, a_k))\), and a continuous map \(u \in \mathcal{C}(\Omega \setminus \bigcup_{i = 1}^k \Bar{B}_\rho (a_i), \manifold{N})\) such that \(u \vert_{\partial \Omega} = g\) and for each \(i \in \{1, \dotsc, k\}\), \(u (a_i + \rho\,\cdot) \vert_{\Sset^1} = \gamma_i\).
\end{definition}

The notion of topological resolution is independent of the order of the curves in the \(k\)-tuple \((\gamma_1,\dotsc,\gamma_k)\). Moreover, if \((b_1,\dotsc,b_k)\in \Conf{k} \Omega\) and \(\eta\in (0, \Bar{\rho} (b_1,\dotsc, b_k))\), then there exists a homeomorphism between \(\Omega\setminus \bigcup_{i = 1}^k \Bar{B}_{\rho}(a_i)\) and \(\Omega\setminus \bigcup_{i = 1}^k \Bar{B}_{\eta} (b_i)\) that shows that the statement of the previous definition is independent of the choice of the points and of the radius.
Similarly, if for each \(i \in \{1, \dotsc, k\}\) the map \(\Tilde{\gamma}_i\) is homotopic to \(\gamma_i\) and if \(g\) is homotopic to \(\Tilde{g}\), then, by definition of homotopy and since an annulus is homeomorphic to a finite cylinder, we have that \(\Tilde{\gamma}_1,\dotsc,\Tilde{\gamma}_k\) is also a topological resolution of \(\Tilde{g}\).
Hence the property of being a topological resolution is invariant under homotopies.

\Cref{def_resolution} can be extended to the case where one has \(g \in W^{1/2, 2} (\partial \Omega, \manifold{N})\) and \(\gamma_1, \dotsc, \gamma_k \in W^{1/2, 2} (\Sset^1, \manifold{N})\).
Indeed, when \(\Gamma \simeq \Sset^1\) is a closed curve, maps in \(W^{1/2, 2} (\Gamma, \manifold{N})\) are in the space \(\VMO (\Gamma, \manifold{N})\), whose path-connected components are known to be the closure in \(\VMO (\Gamma, \manifold{N})\)
of path-connected components of \(\mathcal{C} (\Gamma, \manifold{N})\) \cite{Brezis_Nirenberg_1995}*{Lemma A.23}.

\begin{definition}
\label{def_resolution_VMO}
Given a Lipschitz bounded domain \(\Omega \subset \Rset^2\) and \(k\in \Nset_\ast\), we say that \((\gamma_1, \dotsc, \gamma_k)\in \VMO (\Sset^1, \manifold{N})^k\) is a \emph{topological resolution} of \(g \in \VMO (\partial \Omega, \manifold{N})\) whenever \((\gamma_1, \dotsc, \gamma_k)\) is homotopic in \(\VMO (\Sset^1, \manifold{N})^k\) to a \emph{topological resolution} \((\Tilde{\gamma}_1, \dotsc, \Tilde{\gamma}_k)\) of a map \(\Tilde{g} \in \mathcal{C}(\partial \Omega, \manifold{N})\) which is homotopic to \(g\) in \(\VMO (\partial \Omega, \manifold{N})\).
\end{definition}

Since \cref{def_resolution} is invariant under homotopies and continuous maps are homotopic in \(\VMO\) if and only if they are homotopic through a continuous homotopy, \cref{def_resolution_VMO} generalises \cref{def_resolution}.

In the particular case where we have \((\gamma_1, \dotsc, \gamma_k) \in W^{1/2, 2} (\Sset^1, \manifold{N})^k \subset \VMO (\Sset^1, \manifold{N})^k\) and \(g \in W^{1/2, 2} (\partial \Omega, \manifold{N}) \subset \VMO (\partial \Omega, \manifold{N})\), topological resolutions can be characterised through the existence of an extension in \(W^{1, 2} (\Omega \setminus \bigcup_{i = 1}^k \Bar{B}_{\rho} (a_i), \manifold{N})\).

\begin{proposition}\label{prop:homotopy C0=homotopyH1/2}
Given a Lipschitz bounded domain \(\Omega \subset \Rset^2\), \(k\in \Nset_\ast\), \((a_1, \dotsc, a_k)\in \Conf{k} \Omega\) and \(\rho \in (0, \Bar{\rho} (a_1, \dotsc, a_k))\), we have that \(( \gamma_1, \dotsc, \gamma_k) \in W^{1/2,2}(\mathbb{S}^1,\mathcal{N})^k\) is a topological resolution of \(g \in W^{1/2,2}(\partial \Omega,\mathcal{N})\) if and only if there exists a map \(u \in W^{1,2}(\Omega \setminus \bigcup_{i = 1}^k \Bar{B}_\rho (a_i), \manifold{N})\) such that \(\tr_{\partial \Omega} = g\) and for each \(i \in \{1, \dotsc, k\}\), \(\tr_{\Sset^1}u (a_i + \rho\,\cdot)  = \gamma_i\).
\end{proposition}
\begin{proof}[Sketch of the proof]
Assuming that the maps form a topological resolution, we can first extend the boundary datum to a map which is in \(W^{1,2}\) of a small neighbourhood of the boundary of \(\Omega \setminus \bigcup_{i = 1}^k \Bar{B}_\rho (a_i)\) and smooth away from the boundary \cite{Bethuel_Demengel_1995}*{Theorem 2}; the map can then be extended smoothly to the rest of \(\Omega \setminus \bigcup_{i = 1}^k \Bar{B}_\rho (a_i)\) since we have a topological resolution.

Conversely, since \(\Omega \subset \Rset^2\), the map \(u\) can be approximated by smooth maps \cite{SchoenUhlenbeck1982}*{\S 3}; by continuity of the traces the restrictions of the approximations are eventually homotopic to the traces of \(u\).
\end{proof}

Topological resolutions can be characterised algebraically in the fundamental group \(\pi_1 (\manifold{N})\). We recall that each homotopy class of maps from \(\Sset^1\) to \(\manifold{N}\) is associated to a conjugacy class of the fundamental group \(\pi_1 (\manifold{N})\) (see for example \citelist{\cite{Hatcher_2002}*{Exercise 1.1.6}\cite{Mermin_1979}*{\S II}}). If \(\Gamma \subset \Rset^2\) is an embedded compact connected closed curve, by using an orientation preserving homeomorphism between \(\Gamma\) and \(\Sset^1\) we identify a map from \(\Gamma\) to \(\mathcal{N}\) with a map from \(\mathbb{S}^1\) to \(\mathcal{N}\).

\begin{proposition}\label{prop:condition_homotopy}
Let \( \Omega\) be a Lipschitz bounded domain. We write \(\partial \Omega = \bigcup_{i = 0}^\ell \Gamma_i\), where \(\ell \geq 0\), \( \Gamma_0, \Gamma_1, \dotsc, \Gamma_\ell\) are connected embedded compact curves in \(\Rset^2\) such that \(\Gamma_0\) is the outer component of \(\partial \Omega\) and, when \(\ell \geq 1\), for each \(i \in \{1, \dotsc, \ell\}\), the sets \(\Gamma_i\) are the inner components of the boundary.
The list \((\gamma_1, \dotsc, \gamma_k) \in \VMO(\mathbb{S}^1,\manifold{N})^k\) is a topological resolution of \(g\in \VMO(\partial \Omega,\manifold{N})\) if  and only if
\begin{itemize}
\item[(a)]\label{it:1hom}  for every \(i \in \{1, \dotsc, k\}\), there exists \(h_i \in \pi_1 (\manifold{N})\) belonging to the conjugacy class of \(\pi_1 (\manifold{N})\) associated to \(\gamma_i\),
\item[(b)]\label{it:2hom} for every \(j \in \{0, \dotsc, \ell\}\) there exists \(g_j \in \pi_1 (\manifold{N})\) belonging to the conjugacy class of \(\pi_1 (\manifold{N})\) associated to \(g \vert_{\Gamma_j}\),
\end{itemize}
  such that
\begin{equation}\label{eq:condition_homotopy}
 h_1 \dotsb h_k \cdot g_1 \dotsb g_\ell = g_0.
\end{equation}
\end{proposition}

All the implicit isomorphisms between homotopy groups of curves are built by taking the orientation inherited from that of the plane.
For an equivalent condition involving conjugacy classes of the fundamental group we refer to \cite{Canevari_2015}*{Lemma 2.2}.

Proposition \ref{prop:condition_homotopy} tells us when a map \(g \in \mathcal{C}(\partial \Omega,\mathcal{N})\) can be extended into a map \(u \in \mathcal{C}(\overline{\Omega},\mathcal{N})\) with \(u \vert_{\partial \Omega}= g\). When \(\Omega\) is simply connected a map  \(g \in \mathcal{C}(\partial \Omega,\mathcal{N})\) possesses a continuous extension in \(\Omega\) if and only if it is (freely) homotopic to a constant. Thanks to Proposition \ref{prop:homotopy C0=homotopyH1/2}, the algebraic conditions can be transposed to the case of \(H^{1/2}\) boundary data.

Although the product in the fundamental group \(\pi_1 (\manifold{N})\) is nonabelian in general --- and in most of our model situations afterwards --- the order of the product in \eqref{eq:condition_homotopy} is not important since the factors can be chosen freely in conjugacy classes; algebraically, this is related to the fact that in any group \(g \cdot h = (g \cdot h \cdot g^{-1})\cdot g = h \cdot (h^{-1} \cdot g \cdot h)\). 

Computationally, checking condition \eqref{eq:condition_homotopy} seems to require testing all the possible elements in conjugacy classes, leading typically to a cost which is exponential in \(k\) unless \(\pi_1 (\manifold{N})\) is abelian, as it is the case for the Ginzburg--Landau functional (\(\manifold{N} = \Sset^1\) and \(\pi_1 (\Sset^1) \simeq \Zset\)) and the Landau--de Gennes functional (\(\manifold{N} = \Rset P^2\) and \(\pi_1 (\Rset P^2) = \Zset / 2 \Zset\)).

In the case \( \mathcal{N}=\mathbb{S}^1\), \(\pi_1(\mathcal{N})=\mathbb{Z}\), condition \eqref{eq:condition_homotopy} can be expressed in terms of the degree. It reads 
\[
\sum_{i=1}^k \deg (\gamma_i,\mathbb{S}^1)+\sum_{j=1}^\ell \deg (g \vert_{\Gamma_j},\Gamma_j)= \deg(g \vert_{\Gamma_0}, \Gamma_0),
\]
where we oriented all the curves anti-clockwise. More generally, in the case where \( \pi_1(\manifold{N})\) is abelian, by using an additive notation, condition \eqref{eq:condition_homotopy} can be expressed in a similar way to the one with the degrees. 

The discussion above is related to the polygroup structure of the set of conjugacy classes in \(\pi_1 (\manifold{N})\) \citelist{\cite{Campaigne_1940}\cite{Dietzman_1946}}. In general, the product of two conjugacy classes \(c,c'\) in a group \(G\), denoted by \(c\ast c'\), is the set of all conjugacy classes that are included in the set \(c\cdot c'\defeq\{x\cdot x'\st x\in c, x'\in c'\}\). The set \(c\cdot c'\) can be shown to be stable by conjugation; in particular, the conjugacy class of some element \(g\in G\) belongs to \(c\ast c'\) if and only if \(g\in c\cdot c'\). This product provides the set of conjugacy classes with the structure of a commutative polygroup. It is in particular associative, and the condition of \cref{prop:condition_homotopy} reads \(c_0\in c_1\ast\dotsb \ast c_\ell\ast\Tilde{c}_1\ast\dotsb \ast\Tilde{c}_k\), where \(c_0,c_1, \dotsc, c_\ell,\Tilde{c}_1\dotsc \Tilde{c}_k\) are the conjugacy classes in \(\pi_1(\manifold{N})\) associated to \(g_0, g_1, \dotsc, g_\ell,h_1,\dotsc,h_k\). The fact that the set of conjugacy classes has the structure of a polygroup rather than a group is related to the fact that some loops \(\gamma_1,\dotsc,\gamma_k\) and  some inner boundary data \(g_{\vert\Gamma_1},\dotsc,g_{\vert\Gamma_\ell}\) can be compatible with two exterior boundary data \(g_{\vert\Gamma_0}\), \(\Tilde{g}_{\Gamma_0}\) which belong to different free-homotopy classes.

\begin{proof}[Proof of \cref{prop:condition_homotopy}]
We first assume that \( (\gamma_1, \dotsc, \gamma_k)\) is a topological resolution of \(g\). This means that there exists a map \(u \in \mathcal{C}(\overline{\Omega}\setminus \bigcup_{i=1}^k B_\rho(a_i),\mathcal{N})\) such that \( u \vert_{\partial \Omega}=g\) and \(u \vert_{ \partial B_\rho(a_i)} =\gamma_i(a_i+\rho \cdot)\) where \( (a_1,\dotsc, a_k)\) are points in \(\Conf{k} \Omega\) and \(0<\rho<\Bar{\rho}(a_1,\dotsc,a_k)\). We take a point \(x \in \Omega \setminus \bigcup_{i=1}^k \Bar B_\rho(a_i)\), we consider paths \( c_1,\dotsc,c_k\) joining \(x\) to each loop \(\partial B_\rho(a_i)\), \(1\leq i\leq k\), and paths \( c_0',c'_1,\dotsc,c'_\ell\) joining \(x\) to each \(\Gamma_j\), \(0 \leq j \leq \ell\), in such a way that no paths among \(c_1,\dotsc,c_k, c_0',c'_1,\dotsc,c'_\ell\) intersect each other. Now we consider the loop
\[ 
\alpha\defeq  (c_1 \ast \alpha_1 \ast \Bar{c}_1)\ast \dotsb \ast (c_k \ast \alpha_k \ast \Bar{c}_k) \ast (c'_1 \ast \alpha'_1 \ast \Bar{c}'_1) \ast \dotsb \ast (c_\ell' \ast \alpha'_\ell \ast \Bar{c}_\ell')
\]
where \( \alpha_i:[0,1] \rightarrow \partial B_\rho(a_i)\) are parametrisations of \( \partial B_\rho(a_i)\), \(1 \leq i\leq k\), \( \alpha_i':[0,1] \rightarrow \Gamma_j\), are parametrisations of \( \Gamma_j, 1\leq j \leq \ell\), \(\ast\) denotes the concatenation of paths and \(\Bar{c}\) denotes the inverse of the path \(c\). 
The image by \(u\) of the loop \(\alpha\) is a loop in \(\mathcal{N}\) passing by \(u(x)\in \mathcal{N}\). 
The homotopy class in \(\pi_1(\mathcal{N}, u(x))\) of this loop, denoted by \( [u\compose \alpha]\), is equal to the product of the homotopy classes \( [u \compose  (c_1 \ast \alpha_1 \ast \Bar{c}_1)]\dotsm [ u \compose (c_k \ast \alpha_k \ast \Bar{c}_k)] \cdot [u \compose (c_1' \ast \alpha_1' \ast \Bar{c}'_1)] \dotsm  [u \compose (c_\ell' \ast \alpha_1' \ast \Bar{c}'_\ell )]\). 
Now the loop \(u \compose \alpha\) is freely homotopic to \( g \vert_{\Gamma_{0}}\), via the map \(u \in \mathcal{C}(\overline{\Omega}\setminus \bigcup_{j=1}^k B_\rho(a_i),\mathcal{N})\). (Note that the set \(\overline{\Omega}\setminus \bigl(\operatorname{Im}(\alpha)\cup\bigcup_{j=1}^k B_\rho(a_i)\bigr)\)
is homeomorphic to a finite cylinder.) 
Because of the correspondence between free homotopy classes and conjugacy classes of \(\pi_1(\mathcal{N})\) this means that \( [u\compose \alpha]\) belongs to the conjugacy class of \( [ c'_0 \ast g \vert_{\Gamma_{0}} \ast \Bar{c}'_0]\). 
This last condition translates into the announced relation \eqref{eq:condition_homotopy}.

Conversely we assume that \( (\gamma_1,\dotsc,\gamma_k)\) is given such that \eqref{eq:condition_homotopy} holds. Then, with the same notation as in the first part of the proof, thanks to 
\eqref{eq:condition_homotopy}, we see that the loops \( g \vert_{\Gamma_0}\) and \( (\gamma_1 \compose (c_1 \ast \alpha_1 \ast \Bar{c}_1)) \ast \dotsb \ast (\gamma_k \compose (c_k \ast \alpha_k \ast \Bar{c}_k))\ast (g \vert_{\Gamma_1} \compose (c_1' \ast \alpha_1' \ast \Bar{c}'_1)) \ast \dotsb \ast (g \vert_{\Gamma_\ell} \compose (c_\ell' \ast \alpha_1' \ast \Bar{c}'_\ell))\) are freely homotopic. This homotopy can be used to construct the desired extension.
\end{proof}

\subsection{Singular energy}
We define the minimal length in the homotopy class of \(\gamma\in\VMO (\Sset^1,\manifold{N})\) as
\begin{equation}
\label{eq_kahz1OeBa9gooqu6ep1aixei}
\equivnorm{\gamma} \defeq  \inf \Bigl\{ \int_{\Sset^1} \abs{\Tilde{\gamma}'} \st \Tilde{\gamma} \in \mathcal{C}^1 (\Sset^1, \manifold{N}) \text{ and } \gamma \text{ are homotopic in \(\VMO (\Sset^1, \manifold{N})\)} \Bigr\},
\end{equation}
where \(\abs{\cdot}\) is the norm induced by the Riemannian metric \(g_\manifold{N}\) of the manifold \(\manifold{N}\) on each fiber of the tangent bundle \(T \manifold{N}\), or equivalently, since \(\mathcal{N}\) is isometrically embedded in \(\Rset^\nu\), \( \abs{\cdot}\) also denotes the Euclidean norm. For example, for \(D \in \mathbb{N}\), if we consider the maps \( \gamma_\Deriv :\mathbb{S}^1 \rightarrow \mathbb{S}^1\) given for \(z \in \Sset^1 \subset \Cset\) by \( \gamma_\Deriv(z)=z^\Deriv\) then \( \equivnorm{\gamma_\Deriv}=2\pi \abs{\Deriv }\).

By the Cauchy--Schwarz inequality and by a classical re-parametrization, we have
\begin{equation}
  \label{eq_Eojei1cooruef6uiP}
 \inf \Bigl\{ \int_{\Sset^1} \frac{\abs{\Tilde{\gamma}'}^2}{2} \st \Tilde{\gamma} \in \mathcal{C}^1 (\Sset^1, \manifold{N}) \text{ and } \gamma \text{ are homotopic in \(\VMO (\Sset^1, \manifold{N})\)} \Bigr\}
 = \frac{\equivnorm{\gamma}^2}{4 \pi}.
\end{equation}
In particular, if \(\gamma\) is a minimising closed geodesic, then
\begin{equation}
\label{eq_ji7Aire7IeVaesi5eiX6dahJ}
\equivnorm{\gamma}= \int_{\Sset^1} \abs{\gamma'} = \sqrt{4 \pi \int_{\Sset^1} \frac{\abs{\gamma'}^2}{2}} = 2 \pi  \norm{\gamma'}_{L^\infty},
\end{equation}
since \(\abs{\gamma'}\) is constant for a minimising geodesic.

The \emph{systole} of the manifold \(\manifold{N}\) is the length of the shortest closed non-trivial geodesic on \(\manifold{N}\):
\begin{equation}
\label{eq_systole}
  \operatorname{sys} (\manifold{N}) \defeq
  \inf \bigl\{\lambda (\gamma) \st \gamma \in \mathcal{C}^1(\Sset^1,\manifold{N}) \text{ is not homotopic to a constant}\bigr\}.
\end{equation}
In particular, for every \(\gamma\in\VMO (\Sset^1,\manifold{N})\), one has
\(
 \equivnorm{\gamma}
 \in \{0\} \cup [\operatorname{sys} (\manifold{N}), +\infty)
\).

We now define the singular energy of a function \(g \in \VMO (\partial \Omega,\manifold{N} )\):

\begin{definition}
\label{def_loose_equiv_norm}
If \(\Omega \subset \Rset^2\) is a Lipschitz bounded domain and \(g \in \VMO ( \partial \Omega, \manifold{N})\), we define the \emph{singular energy}
  \begin{equation*}
    \Esing (g)
    \defeq \inf \Biggl\{ \sum_{i = 1}^k \frac{\equivnorm{\gamma_i}^2}{4 \pi} \st k \in \Nset_\ast \text{ and \((\gamma_1, \dotsc, \gamma_k)\) is a topological resolution of \(g\)}\Biggr\}.
\end{equation*}
\end{definition}

For example, if \(g \in \VMO(\partial \Omega,\mathbb{S}^1)\) and \(\deg (g)=\Deriv \in \mathbb{Z}\) then \(\Esing(g)=\pi |\Deriv|\).
It follows from the definition that \(\Esing\) is invariant under homotopies and that for every \(\gamma \in \VMO (\Sset^1, \manifold{N})\), one has \(\Esing (\gamma)
\le \frac{\equivnorm{\gamma}^2}{4 \pi}\). In general, the singular energy is bounded by Sobolev energies:
if \(\gamma \in W^{1, 2} (\partial \Omega, \manifold{N})\), then
\[
\Esing (\gamma) \le \int_{\partial \Omega} \frac{\abs{\gamma'}^2}{2};
\]
if \(\gamma \in W^{1, 1} (\partial \Omega, \manifold{N})\), then
\begin{equation}
\label{eq_Aikachae4ieSohJ0oozaht7e}
\Esing(\gamma) \le \frac{1}{4\pi} \biggl( \int_{\partial \Omega} \abs{\gamma'}\biggr)^2
\end{equation}
and if \(\gamma \in W^{1/p, p} (\partial \Omega, \manifold{N})\) for some \(p > 1\), 
then \cite{VanSchaftingen_2019}*{Theorem 1.5},
\[
\Esing (\gamma) \le C \abs{\gamma}_{W^{1/p, p} (\partial \Omega)}^{2 p}.
\]

\begin{proposition}
\label{lemma_evolutiondegres}
If \(\Omega\) is a Lipschitz bounded domain and \(\Omega_1, \dotsc \Omega_\ell\) are disjoint Lipschitz bounded domains such that \(\Bar{\Omega}_i \subset \Omega\) and if \(u \in W^{1, 2} (\Omega \setminus \bigcup_{j = 1}^\ell \bar{\Omega}_j,\mathcal{N})\), then
  \begin{equation*}
    \Esing (\tr_{\partial \Omega}u) \le
    \sum_{j = 1}^\ell
    \Esing (\tr_{\partial \Omega_i}u).
  \end{equation*}

\end{proposition}

\begin{proof}
We observe that the juxtaposition of the topological resolutions of \(\tr_{\partial \Omega_1} u, \dotsc, \tr_{\partial \Omega_\ell} u\) form a topological resolution of \(\tr_{\partial \Omega} u\) by \cref{prop:homotopy C0=homotopyH1/2}.
\end{proof}

We define minimal topological resolutions to be optimal resolutions in \cref{def_loose_equiv_norm}:
\begin{definition}
\label{minimal_resolution}
We say that \((\gamma_1,\dotsc,\gamma_k)\) is a \emph{minimal topological resolution} of \(g\)
if \((\gamma_1,\dotsc,\gamma_k)\) is a topological resolution of \(g\), 
\[
\Esing (g) = \sum_{i = 1}^k \frac{\equivnorm{\gamma_i}^2}{4 \pi}
\]
and \(\equivnorm{\gamma_i} > 0\) for every \(i \in \{1, \dotsc, k\}\).
\end{definition}

\begin{definition}
A closed curve \(\gamma \in \mathcal{C}(\Sset^1, \manifold{N})\) is \emph{atomic} whenever \((\gamma)\) is a minimal topological resolution of \(\gamma\).
\end{definition}

It will appear in the examples that minimal topological resolutions for a given \(g\) are not necessarily unique up to homotopy;
even being atomic does not exclude the existence of a minimal topological resolution into several maps (see \S\ref{sect_gooWae7aNu2roatoozahb0al} below).

\subsection{Renormalised energies of configurations of points}
\label{section_ren1}
Given a bounded Lipschitz domain \(\Omega \subset \Rset^2\), an integer \(k \in \Nset_\ast\), \((a_1, \dotsc, a_k) \in \Conf{k} \Omega\), a radius \(\rho \in (0, \Bar{\rho} (a_1, \dotsc, a_k))\), a map \(g \in W^{1/2,2}(\partial \Omega, \manifold{N})\) and a topological resolutions \((\gamma_1,  \dotsc, \gamma_k)\in W^{1/2,2}(\Sset^1, \manifold{N})^k\) of \(g\), we consider the \emph{geometrical energy} outside disks
\begin{multline}
\label{eq_def_renorm_geom_rho} \mathcal{E}_{g, \gamma_1, \dotsc, \gamma_k}^{\mathrm{geom}, \rho} (a_1, \dotsc, a_k)\\
  \defeq  \inf \biggl\{  \int_{\Omega \setminus \bigcup_{i = 1}^k \Bar{B}_\rho (a_i)}
  \frac{\abs{\Deriv u}^2}{2} \st u \in W^{1, 2} \big(\Omega \setminus \bigcup_{i = 1}^k \Bar{B}_\rho (a_i), \manifold{N}\big),\, \tr_{\partial \Omega} u = g \\[-2ex]
  \text{ and for each \(i \in \{1, \dotsc, k\}\)} \tr_{\Sset^1} u (a_i + \rho\,\cdot) = \gamma_i
  \biggr\}
\end{multline}
and the \emph{topological energy} outside disks
\begin{multline}
\label{eq_def_renorm_top_rho}
  \mathcal{E}_{g, \gamma_1, \dotsc, \gamma_k}^{\mathrm{top}, \rho} (a_1, \dotsc, a_k)\\
  \defeq \inf \biggl\{ \int_{\Omega \setminus \bigcup_{i = 1}^k \Bar{B}_\rho (a_i)}
  \frac{\abs{\Deriv u}^2}{2} \st u \in W^{1, 2} \big(\Omega \setminus \bigcup_{i = 1}^k \Bar{B}_\rho (a_i), \manifold{N}\big),\, \tr_{\partial \Omega} u = g \text{ and }\\[-1.5ex]
    \text{ for each \(i \in \{1, \dotsc, k\}\), }\tr_{\Sset^1} u (a_i + \rho\,\cdot)\text{ and }\gamma_i \text{ are homotopic in \(\VMO (\Sset^1, \manifold{N})\)}
  \biggr\}.
\end{multline}

In the next two propositions we show that the geometrical energy and the topological energy enjoy monotonicity properties:
\begin{proposition}
\label{proposition_geometric_non-decreasing}
Let \(k\in\Nset_\ast\), \((a_1, \dotsc, a_k) \in \Conf{k}\Omega\), \(g\in W^{1/2,2}(\partial\Omega,\manifold{N})\) and \((\gamma_1, \dotsc, \gamma_k) \in W^{1/2,2}(\Sset^1, \manifold{N})^k\) be a topological resolution of \(g\). If \(\gamma_1, \dotsc, \gamma_k\) are minimal geodesics, then the function
\[
 \rho\in (0, \Bar{\rho} (a_1, \dotsc, a_k))\mapsto\mathcal{E}_{g, \gamma_1, \dotsc, \gamma_k}^{\mathrm{geom}, \rho} (a_1, \dotsc, a_k) - \sum_{i = 1}^k \frac{\equivnorm{\gamma_i}^2}{4 \pi}  \log \frac{1}{\rho}
\]
is non-decreasing.
\end{proposition}
\begin{proof}
For \(0<\rho<\sigma<\Bar{\rho}(a_1,\dotsc,a_k)\), for every \(u \in W^{1, 2} (\Omega \setminus \bigcup_{i = 1}^k \Bar{B}_\sigma(a_i), \manifold{N})\) such that \(\tr_{\partial \Omega} u = g\) on \(\partial \Omega\)
and \(\tr_{\Sset^1} u (a_i + \sigma\,\cdot)  = \gamma_i\) on \(\Sset^1\) for every \(i \in \{1, \dotsc, k\}\), we define
\(v : \Omega \setminus \bigcup_{i = 1}^k \Bar{B}_\rho (a_i) \to \manifold{N}\) for \(x \in \Omega \setminus \bigcup_{i = 1}^k \Bar{B}_\rho(a_i)\) by
\[
 v (x) \defeq
 \begin{cases}
   \gamma_i \left(\frac{x - a_i}{\abs{x-a_i}}\right) & \text{if \(x \in B_\sigma (a_i)\setminus \Bar{B}_\rho (a_i)\) with \(i \in \{1, \dotsc, k\}\)},\\
   u (x) & \text{otherwise}.
 \end{cases}
\]
For every \(i \in \{1, \dotsc, k\}\), since \(\gamma_i\) is a minimising geodesic, we have by \eqref{eq_ji7Aire7IeVaesi5eiX6dahJ}
\[
\int_{B_\sigma (a_i) \setminus \Bar{B}_{\rho}(a_i)}  \frac{\abs{\Deriv v}^2}{2}
=  \int^\sigma_\rho \biggl(\int_{\Sset^1} \frac{\abs{\gamma_i'}^2}{2}\biggr) \frac{\dif r}{r}
 = \frac{\equivnorm{\gamma_i}^2}{4 \pi}  \log \frac{\sigma}{\rho}.
\]
We thus have \(v \in W^{1, 2} (\Omega \setminus \bigcup_{i = 1}^k \Bar{B}_\rho (a_i), \manifold{N})\) and
\[
  \int_{\Omega \setminus \bigcup_{i = 1}^k \Bar{B}_\rho (a_i)}
  \frac{\abs{\Deriv v}^2}{2}
 = \int_{\Omega \setminus \bigcup_{i = 1}^k \Bar{B}_\sigma (a_i)}
 \frac{\abs{\Deriv u}^2}{2} + \sum_{i = 1}^k \frac{\equivnorm{\gamma_i}^2}{4 \pi}  \log \frac{\sigma}{\rho},
\]
and the conclusion follows.
\end{proof}
\begin{proposition}
\label{proposition_topological_non-increasing}
Let \(k\in\Nset_\ast\), \((a_1, \dotsc, a_k )\in \Conf{k}\Omega\), \(g\in W^{1/2,2}(\partial\Omega,\manifold{N})\) and \((\gamma_1, \dotsc, \gamma_k) \in W^{1/2,2}(\Sset^1, \manifold{N})^k\) be a topological resolution of \(g\).
If \(\gamma_1, \dotsc, \gamma_k\) are minimal geodesics, then the function
\[
 \rho\in (0, \Bar{\rho} (a_1, \dotsc, a_k))\mapsto\mathcal{E}_{g, \gamma_1, \dotsc, \gamma_k}^{\mathrm{top}, \rho} (a_1, \dotsc, a_k) - \sum_{i = 1}^k \frac{\equivnorm{\gamma_i}^2}{4 \pi}  \log \frac{1}{\rho}
\]
is non-increasing.
\end{proposition}

\Cref{proposition_topological_non-increasing} will follow from the following lemma:
\begin{lemma}
\label{lemma_energy_growth_map_radii}
If \(k\in\Nset_\ast\), \((a_1, \dotsc, a_k) \in\Conf{k} \Omega\), \(0 < \rho < \sigma < \Bar{\rho} (a_1, \dotsc, a_k)\), if \(u \in W^{1, 2} (\Omega \setminus \bigcup_{i = 1}^k \Bar{B}_\rho (a_i), \manifold{N})\) and if for every \(i \in \{1, \dotsc, k\}\) the maps \(\tr_{\Sset^1} u (a_i + \rho\,\cdot) \) and \(\gamma_i\) are homotopic, then
\[
\int_{\Omega  \setminus \bigcup_{i = 1}^k \Bar{B}_{\rho} (a_i)} \frac{\abs{\Deriv u}^2}{2} \geq \int_{\Omega  \setminus \bigcup_{i = 1}^k \Bar{B}_{\sigma} (a_i)} \frac{\abs{\Deriv u}^2}{2}
 + \sum_{i = 1}^k \frac{\equivnorm{\gamma_i}^2}{4 \pi}  \log \frac{\sigma}{\rho}.
\]
\end{lemma}

\begin{proof}
We first have, by additivity of the integral,
\[
\int_{\Omega  \setminus \bigcup_{i = 1}^k \Bar{B}_{\rho} (a_i)} \frac{\abs{\Deriv u}^2}{2}
= \int_{\Omega  \setminus \bigcup_{i = 1}^k \Bar{B}_{\sigma} (a_i)} \frac{\abs{\Deriv u}^2}{2}
+ \sum_{i = 1}^k \int_{B_{\sigma} (a_i) \setminus {\Bar{B}_{\rho} (a_i)}} \frac{\abs{\Deriv u}^2}{2}.
\]
Next, for every \(r \in (\rho, \sigma)\), we observe that the map \(\tr_{\Sset^1} u (a_i + r \cdot)\) is homotopic to
\(\gamma_i\).
Hence, by \eqref{eq_Aikachae4ieSohJ0oozaht7e}
\[
\int_{B_{\sigma} (a_i) \setminus {\Bar{B}}_{\rho} (a_i)} \frac{\abs{\Deriv u}^2}{2}
\ge \int_\rho^{\sigma}\int_{\Sset^1}\frac{1}{2} \Bigabs{\frac{\dif}{r\dif \theta}u(a_i+r\theta)}^2r\dif\theta\dif r
  \ge  \int_\rho^{\sigma}  \frac{\equivnorm{\gamma_i}^2}{4 \pi r} \dif r
  = \frac{\equivnorm{\gamma_i}^2}{4 \pi}  \log \frac{\sigma}{\rho}.\qedhere
\]
\end{proof}

Let \(k\in\Nset_\ast\), \(g\in W^{1/2,2}(\partial\Omega,\manifold{N})\) and \((\gamma_1, \dotsc, \gamma_k) \in \mathcal{C}^1 (\Sset^1, \manifold{N})^k\) be a topological resolution of \(g\) such that \(\gamma_1, \dotsc, \gamma_k\) are minimal geodesics.
Since the constraint in the definition of \eqref{eq_def_renorm_top_rho} is weaker than the one in \eqref{eq_def_renorm_geom_rho}, we have for each \(\rho\in (0, \Bar{\rho} (a_1, \dotsc, a_k))\),
\begin{equation}
\label{eq_top_le_geom}
\mathcal{E}_{g, \gamma_1, \dotsc, \gamma_k}^{\mathrm{top}, \rho} (a_1, \dotsc, a_k) - \sum_{i = 1}^k \frac{\equivnorm{\gamma_i}^2}{4 \pi}  \log \frac{1}{\rho}\le \mathcal{E}_{g, \gamma_1, \dotsc, \gamma_k}^{\mathrm{geom}, \rho} (a_1, \dotsc, a_k) - \sum_{i = 1}^k \frac{\equivnorm{\gamma_i}^2}{4 \pi}  \log \frac{1}{\rho},
\end{equation}
so that the left and right hand sides of \eqref{eq_top_le_geom} are bounded.
Therefore, we can define the \emph{geometrical renormalised energy} by
\begin{equation}
\label{eq_def_renorm_geom}
\begin{split}
 \mathcal{E}_{g, \gamma_1, \dotsc, \gamma_k}^{\mathrm{geom}}  (a_1, \dotsc, a_k) &\defeq  \lim_{\rho \to 0} \mathcal{E}_{g, \gamma_1, \dotsc, \gamma_k}^{\mathrm{geom}, \rho} (a_1, \dotsc, a_k) - \sum_{i = 1}^k \frac{\equivnorm{\gamma_i}^2}{4 \pi}  \log \frac{1}{\rho}\\
 &= \inf_{\rho \in (0, \Bar{\rho} (a_1, \dotsc, a_k))}
\mathcal{E}_{g, \gamma_1, \dotsc, \gamma_k}^{\mathrm{geom}, \rho} (a_1, \dotsc, a_k) - \sum_{i = 1}^k \frac{\equivnorm{\gamma_i}^2}{4 \pi}  \log \frac{1}{\rho},
\end{split}
\end{equation}
and the \emph{topological renormalised energy} by
\begin{equation}
\label{eq_def_renorm_top}
\begin{split}
 \mathcal{E}_{g, \gamma_1, \dotsc, \gamma_k}^{\mathrm{top}}  (a_1, \dotsc, a_k) &\defeq  \lim_{\rho \to 0} \mathcal{E}_{g, \gamma_1, \dotsc, \gamma_k}^{\mathrm{top}, \rho} (a_1, \dotsc, a_k) - \sum_{i = 1}^k \frac{\equivnorm{\gamma_i}^2}{4 \pi}  \log \frac{1}{\rho}\\
 &= \sup_{\rho \in (0, \Bar{\rho} (a_1, \dotsc, a_k))}
\mathcal{E}_{g, \gamma_1, \dotsc, \gamma_k}^{\mathrm{top}, \rho} (a_1, \dotsc, a_k) - \sum_{i = 1}^k \frac{\equivnorm{\gamma_i}^2}{4 \pi}  \log \frac{1}{\rho}.
\end{split}
\end{equation}
We immediately have, in view of \eqref{eq_top_le_geom} that
\begin{equation}
\label{eq_comparison_top_geom}
 \mathcal{E}_{g, \gamma_1, \dotsc, \gamma_k}^{\mathrm{top}}  (a_1, \dotsc, a_k)
 \le \mathcal{E}_{g, \gamma_1, \dotsc, \gamma_k}^{\mathrm{geom}}  (a_1, \dotsc, a_k).
\end{equation}

When \(\manifold{N} = \Sset^1\) is the circle and the domain \(\Omega\) is simply connected, the geometrical renormalised energy \(\mathcal{E}_{g, \gamma_1, \dotsc, \gamma_k}^{\mathrm{geom}}\) 
corresponds in Bethuel, Brezis and Hélein's work to the energy \({E}_2\) \cite{Bethuel_Brezis_Helein_1994}*{Theorem I.3}
and the topological renormalised energy \(\mathcal{E}_{g, \gamma_1, \dotsc, \gamma_k}^{\mathrm{top}}\) to the energy \({E}_1\) \cite{Bethuel_Brezis_Helein_1994}*{Theorem I.2}.
Thanks to the underlying linear structure of the problem, both can be computed as solutions of a linear elliptic problem  \cite{Bethuel_Brezis_Helein_1994}*{\S I.4}; due to the more nonlinear character of our general setting, their approach seems quite unlikely to work here.

\section{Dependence on the charges of singularities}
The renormalised energies are functions of the points \( (a_1,\dotsc,a_k)\), of a boundary data \(g \in W^{1/2,2}(\partial \Omega,\mathcal{N})\) and of one of its topological resolution \( (\gamma_1,\dotsc,\gamma_k)\in W^{1/2,2}(\mathbb{S}^1,\mathcal{N})\). The curves \( \gamma_i\) are referred to as \emph{charges} of the singularities \( ((a_1,\gamma_1),\dotsc,(a_k,\gamma_k))\).
\subsection{Dependence on the charges of the topological renormalised energy}

The topological renormalised energy only depends on the homotopy classes of maps near the singularities:

\begin{proposition}
Let \(\Omega\subset\Rset^2\) be a Lipschitz bounded domain, \(g \in W^{1/2,2} (\partial \Omega, \manifold{N})\), \(k\in\Nset_\ast\), \((a_1, \dotsc, a_k)\in\Conf{k}\Omega\), and \(\gamma_1, \dotsc, \gamma_k,\Tilde{\gamma}_1,\dots,\Tilde{\gamma}_k \in W^{1/2,2}(\Sset^1,\manifold{N})\).
If for every \(i \in \{1, \dotsc, k\}\), \(\Tilde{\gamma}_i\) is homotopic to \(\gamma_i\), then
\[
\mathcal{E}_{g,\Tilde{\gamma}_1, \dotsc, \Tilde{\gamma}_k}^{\mathrm{top}}  (a_1, \dotsc, a_k)
= \mathcal{E}_{g,\gamma_1, \dotsc, \gamma_k}^{\mathrm{top}}  (a_1, \dotsc, a_k).
\]
\end{proposition}
\begin{proof}
This follows from the fact that \eqref{eq_def_renorm_top_rho}
is invariant under homotopies and from the definition in \eqref{eq_def_renorm_top}.
\end{proof}

\subsection{Synharmony of maps}
In order to study the dependence on the charges of the geometrical renormalised energy, we introduce the synharmony between geodesics which quantifies how homotopic mappings can be connected through a harmonic map.

\begin{definition}
\label{definition_synharmonic}
The \emph{synharmony} between two given maps \(\gamma, \beta\in W^{1/2,2}(\Sset^1 , \manifold{N})\), is defined as
\begin{multline*}
  \synhar{\gamma}{\beta}
  \defeq \inf \, \biggl\{ \int_{\Sset^1 \times [0, T]} \biggl(\frac{\abs{\Deriv u}^2}{2} - \frac{\equivnorm{\gamma}^2}{8 \pi^2} \biggr)
  \st  T \in (0, +\infty),
  u \in W^{1, 2} (\Sset^1 \times [0, T], \manifold{N}),\\[-1em]
  \tr_{\Sset^1 \times \{0\}} u  = \gamma \text{ and }
  \tr_{\Sset^1 \times \{T\}} u  = \beta \text{ on \(\Sset^1\)} \biggr\}.
\end{multline*}
\end{definition}

The synharmony is an extended pseudo-distance which is continuous with respect to the strong topology in \(W^{1/2,2} (\Sset^1, \manifold{N})\).

\begin{proposition}
\label{proposition_synharmonic_pseudometric}
For every \(\gamma,\beta,\alpha\in W^{1/2,2}(\Sset^1 , \manifold{N})\), one has
\begin{enumerate}[(i)]
  \item \label{item_synhar_vanish}
  \(
  \synhar{\gamma}{\gamma} = 0,
  \)
  \item \label{item_synhar_nonnegative}
 \(
 \synhar{\gamma}{\beta} \ge 0,
 \)
 \item\label{item_synhar_finite}
 \(
 \synhar{\gamma}{\beta} <+\infty
 \) if and only if  \(\gamma\) and \(\beta\) are homotopic in \(\VMO (\Sset^1, \manifold{N})\),
 \item \label{item_synhar_symmetry}
 \(
   \synhar{\gamma}{\beta} = \synhar{\beta}{\gamma},
 \)
 \item \label{item_synhar_triangle}
 \(
   \synhar{\gamma}{\beta} \le \synhar{\gamma}{\alpha} +  \synhar{\alpha}{\beta}.
 \)
 \item \label{item_H1weaker}if the sequence \((\gamma_n)_{n\in\Nset}\) converges to \(\gamma\) strongly in \(W^{1/2,2} (\Sset^1, \manifold{N})\), then
\begin{equation*}
\liminf_{n\to\infty}  \synhar{\gamma}{\gamma_n}=0.
\end{equation*}
\end{enumerate}
\end{proposition}
\begin{proof}
  For \eqref{item_synhar_vanish} we take \(u \in W^{1, 2} (\Sset^1 \times [0, 1], \manifold{N})\) such that \(\tr_{\Sset^1 \times \{0\}} u = \gamma\), and we define for \(T \in (0, 2)\),
  \(u_{T} (x, t) \defeq u (x, T/2 - \abs{t - T/2})\).   Since for each \(0\leq t\leq T\) we have \(0\leq T/2 - \abs{t - T/2}\leq 1\) the map \(u_T\) is well-defined and satisfies \(\tr_{\Sset^1 \times \{0\}} u_T =\tr_{\Sset^1 \times \{T\}} u_T  = \gamma\) and 
  \[
   \int_{\Sset^1 \times [0, T]} \biggl(\frac{\abs{\Deriv u_T}^2}{2} -  \frac{\equivnorm{\gamma}^2}{8 \pi^2} \biggr)
   =2 \int_{\Sset^1 \times [0, T/2]} \biggl(\frac{\abs{\Deriv u}^2}{2} -  \frac{\equivnorm{\gamma}^2}{8 \pi^2}\biggr);
  \]
the conclusion follows from Lebesgue's dominated convergence theorem as \(T \to 0\).

  The property \eqref{item_synhar_nonnegative} follows from the fact that if \(u \in W^{1, 2} (\Sset^1 \times [0, T], \manifold{N})\) and if \(\tr_{\Sset^1 \times \{0\}} u = \gamma\), then for almost every \(t \in [0, T]\), the map \(u (\cdot, t)\) is homotopic to \(\gamma\) and thus in view of \eqref{eq_Eojei1cooruef6uiP}
\[
  \int_{\Sset^1} \frac{\abs{\Deriv u (\cdot, t)}^2}{2} \ge \frac{\equivnorm{\gamma}^2}{4 \pi} = \int_{\Sset^1}\frac{\equivnorm{\gamma}^2}{8 \pi^2}.
\]

The finiteness property \eqref{item_synhar_finite}, 
follows from \cref{definition_synharmonic} and the trace and extension theory for Sobolev mappings \cite{Bethuel_Demengel_1995}.

For \eqref{item_synhar_symmetry}, we rely on the fact that \(\equivnorm{\gamma} = \equivnorm{\beta}\) if \(\gamma\) and \(\beta\) are homotopic.

In order to prove \eqref{item_synhar_triangle}, given a map \(u \in W^{1, 2} (\Sset^1 \times [0, T], \manifold{N})\) such that \(\tr_{\Sset^1 \times\{0\}} u = \gamma\) and  \(\tr_{\Sset^1 \times \{T\}} u = \alpha\), and a map  \(v \in W^{1, 2} (\Sset^1 \times [0, S], \manifold{N})\), such that  \(\tr_{\Sset^1 \times \{0\}} v = \alpha\) and \(\tr_{\Sset^1 \times \{S\}} v  = \beta\), we define  the map \(w \in W^{1, 2} (\Sset^1 \times [0, T + S],\mathcal{N})\) by \(w (\cdot, t) \defeq u (\cdot, t)\) if \(t \in [0, T]\) and \(w (\cdot, t) \defeq  v (\cdot, t - T)\) if \(t \in [T, T + S]\).

We now check \eqref{item_H1weaker}.
For every \(n \in \Nset\), we let \(v_n : \Sset^1 \times (0, +\infty)\rightarrow \Rset^\nu\) and \(v : \Sset^1 \times (0, +\infty)\rightarrow \Rset^\nu\) be the harmonic extensions of \(\gamma_n\) and of \(\gamma\) respectively.
Since \((\gamma_n)_{n \in \Nset}\) converges strongly to \(\gamma\), then for every \(T \in (0, +\infty)\),
\((v_n \vert_{\Sset^1 \times [0, T]})_{n \in \Nset}\) converges strongly to \(v \vert_{\Sset^1 \times [0, T]}\) in \(W^{1,2}(\Sset^1\times [0,T],\Rset^\nu)\)
and \(v_n \to v\) locally uniformly in \(\Sset^1 \times (0, +\infty)\).
We define the function \(v_{n, T} : [0, T] \to \Rset^\nu\) by
\[
v_{n, T} (x, t)
=
\begin{cases}
  v_n (x, t) & \text{if \(0 \le t \le T/3\)},\\
  \frac{2T - 3t}{T} v_n (x, t) + \frac{3t -T}{T} v (x, T - t) &\text{if \(T/3 \le t \le 2 T/3\)},\\
  v (x, T - t) & \text{if \(2 T/3 \le t \le T\)}.
\end{cases}
\]
We have that \( v_{n,T} \rightarrow v_T\) strongly in \(W^{1,2}(\mathbb{S}^1 \times [0,T],\Rset^\nu)\) and \(v_{n,T} \rightarrow v_{T} \) locally uniformly on \(\mathbb{S}^1\times (0,T]\), where
\[v_T =\begin{cases}
v(x,t) & \text{ if \(0\leq t \leq T/3\)}, \\
\frac{2T-3t}{T}v(x,t)+\frac{3t-T}{T}v(x,T-t) & \text{ if \(T/3 \leq t \leq 2T/3\)}, \\
v(x,T-t) & \text{ if \(2T/3\leq t \leq T\)}.
\end{cases}
\]
and thus, if \((x, t) \in \Sset^1 \times [T/3, 2T/3]\),  
\begin{equation}
\begin{split}
  \abs{\Deriv  v_T (x, t)}
  &\le 
  \bigl(\tfrac{2T-3t}{T}\abs{\Deriv  v(x,t)}+\tfrac{3t-T}{T}\abs{v(x,T-t)}
  +  \tfrac{3}{T}\abs{v (x, t) - v (x, T - t)}\bigr)^2\\
  &\le 2(\tfrac{2T-3t}{T} \abs{\Deriv  v(x,t)}^2 + \tfrac{3t-T}{T}\abs{\Deriv  v(x,T - t)}^2) + \tfrac{18}{T} \abs{v (x, t) - v (x, T - t)}^2.
\end{split}
\end{equation}
By a change of variable we see that
\begin{multline}
\label{eq_TheFei2ahh6tejahdohx8hiw}
\int_{\mathbb{S}^1\times[0,T]} \abs{\Deriv v_T}^2 \\
\leq 2\int_{\mathbb{S}^1\times [0,\frac{T}{3}]} \abs{\Deriv v}^2
+2\int_{\mathbb{S}^1\times[\frac{T}{3},\frac{2T}{3}]}\abs{\Deriv v}^2
+\frac{36}{T^2}\iint_{\mathbb{S}^1\times [\frac{T}{3},\frac{T}{2}]} \abs{v(x,T-t)-v(x,t)}^2\dif t \dif x.
\end{multline}
We now use the fundamental theorem of calculus to write that, for \( T/3 \leq t \leq 2T/3\), \(v(x,T-t)-v(x,t)= \int_t^{T -t} \frac{d}{ds} v(x, s) \dif s\). Hence, by Hölder's inequality we arrive at 
\[
|v(x,T-t)-v(x,t)|^2 \leq \frac{T}{3}\int_{T/3}^{2T/3} \abs{\Deriv v(x, s)}^2 \dif s
\]
and thus 
\begin{equation}
\label{eq_aip2reL5eL7zaef7aeChoosu}
\begin{split}
\frac{36}{T^2}\iint\limits_{\mathbb{S}^1\times [\frac{T}{3}, \frac{T}{2}]} \abs{v(x,T-t)-v(x,t)}^2 \dif x \dif t &\leq \frac{12}{T} \iint\limits_{\mathbb{S}^1\times [\frac{T}{3},\frac{T}{2}]}\int_{T/3}^{2T/3} |\Deriv v(x, s)|^2 \dif s \dif x \dif t \\
& \leq 2\int_{\mathbb{S}^1 \times [\frac{T}{3},\frac{2T}{3}]} \abs{\Deriv v}^2.
\end{split}
\end{equation}
Thus we find by \eqref{eq_TheFei2ahh6tejahdohx8hiw} and \eqref{eq_aip2reL5eL7zaef7aeChoosu} that \( \int_{\mathbb{S}^1\times [0,T]} \abs{\Deriv v_T}^2 \leq  4\int_{\mathbb{S}^1\times [0,T]} \abs{\Deriv v}^2\) and
\[
\lim_{T \to 0} \limsup_{n \to \infty} \int_{\Sset^1 \times [0, T]} \frac{\abs{\Deriv v_{n, T}}^2}{2} 
= \lim_{T \to 0} \int_{\Sset^1 \times [0, T]} \frac{\abs{\Deriv v_{T}}^2}{2} 
= 0.
\]

Since the sequence \((\gamma_n)_{n \in \Nset}\) converges strongly to \(\gamma\) in \(W^{1/2, 2} (\Sset^1, \manifold{N})\) and thus in \(\VMO (\Sset^1, \manifold{N})\), by the characterisation of compact sets in \(\VMO (\Sset^1, \manifold{N})\), \cite{Brezis_Nirenberg_1995}*{Lemma 4}, we have
\[
 \lim_{\delta \to 0} \sup_{n \in \Nset} \sup_{a \in \Sset^1} \fint_{B_\delta (a) \cap \Sset^1}\fint_{B_\delta (a) \cap \Sset^1} \abs{\gamma_n (y) - \gamma_n(x)} \dif y \dif x = 0.
\]
Thus (see \cite{Brezis_Nirenberg_1996}*{Theorem A3.2 and its proof}),
we find
\[
\lim_{T \to 0} \sup \biggl\{\dist (z, \manifold{N}) \st z \in v (\Sset^1 \times (0, T)) \cup \bigcup_{n \in \Nset} v_n (\Sset^1 \times (0, T))\biggr\} = 0.
\]
Therefore, if \(\Pi_{\manifold{N}} : \{z\in\manifold{N}\st \operatorname{dist}(z,\manifold{N})<\delta\} \to \manifold{N}\) is the nearest-point retraction, well-defined and smooth for \(\delta>0\) small enough, for each \(T\) small enough, \(\Pi_{\manifold{N}}\compose v_{n, T}\) is well-defined when \(n \in \Nset\) is large enough and we have
\[
\lim_{T \to 0} \limsup_{n \to \infty} \int_{\Sset^1 \times [0, T]} \frac{\abs{\Deriv  (\Pi_{\manifold{N}}\compose v_{n, T})}^2}{2} = 0.
\]
Hence we have
\[
\limsup_{n \to \infty}
\synhar{\gamma_n}{\gamma}
\le \lim_{T \to 0} \limsup_{n \to \infty} \int_{\Sset^1 \times [0, T]} \biggl(\frac{\abs{\Deriv  (\Pi_{\manifold{N}}\compose v_{n, T})}^2}{2}
- \frac{\equivnorm{\gamma}^2}{8 \pi^2}\biggr) = 0.\qedhere
\]
\end{proof}

The synharmony pseudo-metric is complete:

\begin{proposition}
\label{proposition_synharmony_complete}
If \((\gamma_n)_{n \in \Nset}\) is a sequence in \(W^{1/2, 2} (\Sset^1, \manifold{N})\) such that
\(
 \lim_{n, m \to \infty} \synhar{\gamma_n}{\gamma_m} = 0,
\)
then there exists \(\gamma \in W^{1/2, 2} (\Sset^1, \manifold{N})\) such that \(\lim_{n \to \infty} \synhar{\gamma_n}{\gamma} = 0\).
Moreover either \((\gamma_n)_{n \in \Nset}\) converges strongly to \(\gamma\) in \(W^{1/2, 2} (\Sset^1, \manifold{N})\)
or \(\gamma\) is a geodesic.
\end{proposition}
\begin{proof}
In view of the triangle inequality (\cref{proposition_synharmonic_pseudometric} \eqref{item_synhar_triangle}), it is sufficient to prove the convergence for a subsequence.
Hence, we assume without loss of generality that for every \(n \in \Nset\),
\(
 \synhar{\gamma_n}{\gamma_{n + 1}} < \frac{1}{2^n}
\); in particular, all the maps \(\gamma_n\) are mutually homotopic by \cref{proposition_synharmonic_pseudometric} \eqref{item_synhar_finite}.
By definition of synharmonic distance (\cref{definition_synharmonic}), for each \(n \in \Nset\) there exists
\(L_n \in (0, +\infty)\) and a map \(u_n \in W^{1, 2} (\Sset^1 \times [0, L_n], \manifold{N})\) such that
\(\tr_{\Sset^1 \times \{0\}} u_n  = \gamma_n\) on \(\Sset^1\), \(\tr_{\Sset^1 \times \{L_n\}} u_n  = \gamma_{n + 1}\) on \(\Sset^1\) and
\[
\int_{\Sset^1 \times [0, L_n]} \biggl(\frac{\abs{\Deriv u_n}^2}{2} - \frac{\equivnorm{\gamma_n}^2}{8 \pi^2} \biggr)
  \le \frac{1}{2^{n }}.
\]
We define for each \(n \in \Nset_*\), \(T_n \defeq \sum_{i = 0}^{n - 1} L_i\), \(T_0\defeq 0\), \(T_\infty \defeq \sum_{i \in \mathbb{N}} L_i \in (0, +\infty]\) and the map \(u : [0, T_\infty) \to \manifold{N}\) by
\(u (x, t) \defeq u_n (x, t - T_n)\)  if \(T_n \le t <T_{n + 1}\).

If \(T_\infty < +\infty\), then we set \(\gamma \defeq \tr_{\Sset^1 \times \{T_\infty\}} u\), for which we observe that
\(\lim_{n \to \infty} \synhar{\gamma_n}{\gamma} = 0\) and \((\gamma_n)_{n \in \Nset}\) converges strongly to \(\gamma\) in \(W^{1/2, 2} (\Sset^1, \manifold{N})\).  Indeed, \(\gamma_n= \tr_{\Sset^1 \times \{T_n\}} u =\tr_{\Sset^1 \times \{0\}} u(\cdot,\cdot+T_n) \to \tr_{\Sset^1 \times \{0\}}u(\cdot,\cdot+T_\infty)= \gamma\) as \(n\to\infty\) in \(W^{1/2,2}(\Sset^1,\manifold{N})\) thanks to the continuity of the translations in \(W^{1,2}\) and of the continuity of the trace.

If \(T_\infty = +\infty\), we first observe that, since \(\equivnorm{\gamma_n}=\equivnorm{\gamma_0}\) for every \(n\in\Nset\),
\[
 \int_0^\infty \biggl(\int_{\Sset^1} \frac{\abs{\Deriv u (\cdot, t)}^2}{2} - \frac{\equivnorm{\gamma_0}^2}{8 \pi^2} \biggr) \dif t
 \le \sum_{n = 0}^\infty \frac{1}{2^{n}} = 2.
 \]
Hence, there exists a sequence \((t_n)_{n \in \Nset}\to\infty\) such that
\[
  \lim_{n \to \infty} \int_{\Sset^1} \frac{\abs{u (\cdot, t_n)'}^2}{2} = \frac{\equivnorm{\gamma_0}^2}{4 \pi}.
\]
By compactness in \(W^{1, 2} (\Sset^1, \manifold{N})\), the sequence \((u (\cdot, t_n))_{n \in \Nset}\) converges strongly to some minimising geodesic \(\gamma \in W^{1, 2} (\Sset^1, \manifold{N})\) and it follows then that \((\gamma_n)_{n \in \Nset}\) converges in synharmony to \(\gamma\).
\end{proof}

\begin{proposition}
\label{proposition_compact_synharmony}
If \(K\) is a compact subset of \(W^{1/2,2}(\Sset^1,\manifold{N})\) such that every \(\alpha, \beta \in K\) are homotopic, then
\[
 \sup \,\bigl\{
 \synhar{\alpha}{\beta} \st
 \alpha,\beta\in K
 \bigr\}
 < + \infty.
\]
\end{proposition}

In particular, if \(Y\) is a bounded subset of \(W^{1, 2} (\Sset^1, \manifold{N})\) such that every \(\alpha,\beta\) in \(Y\) are homotopic, then \(Y\) is bounded in synharmony.

\begin{proof}[Proof of \cref{proposition_compact_synharmony}]
Let \((\alpha_n)_{n \in \Nset}\) and \((\beta_n)_{n \in \Nset}\) be maximising sequences for \( \sup \,\bigl\{
 \synhar{\alpha}{\beta} \st
 \alpha,\beta\in K
 \bigr\}\). Then, by compactness of \(K\), there exist \(\alpha_*,\beta_* \in K \subset  W^{1/2,2}(\Sset^1,\mathcal{N})\) such that, up to extraction of a subsequence, \(\alpha_n \rightarrow \alpha\) and \(\beta_n\rightarrow \beta\) in \(W^{1/2,2}(\Sset^1,\Rset^\nu)\). By \cref{proposition_synharmonic_pseudometric}, we find \(\synhar{\alpha_*}{\beta_*} =\sup \,\bigl\{
 \synhar{\alpha}{\beta} \st
 \alpha,\beta\in K
 \bigr\} <+\infty.\)
\end{proof}

\begin{definition}
  \label{definition_synharmonic_maps}
  Two maps \(\gamma, \beta \in W^{1/2, 2} (\Sset^1, \manifold{N})\) are \emph{synharmonic} whenever \(\synhar{\gamma}{\beta} = 0\).
\end{definition}

\begin{proposition}\label{proposition_synharmony_equality_or_geod}
If \(\synhar{\gamma}{\beta}=0\) then  either \(\gamma= \beta\) almost everywhere or \(\gamma\) and \(\beta\) are minimising geodesics.
\end{proposition}

\begin{proof}
By definition of synharmonic maps (\cref{definition_synharmonic_maps}) and of synharmony of maps (\cref{definition_synharmonic}), there exists a sequence \((T_n)_{n \in \Nset}\) in \((0, +\infty)\) and for each \(n \in \Nset\) a map \(u_n \in W^{1, 2} (\Sset^1 \times [0, T_n], \manifold{N})\)
such that \(\tr_{\Sset^1 \times \{0\}} u_n = \gamma\), \(\tr_{ \Sset^1 \times \{T_n\}} u_n  = \beta\) and
\[
\lim_{n \to \infty} \int_{\Sset^1 \times [0, T_n]} 
\biggl(\frac{\abs{\Deriv u_n}^2}{2} - \frac{\equivnorm{\gamma}^2}{8 \pi^2} \biggr)
= 0.
\]
Up to a subsequence we can assume that either \((T_n)_{n \in \Nset}\) converges to \(0\),
or \((T_n)_{n \in \Nset}\) converges to some \(\Bar{T} \in (0, +\infty)\) or \((T_n)_{n \in \Nset}\) diverges towards \(+\infty\).

First, if the sequence \((T_n)_{n \in \Nset}\) converges to \(0\), then by the Cauchy--Schwarz inequality
\[
\int_{\Sset^1} \abs{\gamma - \beta} \le \limsup_{n \to \infty} \int_{\Sset^1 \times [0, T_n]} \abs{\Deriv u_n}
\le \lim_{n \to \infty} \sqrt{2 \pi T_n} \, \biggl(\int_{\Sset^1 \times [0, T_n]} \abs{\Deriv u_n}^2 \biggr)^\frac{1}{2} = 0,
\]
and thus \(\gamma = \beta\).

If the sequence \((T_n)_{n \in \Nset}\) converges to \(\Bar{T}\in (0, +\infty)\), then there exists a map \(u \in W^{1, 2} (\Sset^1 \times [0, \bar{T}], \manifold{N})\)
such that \(\tr_{\Sset^1 \times \{0\}} u = \gamma\), \(\tr_{ \Sset^1 \times \{\Bar{T}\}} u  = \beta\)
and
\[
\int_{\Sset^1 \times [0, \Bar{T}]} \frac{\abs{\Deriv u}^2}{2} - \frac{\equivnorm{\gamma}^2}{8 \pi^2} = 0.
\]
We conclude from this that \(\frac{\partial u}{\partial t} = 0\) almost everywhere in \(\Sset^1 \times [0, \Bar{T}]\) and that \(\gamma = \beta\) is a minimising geodesic.

If \((T_n)_{n \in \Nset}\) diverges to \(+\infty\), there exists a map \(u \in W^{1, 2} (\Sset^1 \times (0, +\infty), \manifold{N})\) such that \(\tr_{\Sset^1 \times \{0\}} u = \gamma\) and for every \(T \in (0, +\infty)\),
\[
\int_{\Sset^1 \times [0, T]} \frac{\abs{\Deriv u}^2}{2} - \frac{\equivnorm{\gamma}^2}{8 \pi^2} = 0 ,
\]
from which it follows that \(\gamma\) is a minimising geodesic.
Similarly, \(\beta\) is also a minimising geodesic.
\end{proof}

The next proposition states that minimising geodesics that are homotopic through minimising geodesics are synharmonic.

\begin{proposition}
\label{proposition_homotopy_synharmonic}
Let \(H \in \mathcal{C}^1 (\Sset^1 \times [0, T], \manifold{N})\).
If for every \(t \in [0, 1]\), the map
\(H (\cdot, t) : \Sset^1 \to \manifold{N}\) is a minimising geodesic, then
\[
 \synhar{H (\cdot, 0)}{H (\cdot, 1)} = 0.
\]
\end{proposition}
\begin{proof}
  \resetconstant
We define for every \(T >0\) the function \(u_T : \Sset^1 \times [0, 1] \to \manifold{N}\)
by
\(
 u_T (x, t) \defeq H (x, \tfrac{t}{T}).
\)
and we observe that

\[
  \int_{\Sset^1 \times [0, T]} 
  \biggl(\frac{\abs{\Deriv u_T}^2}{2} -  \frac{\equivnorm{H(\cdot,0)}^2}{8 \pi^2}\biggr)
  = \frac{1}{T} \int_{\Sset^1 \times [0, 1]} \left|\frac{ \partial H}{\partial t}\right|^2,
\]
which goes to \(0\) as \(T \to +\infty\).
\end{proof}
In particular, if \(R (\theta) \in SO (2)\) denotes the rotation of angle \(\theta \in \Rset\) we deduce that
\(
 \synhar{\gamma \compose R (\theta)}{\gamma} = 0
\),
by applying \cref{proposition_homotopy_synharmonic} with \(H (x, t) = \gamma (R (t \theta))\).

The following proposition provides us with an example of non-synharmonic geodesics on a Riemannian manifold.

\begin{proposition}
\label{proposition_counter_synharm}
Assume that \(\manifold{N} = (\Sset^1 \times \Sset^1, g)\),  \(I_+\) and \(I_-\) are connected nonempty disjoint open sets of  \(\Sset^1\) such that \(\partial I_+ = \partial I_- = \{a_0, a_1\}\), where the metric \(g\) satisfies the following properties:
\begin{enumerate}[(a)]
 \item \label{item_counter_synharm_pinching} if \(y \in  \Sset^1 \times \{a_0, a_1\} \) and \(v=(v_1,0) \in \Rset \times \{0\} \simeq T_{y_1} \Sset^1 \times \{0\}\),
 \[
   g_y (v) = \abs{v_1}^2,
 \]
 \item \label{item_counter_synharm_lower}  if \(y \in  \Sset^1 \times \Sset^1\) and \(v=(v_1,v_2) \in \Rset^2 \simeq T_{y} (\Sset^1 \times \Sset^1)\),
 \[
   g_y (v) \ge \abs{v_1}^2 + \alpha (y) g_0 (v)^2,
 \]
 where \(\alpha \in \mathcal{C}(\Sset^1 \times \Sset^1, [0, + \infty))\)  and \(g_0\) is the Euclidian metric in \(\Rset^2\).
\end{enumerate}
Then
\[
 \synhar{(\id_{\Sset^1}, a_0)}{(\id_{\Sset^1}, a_1)}
 \ge \min \biggl\{\int_{I_+ \times \Sset^1} \alpha
 , \int_{I_- \times \Sset^1} \alpha \biggr\}.
\]
\end{proposition}

The assumption \eqref{item_counter_synharm_pinching} and \eqref{item_counter_synharm_lower}, ensures that the homotopic maps \((\id_{\Sset^1}, a_0)\) and \((\id_{\Sset^1}, a_1)\) are minimising geodesics and that
\( \equivnorm{(\id_{\Sset^1}, a_0)} = \equivnorm{(\id_{\Sset^1}, a_1)} = 2 \pi\).

By choosing an appropriate metric \(g\), the right-hand side in the conclusion of \cref{proposition_counter_synharm} can be made positive, giving an example of homotopic minimising geodesics that are not synharmonic.

\begin{proof}[Proof of \cref{proposition_counter_synharm}]
Let \(u \in \mathcal{C}^\infty (\Sset^1 \times [0, T], \Sset^1 \times \Sset^1)\) with \(u (\cdot, 0) = (\id_{\Sset^1}, a_0)\) and \(u (\cdot, T) = (\id_{\Sset^1}, a_1)\), then
\[
\int_{\Sset^1 \times [0, T]} \frac{\abs{\Deriv u}_g^2}{2}
\ge \int_{\Sset^1 \times [0, T]} \frac{\abs{\Deriv u_1}^2}{2}
+ \int_{\Sset^1 \times [0, T]} \alpha (u) \frac{\abs{\Deriv u}_{g_0}^2}{2},
\]
where \(u_1 \in \mathcal{C}^\infty (\Sset^1 \times [0, T], \Sset^1)\) is the first component of the map \(u\).
We first have
\[\int_{\Sset^1 \times [0, T]}\frac{ \abs{\Deriv u_1}^2}{2}
\ge
 T \frac{\equivnorm{( \id_{\Sset^1},a_0)}^2}{4 \pi}.
\]
Next, since \(\abs{\Deriv u}^2_{g_0} \ge 2 \abs{\det \Deriv u}\), we have by the area formula
\[
\begin{split}
  \int_{\Sset^1 \times [0, T]} \alpha (u)\frac{ \abs{\Deriv u}_{g_0}^2}{2}
 \ge \int_{\Sset^1 \times [0, T]} \alpha (u) \abs{\det \Deriv u} \ge \int_{\Sset^1 \times \Sset^1} \alpha (y) \mathcal{H}^0 (u^{-1} (\{y\})) \dif y.
\end{split}
\]
We observe now that by a topological argument either \(u (\Sset^1 \times [0, T]) \supseteq I_+ \times \Sset^1\) or \(u (\Sset^1 \times [0, T]) \supseteq I_- \times \Sset^1\).
Indeed, otherwise there  exist points \(b_+ \in I_+ \times \Sset^1 \setminus
u (\Sset^1 \times [0, T]) \) and \(b_- \in I_- \times \Sset^1 \setminus u (\Sset^1 \times [0, T])\).
But, there exists a continuous map \(\rho : \Sset^1 \times \Sset^1 \setminus \{b_+, b_-\} \to \Sset^1 \times \{a_0\}\simeq\Sset^1\) such that \(\rho\vert_{\Sset^1 \times \{a_0\}} = \id\) and \(\rho \vert_{\Sset^1 \times \{a_1\}}\) is constant, and thus \(\rho \compose u : \Sset^1 \times [0, T] \to \Sset^1\) is a homotopy between the identity and a constant map, which is a contradiction.

It follows thus that
\[
\int_{\Sset^1 \times [0, T]} \frac{\abs{\Deriv u}_g^2}{2}
 \ge T \frac{\equivnorm{(\id_{\Sset^1}, a_0)}^2}{4 \pi}
 + \min \biggl(\int_{I_+ \times \Sset^1} \alpha
, \int_{I_- \times \Sset^1} \alpha \biggr).
\]
The result follows then by a standard approximation argument.
\end{proof}

\subsection{Dependence on the charges of the geometrical renormalised energy}
The dependence of the geometric renormalised energy on the charges (i.e. the maps \(\gamma_i\)) is controlled by the synharmony.

\begin{proposition}
\label{proposition_renorm_geom_dep_gamma}
Let \(\Omega\subset\Rset^2\) be a Lipschitz bounded domain, \(g \in W^{1/2,2} (\partial \Omega, \manifold{N})\), \(k\in\Nset_\ast\) and \((a_1, \dotsc, a_k) \in\Conf{k} \Omega\). If \((\gamma_1, \dotsc, \gamma_k)\) and \((\tilde{\gamma}_1,\dotsc,\tilde{\gamma}_k)\in \mathcal{C}^1(\Sset^1,\manifold{N})^k\) are topological resolutions of \(g\) made up of minimising geodesics, then
\[
\bigabs{\mathcal{E}_{g,\Tilde{\gamma}_1, \dotsc, \Tilde{\gamma}_k}^{\mathrm{geom}}  (a_1, \dotsc, a_k)
- \mathcal{E}_{g,\gamma_1, \dotsc, \gamma_k}^{\mathrm{geom}}  (a_1, \dotsc, a_k)}
\le \sum_{i = 1}^k \synhar{\gamma_i}{\Tilde{\gamma}_i}.
\]
\end{proposition}
In particular, if for every \(i \in \{1, \dotsc, k\}\), \(\gamma_i\) and \(\Tilde{\gamma_i}\) are synharmonic minimising geodesics, then
\[
 \mathcal{E}_{g,\Tilde{\gamma}_1, \dotsc, \Tilde{\gamma}_k}^{\mathrm{geom}}  (a_1, \dotsc, a_k)
= \mathcal{E}_{g,\gamma_1, \dotsc, \gamma_k}^{\mathrm{geom}}  (a_1, \dotsc, a_k).
\]
The dependence of the synharmonicity is optimal, as can be seen by observing that when \(\Omega = \mathbb{D}\) is the unit disk and \(\gamma\), \(\Tilde{\gamma}\) are homotopic minimal geodesics, then
\begin{equation*}
  \mathcal{E}_{\gamma, \gamma}^{\mathrm{geom}}(0) = 0
  \quad\text{and}\quad
  \mathcal{E}_{\gamma, \Tilde{\gamma}}^{\mathrm{geom}} (0) = \synhar{\gamma}{\Tilde{\gamma}},
\end{equation*}
and according to \cref{proposition_counter_synharm}, the latter quantity can be positive for homotopic minimising geodesics.

\begin{proof}
[Proof of \cref{proposition_renorm_geom_dep_gamma}]
We can assume without loss of generality that \(\synhar{\gamma_i}{\Tilde{\gamma}_i}<+\infty\) for each \(i\). We take \(\sigma\in (0,\Bar{\rho}(a_1,\dots,a_k))\) and \(L > 0\).
Given a mapping \(u \in W^{1, 2} (\Omega \setminus \bigcup_{i = 1}^k \Bar{B}_{\sigma} (a_i), \manifold{N})\) such that \(\tr_{\Sset^1} u (a_i + \sigma\,\cdot)  = \gamma_i\), 
and \(u_i \in W^{1, 2} (\Sset^1 \times [0, L], \manifold{N})\) such that \(\tr_{\Sset^1 \times \{0\}} u_i  (\cdot, 0) = \gamma_i\) and \(\tr_{\Sset^1 \times \{L\}} u_i = \Tilde{\gamma}_i\), we set \(\rho = e^{-L} \sigma\) and we define \(v \in W^{1, 2} (\Omega \setminus \bigcup_{i = 1}^k \Bar{B}_{\rho} (a_i), \manifold{N})\) for each \(x \in \Omega \setminus \bigcup_{i = 1}^k \Bar{B}_{\rho} (a_i)\) by
\[
v (x)
\defeq 
\begin{cases}
 u (x) & \text{if \( x \in \Omega \setminus \bigcup_{i = 1}^k \Bar{B}_{\sigma} (a_i)\)},\\
 u_i \Bigl(\frac{x  - a_i}{\abs{x - a_i}},  \log \frac{\sigma}{\abs{x - a_i}}\Bigr)
 & \text{if \(x \in \Bar{B}_{\sigma} (a_i) \setminus \Bar{B}_{\rho} (a_i)\) for some \(i \in \{1, \dotsc, k\}\)}.
\end{cases}
\]
We have then 
\[
\begin{split}
     \int_{\Omega \setminus \bigcup_{i = 1}^k \Bar{B}_{\rho} (a_i)}
       \frac{\abs{\Deriv v}^2}{2}
   -
     \sum_{i =1}^k
       \frac{\equivnorm{\gamma_i}^2}{4 \pi} \log \frac{1}{\rho}
 =
   & \int_{\Omega \setminus \bigcup_{i = 1}^k \Bar{B}_{\sigma} (a_i)}
       \frac{\abs{\Deriv u}^2}{2}
   -
     \sum_{i =1}^k  \frac{\equivnorm{\gamma_i}^2}{4 \pi} \log \frac{1}{\sigma}\\
 &+ \sum_{i = 1}^k \int_{\Sset^1 \times [0, L]}\biggl(
 \frac{\abs{\Deriv u_i}^2}{2} - \frac{\equivnorm{\gamma_i}^2}{8 \pi^2}\biggr).
\end{split}
\]
It follows by \eqref{eq_def_renorm_geom} that
\[
 \mathcal{E}_{g,\Tilde{\gamma}_1, \dotsc, \Tilde{\gamma}_k}^{\mathrm{geom}}  (a_1, \dotsc, a_k)
 \le \mathcal{E}_{g, \gamma_1,\dotsc, \gamma_k}^{\mathrm{geom}}  (a_1, \dotsc, a_k)
+ \sum_{i = 1}^k \synhar{\gamma_i}{\Tilde{\gamma}_i},
\]
since the infimum in \cref{definition_synharmonic} stays the same under a lower bound on \(L\) when either \(\beta\) or \(\gamma\) is a geodesic.
\end{proof}

\section{Dependence on the location of singularities}

The next proposition states that the topological and geometrical renormalised energies are Lipschitz continuous functions of the locations of the singularities.

\begin{proposition}
\label{theorem_continuity_renormalised}
Let \(\Omega\subset\Rset^2\) be a bounded Lipschitz domain, \(g \in W^{1/2}(\partial\Omega,\manifold{N})\), \(k\in\Nset_\ast\) and \((\gamma_1, \dotsc, \gamma_k) \in W^{1/2,2}(\Sset^1, \manifold{N})^k\) be a topological resolution of \(g\).
The renormalised energy \(\mathcal{E}^{\mathrm{top}}_{g, \gamma_1, \dotsc, \gamma_k}\)  is locally Lipschitz-continuous on \(\Conf{k} \Omega\).
If moreover \(\gamma_1,\dotsc,\gamma_k\) are minimising geodesics, 
then \(\mathcal{E}^{\mathrm{geom}}_{g,\gamma_1\dotsc,\gamma_k}\) is locally Lipschitz-continuous on \(\Conf{k} \Omega\).
\end{proposition}

Our proof follows the strategy that was outlined for the Lipschitz-continuity of renormalised energies for \(n\)--harmonic maps into \(\Sset^{n - 1}\) on \((n-1)\)--dimensional domains \cite{Hardt_Lin_Wang_1997}*{\S 9}.

In order to compare the renormalised energies at \(k\)-tuples \((a_1,\dotsc,a_k),(b_1,\dotsc,b_k)\in \Conf{k} \Omega\), we use a deformation of \(\Omega \setminus \bigcup_{i = 1}^k \Bar{B}_{\rho} (b_i)\) onto \(\Omega \setminus \bigcup_{i = 1}^k \Bar{B}_{\rho} (a_i)\) given by the following lemma:
\begin{lemma}
\label{lemma_deformation_renormalised}
Let \(k \in \Nset_\ast\) and let \(a_1,\dotsc,a_k, b_1, \dotsc, b_k \in \Omega\).
If \(2 \max_{1 \le i \le k} \abs{a_i-b_i}<\tau \le \Bar{\rho} (b_1, \dotsc, b_k)/3\) then there exists a diffeomorphism \(\Phi:\Omega\to\Omega\) such that
\begin{enumerate}[(i)]
\item\label{phi_identity}
for every \(x\in \Omega \setminus \bigcup_{i = 1}^k B_{2\tau} (b_i)\),
\(\Phi(x)=x\),
\item\label{phi_null}
for each \(i \in \{1, \dotsc, k\}\) and every \(x \in B_\tau (b_i)\), \(\Phi (x) = x + a_i - b_i\),
\item\label{phi_estimate}
for every \(x \in \Omega\),
\(\abs{\Deriv \Phi(x) -\mathrm{id}}\le \frac{2}{\tau}\max_{1 \le i \le k}\{\abs{a_i-b_i}\}\).
\end{enumerate}
\end{lemma}
\begin{proof}[Proof of \cref{lemma_deformation_renormalised}]
Let \(\varphi\in \mathcal{C}^\infty(\Rset^2,[0,1])\) be a smooth function such that \(
\abs{\Deriv \varphi} < 2\) in \(\Rset^2\),
\(\varphi=1\) in \( B_1(0)\)
and \(\varphi=0\) in \(\Rset^2\setminus B_2(0)\).
We define the function \(\Phi\in\mathcal{C}^\infty(\Omega,\Rset^2)\) by setting for each \(x\in\Omega\),
\begin{equation}
\label{def_varphi}
\Phi (x) \defeq x + \sum_{i=1}^k \varphi \biggl(\frac{x - b_i}{\tau}\biggr) (a _i - b_i),
\end{equation}
so that \eqref{phi_identity} and \eqref{phi_null} are satisfied.
We also get that \(\Phi(\Omega)\subset\Omega\): if \(x \in B_{2 \tau} (b_i)\), then \(\Phi (x) \in B_{5 \tau/2}(b_i) \subset B_{5\Bar{\rho}/{6}}(b_i)\subset \Omega\). Moreover, for every \(x \in \Omega\),
\begin{equation}\label{eq:injectivity_differentiability}
|D\Phi(x)-\mathrm{id}|
\le \max_{1 \le i \le k} \Bigl\{\tfrac{1}{\tau}\bigabs{\Deriv \varphi \bigl(\tfrac{x-b_i}{\tau}\bigr)} \,\abs{a_i-b_i}\Bigr\} \le \frac{2}{\tau}\max_{1 \le i \le k} \{\abs{a_i-b_i}\} <1,
\end{equation}
so that \eqref{phi_estimate} holds,
where the norm of the left-hand side denotes the norm of linear operators. The function \(\Phi\) is onto: it is onto in \(\Omega \setminus \cup_{i=1}^k \Bar{B}_{2\tau}(b_i)\); and, since \(\Phi_{\vert \partial B_{2\tau}(b_i)}= \mathrm{id}_{\vert\partial B_{2\tau}(b_i)}\) for all \(i\in\{ 1,\dots,k\}\), \(\Phi\) is also onto in \(\cup_{i=1}^k \Bar{B}_{2\tau}(b_i)\) thanks to a degree argument. The function \(\Phi\) is also into in \(\Omega\). Indeed, if \(x,y \in \Omega\) are such that \(\Phi(x)=\Phi(y)\) and \(x \neq y\), then we find, thanks to \eqref{eq:injectivity_differentiability}, that \(|\Phi(x)-x-\Phi(y)+y|=|x-y|<|x-y|\). This is a contradiction. At last, \(\Phi\) is a diffeomorphism thanks to \eqref{eq:injectivity_differentiability} which proves that \(\Phi^{-1}\) is differentiable.
\end{proof}

\begin{proof}[Proof of \cref{theorem_continuity_renormalised}]
Let \(\tau>0\) and \(a_1,\dotsc,a_k, b_1,\dotsc,b_k\in\Omega\)
satisfying \(2\max_{1 \le i \le k} \abs{a_i-b_i} <  \tau \le \Bar{\rho} (b_1, \dotsc, b_k)/3\).
Let \(\Phi\) be the diffeomorphism given by \cref{lemma_deformation_renormalised} for these points.
For every \(\rho\in (0,\tau)\) and \(u\in W^{1,2}(\Omega \setminus \bigcup_{i = 1}^k \Bar{B}_{\rho} (a_i),\manifold{N})\), we estimate the energy of the map \(v\defeq u\compose\Phi\in W^{1,2}(\Omega \setminus \bigcup_{i = 1}^k \Bar{B}_{\rho} (b_i),\manifold{N})\).

From \eqref{phi_identity}, \eqref{phi_null} in \cref{lemma_deformation_renormalised} and from the change of variable \(x=\Phi (y)\) we have
\begin{equation}
\label{energy_uv}
\begin{split}
\int_{\Omega \setminus \bigcup_{i = 1}^k \Bar{B}_{\rho} (b_i)}\abs{\Deriv v(y)}^2\dif y&-\int_{\Omega \setminus \bigcup_{i = 1}^k \Bar{B}_{\rho} (a_i)}\abs{\Deriv u(x)}^2\dif x \\
= &\int_{\bigcup_{i = 1}^k \left( B_{2\tau}(a_i)  \setminus \Bar{B}_{\tau} (a_i)\right)}
\Big(\frac{\abs{\Deriv u(x)D\Phi(\Phi^{-1}(x) )}^2}{\det D\Phi(\Phi^{-1}(x))}-\abs{\Deriv u(x)}^2\Big)\dif x.
\end{split}
\end{equation}
We observe that, if we let \((M{\,:\,}N)=\tr(MN^\ast)\) be the inner poduct of matrices and \(M=\Deriv u(x)\in\Rset^{\nu\times 2}, A=D\Phi(\Phi^{-1}(x))\in GL_2(\Rset)\), we have
\[
\frac{\abs{MA}^2}{\det A}-\abs{M}^2=\left(M{\,:\,}M\Big(\frac{AA^\ast}{\det A}-\mathrm{id}\Big)\right)
\]
and that, by smoothness of the map \( A\in GL_2(\Rset)\mapsto AA^*/(\det A)\), there exists \(\eta > 0\) such that if 
\(\abs{A - \operatorname{id}} \le \eta\)
then 
\[
 \biggabs{ \frac{A A^*}{\det A} - \operatorname{id}}
 \le \C \abs{A - \operatorname{id}}.
\]

It follows then from \eqref{energy_uv} and \eqref{phi_estimate} in \cref{lemma_deformation_renormalised} that if \(\max_{1 \le i \le k} \abs{b_i - a_i} \le \tau \eta/2\), then
\begin{equation}
\label{eq_oquei2cheemahc8Eh7ahThe7}
\int_{\Omega \setminus \bigcup_{i = 1}^k \Bar{B}_{\rho} (b_i)}\frac{\abs{\Deriv v}^2}{2}
\le
\int_{\Omega \setminus \bigcup_{i = 1}^k \Bar{B}_{\rho} (a_i)}\frac{\abs{\Deriv u}^2}{2} + 
\Cl{cst_Aagoojoh7shaeb6sheaD7eiL} \max_{1 \le i \le k} \abs{b_i - a_i}
\int_{\Omega \setminus \bigcup_{i = 1}^k \Bar{B}_\tau (a_i)} \frac{\abs{\Deriv u}^2}{2}.
\end{equation}
If we now assume that \(\tr_{\Sset^1} u (b_i + \rho\,\cdot) = \gamma_i\),
so that \(\tr_{\Sset^1} v (a_i + \rho\,\cdot) = \gamma_i\), we have 
by \eqref{eq_oquei2cheemahc8Eh7ahThe7} and \cref{lemma_energy_growth_map_radii} that 
\begin{multline}
\label{eq_yeeduBie4Ahseifaeph3uedu}
\int_{\Omega \setminus \bigcup_{i = 1}^k \Bar{B}_{\rho} (b_i)}\frac{\abs{\Deriv v}^2}{2}  \\
\le
\int_{\Omega \setminus \bigcup_{i = 1}^k \Bar{B}_{\rho} (a_i)}\frac{\abs{\Deriv u}^2}{2}
+ \Cr{cst_Aagoojoh7shaeb6sheaD7eiL} \max_{1 \le i \le k} \abs{b_i - a_i}
\biggl(\int_{\Omega \setminus \bigcup_{i = 1}^k \Bar{B}_{\rho} (a_i)}\frac{\abs{\Deriv u}^2}{2}
 - \sum_{i = 1}^k \frac{\equivnorm{\gamma_i}^2}{4 \pi} \log \frac{\tau}{\rho}\biggr).
\end{multline}
In view of \eqref{eq_def_renorm_geom_rho}, \eqref{eq_yeeduBie4Ahseifaeph3uedu} implies that 
\begin{multline*}
 \mathcal{E}_{g, \gamma_1, \dotsc, \gamma_k}^{\mathrm{geom}, \rho} (b_1, \dotsc, b_k)\\
 \le \mathcal{E}_{g, \gamma_1, \dotsc, \gamma_k}^{\mathrm{geom}, \rho} (a_1, \dotsc, a_k)
 + \Cr{cst_Aagoojoh7shaeb6sheaD7eiL} \max_{1 \le i \le k} \abs{b_i - a_i}
 \biggl( \mathcal{E}_{g, \gamma_1, \dotsc, \gamma_k}^{\mathrm{geom}, \rho} (a_1, \dotsc, a_k) - \sum_{i = 1}^k \frac{\equivnorm{\gamma_i}^2}{4 \pi} \log \frac{\tau}{\rho}\biggr),
\end{multline*}
so that by \eqref{eq_def_renorm_geom}, we have 
\begin{multline}
\label{eq_ieKah0ei5uloojui8ieVug7o}
 \mathcal{E}_{g, \gamma_1, \dotsc, \gamma_k}^{\mathrm{geom}} (b_1, \dotsc, b_k)\\
 \le \mathcal{E}_{g, \gamma_1, \dotsc, \gamma_k}^{\mathrm{geom}} (a_1, \dotsc, a_k)
 + \Cr{cst_Aagoojoh7shaeb6sheaD7eiL} \max_{1 \le i \le k} \abs{b_i - a_i}
 \biggl( \mathcal{E}_{g, \gamma_1, \dotsc, \gamma_k}^{\mathrm{geom}} (a_1, \dotsc, a_k) + \sum_{i = 1}^k \frac{\equivnorm{\gamma_i}^2}{4 \pi} \log \frac{1}{\tau} \biggr).
\end{multline}
It follows then from \eqref{eq_ieKah0ei5uloojui8ieVug7o} first that \(\mathcal{E}_{g, \gamma_1, \dotsc, \gamma_k}^{\mathrm{geom}}\) is locally bounded and next that \(\mathcal{E}_{g, \gamma_1, \dotsc, \gamma_k}^{\mathrm{geom}}\) is locally Lipschitz continuous.

The proof for the geometrical renormalised energy \(\mathcal{E}^{\mathrm{top}}_{g, \gamma_1, \dotsc, \gamma_k}\) 
is similar, relying on \eqref{eq_def_renorm_top_rho} and \eqref{eq_def_renorm_top} instead of \eqref{eq_def_renorm_geom_rho} and \eqref{eq_def_renorm_geom}.
\end{proof}

\section{Lower bounds}

\subsection{Lower bound on Dirichlet energy of maps}
We first recall the definition of the Hausdorff content in the particular case of compact planar sets.

\begin{definition}
\label{def_hausdorff_content}
The \emph{one-dimensional Hausdorff content} of a compact set \(A \subset \Rset^2\) is defined as
\[
 \mathcal{H}^1_\infty (A)
 \defeq
 \inf \biggl\{ \sum_{B \in \mathcal{B}} \diam (B) \st \ A \subset \bigcup_{B \in \mathcal{B}} B \text{ and } \mathcal{B} \text{ is a finite collection of closed balls}\biggr\}.
\]
\end{definition}

The following gives a lower bound on the Dirichlet energy outside a small compact set \(K\).

\begin{theorem}
\label{prop_lower_bound_compact}
For every Lipschitz bounded domain \(\Omega\subset \Rset^2\), every compact set \(K\subset \Omega\) with \(\mathcal{H}^1_{\infty} (K)>0\), and every map \(u \in W^{1, 2} (\Omega \setminus K, \manifold{N})\), we have
\begin{equation}\label{first_prop_lower_bound_compact}
\int_{\Omega \setminus K} \frac{|\Deriv u|^2}{2}\geq \Esing (\tr_{\partial \Omega}u) \log \frac{\dist(K,\partial \Omega)}{2\, \mathcal{H}^1_{\infty} (K)}.
\end{equation}
More precisely, there exists a constant \(C>0\), that can be taken equal to \(10+\frac{4}{\log 2}\), such that
\begin{equation}\label{second_prop_lower_bound_compact}
\abs{\Deriv u}^2_{L^{2,\infty}(\Omega\setminus K,\Rset^{\nu\times 2})}\le
C
\biggl(\int_{\Omega \setminus K} \frac{\abs{\Deriv u}^2}{2}
 -  \Esing (\tr_{\partial \Omega}u) \log \frac{\dist(K,\partial \Omega)}{2\, \mathcal{H}^1_{\infty} (K)} \biggr),
 \end{equation}
where \(\abs{\Deriv u}^2_{L^{2,\infty}(\Omega\setminus K,\Rset^{\nu\times 2})}\defeq \sup_{s>0}s^2 \, \mathcal{L}^2 (\{x \in \Omega \setminus K \st \abs{\Deriv u} \ge s\})\).
\end{theorem}

When \(\manifold{N} = \Sset^1\), \cref{prop_lower_bound_compact} has its roots in a corresponding estimate for maps outside a finite collection of balls \cite{Bethuel_Brezis_Helein_1994}*{Corollary II.1}. In this case, the first lower bound \eqref{first_prop_lower_bound_compact} is due to Sandier \cite{Sandier_1998}*{Theorem 1} (see also \cite{Jerrard_1999}*{Theorem 1.1}). The left-hand side of \eqref{second_prop_lower_bound_compact} is a weak \(L^2\) estimate on \(\abs{\Deriv u}\),
or equivalently an estimate in the Marcinkiewicz space or endpoint Lorentz space \(L^{2, \infty} (\Omega \setminus K)\). The weak \(L^2\) estimate is reminiscent of the corresponding estimate for Ginzburg--Landau equation in \cite{Serfaty_Tice_2008}.

We will use decomposition of the plane into balls, in what is known as a ball growth  construction \cite{Sandier_Serfaty_2007}*{Theorem 4.2} (see also \citelist{\cite{Sandier_1998}\cite{Jerrard_1999}}).

\begin{proposition}
\label{propositionMultipolar}
If \(\mathcal{B}_0\) is a finite collection of closed balls in \(\Rset^2\) such that \(\bigcup_{B \in \mathcal{B}_0} B\neq\emptyset\) then, for every \(t \in [0,+\infty)\), there exists a finite collection \(\mathcal{B}(t)\) of disjoint non-empty closed balls in \(\Rset^2\) such that
\begin{enumerate}[(i)]
\item \label{it_Sie3i_0}
\(
\mathcal{B}_0 = \bigcup_{B' \in \mathcal{B}(0)} \{B \in \mathcal{B}_0 \st B \subseteq B'\},
\)
\item \label{it_Sie3i_1} 
for every \(s \geq t\geq 0\), 
\[
 \mathcal{B} (t) = \bigcup_{B' \in \mathcal{B} (s)} \{B \in \mathcal{B} (t) \st B \subseteq B'\},
\]
\item \label{it_Sie3i_4} there exist  \(\ell \in \Nset_*\) and some disjoint intervals \(I_1, \dotsc, I_\ell \subset \Rset\) such that \([0,+\infty)=\bigcup_{j=1}^\ell I_j\) and such that for each \(j\in\{1,\dotsc,\ell\}\) and \(t \in I_j\),
\[
  \mathcal{B}(t)= \{\Bar B_{\rho_j^ie^t}(a_j^i) \st i\in \{1,\dotsc,k_j\}\},
\]
with  \(k_j \in \Nset_*\) and for \(i\in \{1,\dotsc,k_j\}\), \(a_j^i \in \Rset^2\), \(\rho_j^i\in \Rset^+\),
\item \label{it_Sie3i_3} for every \(t \geq 0\)
\[ 
  \sum_{B \in \mathcal{B}(t)} \diam (B) =e^{t} \sum_{B \in \mathcal{B}_0} \diam (B).
\]
\end{enumerate}
\end{proposition}
\begin{remark}
\label{treeStructure}
By disjointness of the balls in \(\mathcal{B}(t)\) and by \eqref{it_Sie3i_1}, for every \(B,B'\in\bigcup_{t\in[0,+\infty)}\mathcal{B}(t)\), we have either \(B\cap B'=\emptyset\) or \(B\subset B'\) or \(B'\subset B\).
\end{remark}
The statement of \cref{propositionMultipolar} is essentially the same as in Sandier and Serfaty's book \cite{Sandier_Serfaty_2007}*{theorem 4.2}. A crucial tool we will repeatedly use is the following merging of balls procedure (see \cite{Sandier_Serfaty_2007}*{lemma 4.1})
\begin{lemma}
\label{lemmaMerging}
For every finite set \(\mathcal{B}\) of closed  balls of \(\Rset^2\),
there exists a finite set \(\mathcal{B}'\) of disjoint non-empty closed balls of \(\Rset^2\) such that
\[
 \mathcal{B} = \bigcup_{B' \in \mathcal{B}'} \{B \in \mathcal{B}\st B \subseteq B'\},
\]
and 
\[
  \sum_{B' \in \mathcal{B}'}\diam (B') =\sum_{B \in \mathcal{B}} \diam (B). 
\]
\end{lemma}
\begin{proof}
It is sufficient to prove the lemma for two balls. If \(B_1=\Bar{B}_{r_1}(a_1)\) and \(B_2=\Bar{B}_{r_2}(a_2)\) are not disjoint then we observe that
\[ 
B\defeq \Bar{B}_{r_1+r_2}\left(\frac{r_1a_1+r_2a_2}{r_1+r_2}\right)
\]
contains both \(B_1\) and \(B_2\).
\end{proof}

\begin{proof}%
[Proof of \cref{propositionMultipolar}]
We first apply \cref{lemmaMerging} to cover the collection of disks \(\mathcal{B}_0\) with a collection of \(k_1\in \Nset_\ast\) disjoint non-empty closed balls \(\mathcal{B}(0)=\{\Bar B_{\rho^i_1(a^i_1)}\st  i=1,\dots, k_1\}\).

If \(k_1=1\), we take \(\mathcal{B}(t)\defeq \{\Bar{B}_{\rho_1^1e^t}(a_1^1)\}\) for all \(t\in[0,+\infty)\), \(\ell \defeq 1\) and \(I_1 \defeq [0, +\infty)\); we observe that \(\sum_{B \in \mathcal{B} (t)} \diam (B) = 2\rho_1^1 e^t\).

If \(k_1\ge 2\), we define
\begin{multline*}
t_1
 \defeq \sup \Big\{t \in [0, + \infty) \st \text{for every } i, i' \in \{1, \dotsc, k_1\} \text{ with } i \ne i', \\
 \text {one has }
  \Bar{B}_{\rho_1^i e^t} (a_1^i) \cap
  \Bar{B}_{\rho_{1}^{i'} e^t} (a_{1}^{i'}) = \emptyset \Big\}>0.
\end{multline*}
We set \(I_1\defeq [0,t_1)\) and for every \(t\in [0,t_1)\), \(\mathcal{B}(t)\defeq (\Bar{B}_{\rho_1^ie^t}(a_1^i))_{i=1,\dotsc,k_1}\).
We then apply \cref{lemmaMerging} to cover the family of balls \((\Bar{B}_{\rho_1^ie^{t_1}}(a_1^i))_{i=1,\dotsc,k_1}\) with a family \(\mathcal{B}(t_1)\defeq \{\Bar{B}_{\rho_2^i}(a_2^i)\}\) consisting of \(k_2\in\{1,\dots,k_1-1\}\) disjoint non-empty closed balls.
We then set 
\begin{multline*} t_2
 \defeq \sup \Big\{t \in [t_1, + \infty) \st \text{for every } i, i' \in \{1, \dotsc, k_2\} \text{ with } i \ne i', \\
 \text {one has }
  \Bar{B}_{\rho_2^i e^t} (a_2^i) \cap
  \Bar{B}_{\rho_{2}^{i'} e^t} (a_{2}^{i'}) = \emptyset \Big\},
\end{multline*}
\(I_2 \defeq [t_1,t_2)\) and \(\mathcal{B}(t)\defeq (\Bar{B}_{\rho_2^ie^t}(a_2^i))_{i=1,\dotsc,k_2}\) for every \(t\in I_2\).
Repeating the merging construction of \cref{lemmaMerging} and the previous process, we obtain the conclusion of \cref{propositionMultipolar}. Indeed, the merging process can occur only a finite number of times \(\ell\in\Nset_\ast\) since it strictly decreases the number of balls. Besides, by construction we obtain immediately \eqref{it_Sie3i_1} and \eqref{it_Sie3i_4}.

The exponential growth of \( \sum_{B \in \mathcal{B}(t)} \diam (B)\), i.e., \eqref{it_Sie3i_3} in \cref{propositionMultipolar}, is a consequence of the exponential growth of the radii of the balls in \(\mathcal{B}(t)\) on each interval \(I_j\), and of the conservation of the quantity \( \sum_{B \in \mathcal{B}(t)} \diam (B)\) during the merging steps by \cref{lemmaMerging} which ensures that the map \(t\mapsto \sum_{B \in \mathcal{B}(t)} \diam (B)\) is continuous at the merging times \(t_j\).
\end{proof}

\begin{proof}[Proof of \cref{prop_lower_bound_compact}]
We first observe that when \(2\,\mathcal{H}^1_\infty(K)\ge \dist(K,\partial\Omega)\), then 
\(
\log\frac{\dist(K,\partial \Omega)}{2\, \mathcal{H}^1_\infty(K)}\leq 0
\)
so that \eqref{second_prop_lower_bound_compact} follows from Chebyshev's inequality with \(C=2\). 

We can thus assume that \(2\mathcal{H}^1_\infty(K)< \dist(K,\partial\Omega)\). We fix \(\varepsilon\) such that \(0<\varepsilon<\frac{1}{2}\dist(K,\partial\Omega)-\mathcal{H}^1_\infty(K)\). By definition of the Hausdorff content of a compact set (\cref{def_hausdorff_content}), there exists a finite family of closed balls \(\mathcal{B}_0\) such that 
\begin{enumerate}[(a)]
\item\label{hausdorff_cover} \(K\subset \bigcup_{B\in \mathcal{B}_0}B\),
\item\label{hausdorff_intersection} \(B \cap K \ne \emptyset\) for each \(B\in\mathcal{B}_0\), 
\item\label{hausdorff_fine} \(\sum_{B\in\mathcal{B}_0}\operatorname{diam}(B)<\mathcal{H}^1_\infty(K)+\varepsilon\ <\frac{1}{2} \dist(K,\partial\Omega)\).
\end{enumerate}
We then let \(\{\mathcal{B}(t) \}_{t\in[0,+\infty)}\) be the family of growing balls given by \cref{propositionMultipolar}
 and we set 
\[
  A(t)\defeq\bigcup_{B \in \mathcal{B} (t)} B. 
\]
In particular, by the property \eqref{hausdorff_fine} of the cover \(\mathcal{B}_0\), and by \eqref{it_Sie3i_3} in \cref{propositionMultipolar},
\begin{equation}\label{hausdorff_confinement}
\sum_{B \in \mathcal{B}(0)} \diam (B) = \sum_{B \in \mathcal{B}_0} \diam (B)  <\varepsilon+\mathcal{H}^1_\infty(K)<\frac{1}{2} \dist(K,\partial\Omega).
\end{equation}
For every \(t \in [0, +\infty)\) and \(B \in \mathcal{B} (t)\), by \eqref{it_Sie3i_3} in \cref{propositionMultipolar},
\[
\operatorname{diam} (B) \le \sum_{B \in \mathcal{B} (t)}\diam (B)  = e^t \sum_{B \in \mathcal{B}(0)} \diam (B) .
\]
Since by \cref{propositionMultipolar} \eqref{it_Sie3i_0}, every ball \(B \in \mathcal{B} (t)\) intersects the set \(K\), it entails that
\(A(t)\subset\Bar\Omega\) whenever \(e^t \sum_{B \in \mathcal{B}(0)}\diam (B) \le\dist(K,\partial\Omega)\). Hence, \(A (T) \subset\Bar\Omega\), where
\begin{equation}
\label{eq_auch4aipah9vaephooPhie1h}
T\defeq\log \frac{\dist (K, \partial \Omega)}{\sum_{B \in \mathcal{B}(0)}\diam (B)}\in (\log 2,+\infty).
\end{equation}
According to \eqref{it_Sie3i_4} in \cref{propositionMultipolar}, we can write \([0,T) = \bigcup_{i = 1}^\ell [t_{i-1},t_i)\) with \(\ell\in\Nset_\ast\), \(t_0=0\), \(t_\ell=T\) and for every \(t\in [t_{i-1},t_{i})\), \(\mathcal{B} (t) =\{\Bar{B}_{\rho_j^ie^t} (a_j^i) \st j \in \{1,\dotsc,k_i\}\}\).

By \eqref{it_Sie3i_1} in \cref{propositionMultipolar} and by disjointness of the closed disks in \(\mathcal{B}(t)\) (see \cref{treeStructure}), we have 
\[
\partial B\cap\partial B'=\emptyset\quad\text{if \(B,B'\in\bigcup_{t\in[0,+\infty)}\mathcal{B}(t)\) with \(B\neq B'\).}
\]
We can thus define the function \(U : \Omega \to [0,+\infty)\) by
\[
U(x)\defeq
\begin{cases}
\sqrt{\frac{\Esing(\tr_{\partial B}u)}{\pi\, r^2}}
& \text{if \(x \in \partial B\) for some \(t\in[0,T)\) and \(B=\Bar B_r (a) \in \mathcal{B} (t)\),}\\
0& \text{otherwise.}
\end{cases}
\]
For every \(s, t \in [0, T)\) with \(s \leq t\), we write 
\begin{equation}
\label{decompose_lowerBnd}
\int_{A (t) \setminus A (s)} \frac{\abs{\Deriv u}^2}{2}=
\int_{A (t) \setminus A (s)} \biggl(U\abs{\Deriv u}-\frac{U^2}{2}\biggr)
+\int_{A (t) \setminus A (s)} \frac{(\abs{\Deriv u} - U)^2}{2}.
\end{equation}
By integration in polar coordinates, if \(t_{i-1}\le s\le t<t_i\) with \(i\in \{1,\dotsc,\ell\}\), then
\begin{multline*}
\int_{A (t) \setminus A (s)} \Biggl(U\abs{\Deriv u}-\frac{U^2}{2}\Biggr)
=\sum_{j=1}^{k_i}\int_{ \rho_j^ie^s}^{ \rho_j^ie^t}\biggl(\int_{\partial B_r (a^i_j)} \Big(U\abs{\Deriv u}-\frac{U^2}{2}\Big) \biggl)\dif r
\\
=\sum_{j=1}^{k_i}\int_{ \rho_j^ie^s}^{ \rho_j^ie^t}
   \biggl(\frac{\sqrt{\Esing(\tr_{\partial B_r (a^i_j)}u)}}{\sqrt{\pi}\, r} \Big(\int_{\partial B_r (a^i_j)} \abs{\Deriv u}  \Big) - \frac{\Esing(\tr_{\partial B_r (a^i_j)} u) }{r}\biggr)\dif r \\
    \ge\sum_{j=1}^{k_i}\int_{\rho_j^ie^s}^{ \rho_j^ie^t}\Esing(\tr_{\partial B_r (a^i_j)}u)\frac{\dif r}{r},
\end{multline*}
where in the last inequality we have used that in view of \eqref{eq_Aikachae4ieSohJ0oozaht7e}
\[
 \int_{\partial B_r (a_j^i)} \abs{\Deriv u}
 \ge \sqrt{4\pi \, \Esing (\tr_{\partial B_{r} (a_j^i)} u)}.
\]
We extend in this proof the notion of singular energy of a boundary map to the case of a disconnected open set with Lipschitz boundary by defining:
\[
\Esing (\tr_{\partial A(t)} u)\defeq \sum_{B \in \mathcal{B} (t)} \Esing (\tr_{\partial B} u).
\]
Then, using the change of variables \(r=\rho^i_je^\sigma\), we obtain for \(t_{i-1}\le s\le t<t_i\)  that
\begin{equation}
\label{eq_dahgeek6TheijahRauzuch6i}
\int_{A (t) \setminus A (s)} U\abs{\Deriv u}-\frac{U^2}{2}\ge  \int_{s}^{t}\Esing (\tr_{\partial A(\sigma)} u) \dif \sigma.
\end{equation}
We introduce for every \(t\in[0,T]\) the set
\[
A^-(t)\defeq \bigcup_{0\le s<t}A(s)\subset\Omega
\]
which is the union of a finite number of disjoint open disks in \(\Omega\)
\footnote{Note that \(T\) could be a merging time for which \(A(T)\not\subset\Omega\), while \(A^-(T)\subset\Omega\) by construction.}, and the  set
\[
A_i \defeq\bigcup_{i=1}^\ell A^-(t_i)\setminus A(t_{i-1})\subset\Omega\setminus K.
\]
We have by \eqref{decompose_lowerBnd} and \eqref{eq_dahgeek6TheijahRauzuch6i},
\begin{equation}
\label{eq_ashebi0xohxeiwaeS}
\begin{split}
  \int_{A_i} \frac{\abs{\Deriv u}^2}{2}
  \ge  \int_0^T \Esing (\tr_{\partial A(t)} u) \dif t
  +\int_{A_i} \frac{(\abs{\Deriv u} - U)^2}{2}.
 \end{split}
\end{equation}
By monotonicity of the singular energy (see \cref{lemma_evolutiondegres}) and of the sets \(A(t)\), we have for every \(s,t\in[0,T)\) such that \(s\le t\)
\begin{equation}
\label{singularLowerBound}
 \Esing (\tr_{\partial A(s)} u)\ge\Esing (\tr_{\partial A(t)} u) \ge \Esing(\tr_{\partial \Omega} u ).
\end{equation}
We then have by \eqref{eq_ashebi0xohxeiwaeS}, \eqref{singularLowerBound}, \eqref{hausdorff_confinement} and \eqref{eq_auch4aipah9vaephooPhie1h},
\begin{equation}
\label{lowerBoundGrow}
\begin{split}
 \int_{\Omega\setminus K}\frac{\abs{\Deriv u}^2}{2}
 &\ge  
\Esing (\tr_{\partial \Omega}u) (T-\log 2)
 +\int_0^{\log 2}  \Esing (\tr_{\partial A(t)}u)\dif t
+\int_A \frac{(\abs{\Deriv u} - U)^2}{2}
 \\
  &\ge 
     \Esing (\tr_{\partial \Omega}u) \log \frac{\dist(K,\partial \Omega)}{2\, (\mathcal{H}^1_\infty (K) + \varepsilon)}
    +(\log 2)\Esing (\tr_{\partial A(\log 2)}u) 
    + \int_A \frac{(\abs{\Deriv u} - U)^2}{2}.
 \end{split}
\end{equation}
In particular \eqref{lowerBoundGrow} already yields the first estimate \eqref{first_prop_lower_bound_compact} in view of \eqref{hausdorff_confinement} where \(\varepsilon>0\) is arbitrary.

In order to obtain the stronger estimate \eqref{second_prop_lower_bound_compact}, we first observe, integrating over \([\log 2,T)\) rather than \([0,T)\), that
\begin{equation}
\label{lowerBoundGrowSmall}
\begin{split}
 \int_{\Omega\setminus K}\frac{\abs{\Deriv u}^2}{2}
 &= \int_{A\setminus A(\log 2)}\frac{\abs{\Deriv u}^2}{2}
 +\int_{(\Omega\setminus K)\setminus(A\setminus A(\log 2))} \frac{\abs{\Deriv u}^2}{2}\\
 &\ge \int_{\log 2}^T  \Esing (\tr_{\partial A(t)}u)\dif t
 + \int_{(\Omega\setminus K)\setminus(A\setminus A(\log 2))} \frac{\abs{\Deriv u}^2}{2}
 \\
  &\ge\Esing (\tr_{\partial \Omega}u) \log \frac{\dist(K,\partial \Omega)}{2\,(\mathcal{H}^1_\infty (K) + \varepsilon)}
 +\int_{(\Omega\setminus K)\setminus(A\setminus A(\log 2))} \frac{\abs{\Deriv u}^2}{2}.
 \end{split}
\end{equation}
We observe now  that if \(s > 0\)
\begin{equation}\label{piefghqeafb}
\begin{split}
\mathcal{L}^2\Big(\Big\{ x \in \Omega \setminus K\st |\Deriv u(x)|\geq s \Big\} \Big)
&\leq \mathcal{L}^2 \Big(\Big\{ x \in (\Omega\setminus K)\setminus (A\setminus A(\log 2))\st \abs{\Deriv u(x)}\geq s \Big\} \Big) \\
&\qquad+\mathcal{L}^2\Big(\Big\{x  \in A\setminus A(\log 2)\st \abs{\abs{\Deriv u(x)}-U(x)}\geq \frac{s}{2}\Big\} \Big)\\
&\qquad+\mathcal{L}^2\Big(\Big\{x\in A\setminus A(\log 2)\st U(x)\ge\frac{s}{2}\Big\} \Big).
\end{split}
\end{equation}
By Chebyshev's inequality, we have for the first term in the right hand side of \eqref{piefghqeafb},
\begin{equation}
\label{Chebyshev_I}
s^2\mathcal{L}^2 \Big(\Big\{ x \in  (\Omega\setminus K)\setminus (A\setminus A(\log 2))\st |\Deriv u(x)|\geq s \Big\}\Big) \le
2\int_{ (\Omega\setminus K)\setminus (A\setminus A(\log 2))}\frac{\abs{\Deriv u}^2}{2}
\end{equation}
and for the second term,
\begin{equation}
\label{Chebyshev_II}
s^2\mathcal{L}^2 \Big(\Big\{ x \in A\setminus A(\log 2)\st \abs{|\Deriv u(x)|-U(x)}\geq \frac s2 \Big\}\Big) \le
8\int_{A\setminus A(\log 2)}\frac{(|\Deriv u|-U)^2}{2}.
\end{equation}
For the third term in the right hand side of \eqref{piefghqeafb} we first observe that, by definition of \(U\), we have 
\(\{x\in\Omega\setminus A(\log 2)\st U(x)\ge s/2\}=\bigcup_{B\in\mathcal{B}}\partial B\),
where 
\[
\mathcal{B}\defeq\bigg\{B=\Bar B_r(a)\in\bigcup_{t\in(\log 2,T]}\mathcal{B}(t)\st\frac{\Esing(\tr_{\partial B_r (a)}u)}{\pi r^2} \ge \frac{s^2}{4}\bigg\}.
\]
By construction, the set \(\bigcup_{B\in\mathcal{B}}\partial B\) can be covered by a collection of finite closed disks \(\Tilde{\mathcal{B}} \subset \cup_{t\in (\log 2, T]}\mathcal{B}(t)\) such that for every disk \(\Tilde{B} \in \Tilde{\mathcal{B}}\), \(\frac{\Esing(\tr_{\partial \Tilde{B}}u)}{\pi \diam(\Tilde{B})^2}\ge\frac{s^2}{16}\).
Hence,
\begin{equation*}
s^2\mathcal{L}^2\Big(\Big\{x\in A\setminus A(\log 2)\st U(x)\ge \frac s2\Big\}\Big)
\le s^2\sum_{D\in\mathcal{D}} \frac{\pi}{4}\diam(\Tilde{B})^2
 \le  4\sum_{D\in\mathcal{D}}
 \Esing(\tr_{\partial D}u).
 \end{equation*}
Since \(A(\log 2)\subset \bigcup_{D\in\mathcal{D}}D\), it entails by monotonicity of the singular energy (see \cref{lemma_evolutiondegres}) that 
\begin{equation}
\label{eq_pohxahli3eiyoJo9mu3ahYoh}
s^2\mathcal{L}^2\Bigl(\Big\{x\in A\setminus A(\log 2)\st U(x)\ge\frac s2\Big\}\Bigr)
 \le 4\,\Esing (\tr_{\partial A(\log 2)} u).
\end{equation}
In view of the lower bounds  \eqref{lowerBoundGrow} and \eqref{lowerBoundGrowSmall}, by summing \eqref{Chebyshev_I}, \eqref{Chebyshev_II} and \eqref{eq_pohxahli3eiyoJo9mu3ahYoh} and using \eqref{piefghqeafb}, we obtain
\[
\begin{split}
s^2\mathcal{L}^2\Big(\Big\{ x \in \Omega \setminus K\st |\Deriv u|\geq s \Big\} \Big)
\le
C_2\Big(\int_{\Omega\setminus K}\frac{\abs{\Deriv u}^2}{2} - \Esing (\tr_{\partial \Omega}u) \log \frac{\dist(K,\partial \Omega)}{2\,(\mathcal{H}^1_\infty + \varepsilon)}\Big),
\end{split}
\]
where \(C_2=10+\frac{4}{\log 2}\). The conclusion follows by letting \(\varepsilon\to 0\).
\end{proof}

\subsection{Lower bound on renormalised energies}
We now prove that the renormalised energies are bounded from below. The lower-bound we will obtain depends on the tubular neighbourhood extension energy:
\begin{definition}
\label{def_energy_extension}
Let \(\Omega\subset\Rset^2\) be a Lipschitz bounded domain. For every map \(g \in W^{1/2,2}(\partial \Omega, \manifold{N})\) we define the \emph{tubular neighbourhood extension energy} to be
  \begin{equation*}
    \mathcal{E}^{\mathrm{ext}}(g)
    \defeq
    \inf\Bigl\{\int_{\partial \Omega \times [0,1]} \frac{\abs{\Deriv u}^2}{2} \st u \in W^{1,2}(\partial \Omega\times [0,1],\manifold{N}) \text{ and } \tr_{\partial \Omega \times \{0 \}}u=g \Bigr\}.
  \end{equation*}
\end{definition}

It is known that for every \(g \in W^{1/2, 2} (\partial \Omega, \manifold{N})\), 
one has \(\mathcal{E}^{\mathrm{ext}}(g)<+\infty\) (see for example \cite{Bethuel_Demengel_1995}).
The quantity \(\mathcal{E}^{\mathrm{ext}}(g)\) is however not controlled in terms of the \(W^{1/2,2}\) norm  of the map \(g\): if \(\pi_1 (\manifold{N}) \not \simeq \{0\}\), a bubbling sequence gives a counterexample (see for example \citelist{\cite{Berlyand_Mironescu_Rybalko_Sandier_2014}*{Lemma 2.1}\cite{Mironescu_VanSchaftingen_Trace}}).

\begin{lemma}
\label{lemma_extension}
For every Lipschitz bounded domain \(\Omega \subset \Rset^2\),
there exist \(\delta > 0\) and a constant \(C\in (0,+\infty)\) such that if \(g \in W^{1/2, 2} (\partial \Omega, \manifold{N})\), there exists 
\(G \in W^{1, 2} (\Omega_\delta \setminus \Bar{\Omega}, \manifold{N})\), where \(\Omega_\delta\defeq\{x\in \Rset^2 \st \dist(x,\Omega)<\delta\}\),
such that \(\tr_{\partial \Omega} G = g\) and 
\[
 \int_{\Omega_\delta \setminus \Bar{\Omega}} \frac{\abs{\Deriv G}^2}{2}
 \le C
 \mathcal{E}^{\mathrm{ext}} (g).
\]
\end{lemma}

The proof of \cref{lemma_extension} consists in constructing a bi-Lipschitz map \(\phi: \partial \Omega \times [0,1] \rightarrow\Omega_\delta \setminus \Omega\) such that \( \phi((x,0))=x\) for all \(x\) in \(\partial \Omega\). We reach the conclusion by taking  \(G\defeq u\compose(\phi^{-1})_{\vert\Omega_\delta\setminus\Bar\Omega}\) with \(u\) an arbitrary admissible function in the infimum defining \(\mathcal{E}_\mathrm{ext}(g)\). The details are left to the reader.

\begin{proposition}
For every Lipschitz bounded domain \(\Omega \subset \Rset^2\),
there exists a constant \(C\in(0,+\infty)\) such that for every \(g \in W^{1/2,2} (\partial \Omega, \manifold{N})\), for every \(k\in\Nset_\ast\), for every minimal topological resolution \((\gamma_1, \dotsc, \gamma_k)\in W^{1/2,2}(\Sset^1, \manifold{N})\) of \(g\) and for every \((a_1, \dotsc, a_k) \in \Conf{k} \Omega\),
\[
 \mathcal{E}^\mathrm{geom}_{g, \gamma_1, \dotsc, \gamma_k} (a_1, \dotsc, a_k)
 \ge
 \mathcal{E}^\mathrm{top}_{g, \gamma_1, \dotsc, \gamma_k} (a_1, \dotsc, a_k)
 \ge 
\Esing (g) \log \frac{1}{C \Esing (g)}
- C \Eext (g).
\]
\end{proposition}
\begin{proof}
The first inequality has already been seen in \eqref{eq_comparison_top_geom}.
Let \(G\) be given by \cref{lemma_extension}.
If \(\rho\in(0,\Bar\rho(a_1,\dotsc,a_k))\) and \(u \in W^{1, 2} (\Omega \setminus \bigcup_{i = 1}^k \Bar{B}_{\rho} (a_i), \manifold{N})\), we have by the first estimate of \cref{prop_lower_bound_compact} applied to the extension of \(u\) to \(\Omega_\delta\) by \(G\),
\[
\int_{\Omega \setminus \bigcup_{i = 1}^k B_{\rho} (a_i)} \frac{\abs{\Deriv u}^2}{2}
\ge \Esing (g) \log \frac{\delta}{4 k \rho}
- \int_{\Omega_\delta \setminus \Bar{\Omega}} \frac{\abs{\Deriv G}^2}{2}.
\]
By minimality of the topological resolution \((\gamma_1,\dotsc,\gamma_k)\), \( \Esing (g)=\sum_{i = 1}^k \frac{\equivnorm{\gamma_i}^2}{4 \pi}\) and it follows that 
\[
  \mathcal{E}^\mathrm{top}_{g, \gamma_1, \dotsc, \gamma_k} (a_1, \dotsc, a_k)
  \ge \Esing (g) \log \frac{\delta}{4 k}
   - C_1 \mathcal{E}^{\mathrm{ext}} (g).
\]
By the definition of the systole given in \eqref{eq_systole}, we have 
\(\Esing (g)=\sum_{i = 1}^k \frac{\equivnorm{\gamma_i}^2}{4 \pi}\ge k \frac{\syst (\manifold{N})^2}{4 \pi}\), and the conclusion follows.
\end{proof}

\section{Coercivity of the renormalised energies}
We describe the behaviour of the renormalised energy under collisions of singularities with each other and with the boundary \(\partial \Omega\).

\begin{proposition}
\label{proposition_coercivity}
Let \(\Omega\subset\Rset^2\) be a bounded Lipschitz domain, \(g \in W^{1/2,2} (\partial \Omega, \manifold{N})\), \(k\in\Nset_\ast\) and \((\gamma_1, \dotsc, \gamma_k)\in W^{1/2,2}(\Sset^1, \manifold{N})^k\) be a minimal topological resolution of \(g\).
If \((a_1^n, \dotsc, a_k^n)_{n \in \Nset}\) is a sequence in 
\(\Conf{k} \Omega\) converging to some \((a_1, \dotsc, a_k) \in \Bar{\Omega}^k\) and if 
\[
 \limsup_{n \to \infty} \mathcal{E}_{g, \gamma_1, \dotsc, \gamma_k}^{\mathrm{top}}
 (a_1^n, \dotsc, a_k^n) < + \infty,
\]
then for every \(i \in \{1, \dotsc, k\}\), \(a_i \in  \Omega\)
and if 
\(\{a_1,\dots,a_k\} = \{\Tilde{a}_1,\dots,\Tilde{a}_\ell\}\)
 with \(\Tilde{a}_1,\dots,\Tilde{a}_\ell\) distinct points
 and for each \(j\in\{1,\dots,\ell\}\), \(I_j\defeq \{i \in \{1, \dotsc, k\} \st a_i=\Tilde{a}_j\}\),
  then for each \(j\in\{1,\dots,\ell\}\), there exists \(\Tilde{\gamma}_j\in \mathcal{C}^1 (\Sset^1, \manifold{N})\) such that
\begin{equation*}
 \equivnorm{\Tilde{\gamma}_j}^2 = \sum_{i \in I_j} \equivnorm{\gamma_i}^2,
\end{equation*}
the maps \((\gamma_i)_{i \in I_j}\) are a topological resolution of \(\Tilde{\gamma}_j\), and the maps \(\Tilde{\gamma}_1,\dotsc,\Tilde{\gamma_\ell}\) are a topological resolution of \(g\).
\end{proposition}

The assumption that \((\gamma_1, \dotsc, \gamma_k)\in W^{1/2,2}(\Sset^1, \manifold{N})^k\) is a minimal topological resolution of \(g\) entails in particular that  \(\equivnorm{\gamma_i}>0\) for each \(i \in \{1, \dotsc, k\}\). 

\Cref{proposition_coercivity} implies the classical result of confinement away from the boundary and non-collision of vortices for the classical Ginzburg--Landau case \(\manifold{N} = \Sset^1\) \cite{Bethuel_Brezis_Helein_1994}*{Theorem I.10}.
Indeed, the homotopy class of a map \(\gamma:\Sset^1\to\Sset^1\) is classified by its degree \(\deg \gamma \in \Zset\), and \(\equivnorm{\gamma} = 2 \pi \abs{\deg \gamma}\); the minimality condition imposes that all the maps \(\gamma_1, \dotsc, \gamma_k\) share the same degree \(1\) or \(-1\); the collision condition imposes that if \(m\) singularities converge to the same point, then \(m^2 = m\), that is, \(m = 1\).

In general collision can occur, as can be seen by considering \(\manifold{N} = \Sset^1 \times \Sset^1\).
The Dirichlet energy decouples as the sum of two functionals for maps in \(\Sset^1\) each of which being applied to one component.
The renormalised energy is then the sum of two renormalised energies and nothing prevent singularities of both components to merge in the limit.
Algebraically this is related to the Pythagorean identity occuring at the collision: \(\abs{(1, 1)}^2 = \abs{(1, 0)}^2 + \abs{(0, 1)}^2\). 
A similar situation occurs in models of superfluid Helium 3 in the dipole-free A phase (see \S \ref{sect_gooWae7aNu2roatoozahb0al}).

\Cref{proposition_coercivity} does not prevent collisions from occuring when two atomic singularities can merge in an atomic singularity as  is the case in \cref{sect_gooWae7aNu2roatoozahb0al}.

The proof of \cref{proposition_coercivity} relies on the following estimate.
\begin{lemma}
\label{lemma_singularities_potential}
Let \(\Omega\subset \Rset^2\) be a Lipschitz domain. Let \(a\in\Rset^2\), \(\sigma,\tau\in(0,+\infty)\) with \(\sigma<\tau\), and \(u \in W^{1, 2} (B_{\tau} (a) \setminus \Bar{B}_{\sigma} (a), \manifold{N})\),
then, 
\begin{equation}
\label{eq_ahrae9aach4ahLoungi8jooy}
 \int_{(B_{\tau} (a) \setminus \Bar{B}_{\sigma} (a)) \cap \Omega} \frac{\abs{\Deriv u}^2}{2}
 \ge
 \frac{\equivnorm{\gamma}^2}{4 \pi \nu_{\tau, \sigma} (a)}
  \log \frac{\tau}{\sigma}
 \Biggl(1 -
 \biggl(\frac{2 \pi \int_{B_{\tau} (a) \setminus (\Omega \cup \Bar B_{\sigma} (a))} \abs{\Deriv u}^2 }
 {\equivnorm{\gamma}^{2}  \log \frac{\tau}{\sigma}  }\biggr)^{\frac 12}\Biggr)^2,
\end{equation}
where \(\gamma\in\mathcal{C}^1(\Sset^1,\manifold{N})\) is a map homotopic to \(\tr_{\partial B_r (a)}u\) for every \(r \in [\sigma, \tau]\) and
\[
\nu_{\tau, \sigma} (a)
\defeq \frac{1}{2 \pi  \log \frac{\tau}{\sigma}}\int_{(B_{\tau} (a) \setminus \Bar{B}_\sigma (a)) \cap \Omega}
\frac{1}{\abs{x - a}^2} \dif x.
\]
\end{lemma}
\begin{proof}
For every \(r \in (\sigma, \tau)\), we have in view of \eqref{eq_kahz1OeBa9gooqu6ep1aixei}
\[
 \frac{1}{r}\int_{\partial B_{r} (a)} \abs{\Deriv  {u}}\ge \frac{\equivnorm{\gamma}}{r},
\]
and therefore by integration in polar coordinates
\begin{equation}
  \label{eq_nahRe0aiFeimephii}
\int_{B_{\tau} (a) \setminus \Bar{B}_\sigma (a)} \frac{\abs{\Deriv  {u} (x)}}{\abs{x - a}}\dif x \ge \equivnorm{\gamma}  \log \frac{\tau}{\sigma}.
\end{equation}
On the other hand, by additivity of the integral and by the Cauchy--Schwarz inequality, we have in view of \eqref{eq_ahrae9aach4ahLoungi8jooy}
\begin{equation}
  \label{eq_ohlaeB3ilah0moovu}
 \begin{split}
\int_{B_{\tau} (a) \setminus \Bar{B}_\sigma (a)} &\frac{\abs{\Deriv  {u} (x)}}{\abs{x - a}} \dif x\\
&\le \Bigl(\int_{(B_{\tau} (a) \setminus \Bar{B}_{\sigma} (a)) \cap \Omega} \frac{1}{\abs{x - a}^2} \dif x \Bigr)^{\frac{1}{2}}
\biggl(\int_{(B_{\tau} (a) \setminus \Bar{B}_{\sigma} (a)) \cap \Omega} \abs{\Deriv u}^2 \dif x\biggr)^{\frac{1}{2}}\\
&\qquad +
\Bigl(\int_{(B_{\tau} (a) \setminus \Bar{B}_{\sigma} (a)) \setminus  \Omega} \frac{2}{\abs{x - a}^2}\Bigr)^{\frac{1}{2}}
\biggl(\int_{(B_{\tau} (a) \setminus \Bar{B}_{\sigma} (a)) \setminus \Omega} \frac{\abs{\Deriv u}^2}{2}\biggr)^{\frac{1}{2}}\\
&\leq \Bigl(4 \pi \nu_{\tau, \sigma}(a) \log \frac{\tau}{\sigma} 
\int_{(B_{\tau} (a) \setminus \Bar{B}_{\sigma} (a)) \cap \Omega} \frac{\abs{\Deriv u}^2}{2}\Bigr)^{\frac{1}{2}}\\
&\qquad + \Bigl(4 \pi \log \frac{\tau}{\sigma}\int_{B_{\tau} (a) \setminus (\Omega \cup \Bar{B}_{\sigma} (a))} \frac{\abs{\Deriv u}^2}{2} \Bigr)^{\frac{1}{2}}.
\end{split}
\end{equation}
We reach the conclusion by \eqref{eq_nahRe0aiFeimephii} and \eqref{eq_ohlaeB3ilah0moovu}.
\end{proof}

\begin{proof}%
[Proof of \cref{proposition_coercivity}]
Up to a permutation of the indices, we can assume that \(I_j = \{k_{j - 1}, \dotsc, k_j-1\}\), for \(j\in \{1,\dotsc, l\}\), where \(k_0=1<k_1<\dotsc<k_\ell=k+1\).

Let \(\delta>0\), \(\Omega_\delta\) and \(G  \in W^{1, 2} (\Omega_\delta \setminus \Bar{\Omega}, \manifold{N})\) be given by \cref{lemma_extension}
and choose \(\tau \in (0, \delta)\) such that if \(i, j \in \{1, \dotsc, \ell\}\) and \(i \ne j\), then \(\abs{\Tilde{a}_i - \Tilde{a}_j} \ge 4 \tau\).
We also take \(\rho,\sigma>0\) such that \(2\rho<\sigma < \tau\) so that \(B_\rho (a_i^n) \subset B_{\sigma/2} (a_i)\) for each \(i\) and for \(n\) large enough, condition that we assume to be true from now on.

For \(u \in W^{1, 2} (\Omega_\delta \setminus \bigcup_{i = 1}^k \Bar{B}_{\rho} (a_i^n), \manifold{N})\) such that \(u=G\) on \(\Omega_\delta\setminus\Bar{\Omega}\) and such that \(\tr_{\Sset^1} u(a^n_i+\rho\,\cdot)\) and \(\gamma_i\) are homotopic for each \(i\in\{1,\dots,k\}\), we first estimate
\begin{equation}
\label{firstDecomposition}
 \int_{\Omega_\delta \setminus \bigcup_{i = 1}^k \Bar{B}_{\rho} (a_i^n)}  \frac{\abs{\Deriv  {u}}^2}{2}
 \ge \sum_{j = 1}^\ell \int_{B_{\tau} (\Tilde{a}_j) \setminus B_{\sigma} (\Tilde{a}_j) } \frac{\abs{\Deriv  {u}}^2}{2}
 + \sum_{j = 1}^\ell \int_{B_{\sigma} (\Tilde{a}_j) \setminus \bigcup_{i = k_{j - 1}}^{k_{j}-1}\Bar{B}_{\rho} (a_i^n)} \frac{\abs{\Deriv  {u}}^2}{2}.
  \end{equation}
We define \(\Tilde{\gamma}_j \in \mathcal{C} (\Sset^1, \manifold{N})\) to be a map which is homotopic to \( \tr_{\Sset^1} u (\Tilde{a}_j + r\,\cdot)\) for almost every \(r \in (\sigma, \tau)\).
By the minimality assumption on \((\gamma_1, \dotsc, \gamma_k)\), we have 
\begin{equation}
\label{eq_aCheef4bohtoh0jeuraif1ei}
  \frac{\equivnorm{\Tilde{\gamma}_j}^2}{4 \pi}
  \ge \sum_{i = k_{j - 1}}^{k_j - 1}
  \frac{\equivnorm{\gamma_i}^2}{4 \pi} > 0.
\end{equation}
If we let \(
\eta \defeq \int_{\Omega_\delta \setminus \Omega} \abs{\Deriv G}^2
\),
we have by \cref{lemma_singularities_potential}
\begin{equation}
\label{eq_coercivity_outer}
 \int_{(B_{\tau} (\Tilde{a}_j) \setminus \Bar{B}_{\sigma} (\Tilde{a}_j)) \cap \Omega} \frac{\abs{\Deriv u}^2}{2}
 \ge
 \frac{\equivnorm{\Tilde{\gamma}_j}^2}{4 \pi \nu_{\tau, \sigma} (\Tilde{a}_j)}
  \log \frac{\tau}{\sigma}
 \Biggl(1 -
 \Bigl(\frac{2 \pi \eta }
 {\equivnorm{\Tilde{\gamma}_j}^{2}  \log \frac{\tau}{\sigma}  }\Bigr)^{\frac 12}\Biggr)^2.
\end{equation}
Since for every \(i\in\{k_{j-1},\dotsc,k_j-1\}\) we have \( B_{\rho} (a^n_i)\subset B_{\sigma/2}(\Tilde{a}_j)\), the maps \(\gamma_{k_{j-1}},\dotsc,\gamma_{k_j-1}\) form a topological resolution of \(\tilde{\gamma}_j\) and \(\dist(B_{\rho} (a^n_i), \partial B_{\sigma}(\Tilde{a}_j))\ge \frac\sigma 2\). Since \(\mathcal{H}^1_\infty( \bigcup_{i = k_{j - 1}}^{k_{j}-1}\Bar{B}_{\rho} (a_i^n)) \leq 2\rho(k_j-k_{j-1})\), by \cref{prop_lower_bound_compact} and \eqref{eq_aCheef4bohtoh0jeuraif1ei}, we have
\begin{equation}
\begin{split}
  \label{eq_kuaRaiPhupha2EiSa}
  \int_{B_{\sigma} (\Tilde{a}_j) \setminus \bigcup_{i = k_{j - 1}}^{k_{j}-1}\Bar{B}_{\rho} (a_i^n)} \frac{\abs{\Deriv  {u}}^2}{2}
  &\ge \Esing (\Tilde{\gamma}_i) \log \frac{\sigma}{4 (k_{j} - k_{j - 1})\rho}\\
  &\ge \sum_{i = k_{j - 1}}^{k_j-1}\frac{\equivnorm{\gamma_i}^2}{4 \pi} \log \frac{\sigma}{4 (k_{j} - k_{j - 1})\rho}.
  \end{split}
\end{equation}
By \eqref{firstDecomposition}, \eqref{eq_coercivity_outer} and \eqref{eq_kuaRaiPhupha2EiSa}, taking the limit \(\rho\to 0\), we deduce that
\begin{equation}
\label{ineq_W_top_lower}
\begin{split}
  \liminf_{n \to \infty} \mathcal{{E}}_{g, \gamma_1, \dotsc, \gamma_k}^{\mathrm{top}} (a^n_1, \dotsc, a^n_k)
 \ge & \sum_{j = 1}^{\ell} \biggl[
 \frac{\equivnorm{\Tilde{\gamma}_j}^2}{4 \pi \nu_{\tau, \sigma} (\Tilde{a}_j)}
 \biggl(1 -
 \Bigl(\frac{2 \pi \eta }
 {\equivnorm{\Tilde{\gamma}_j}^{2}  \log \frac{\tau}{\sigma}  }\Bigr)^{\frac{1}{2}}\biggr)^2
  \log \frac{\tau}{\sigma}
 \\
 & \qquad\qquad + \sum_{i = k_{j - 1}}^{k_j-1}\frac{\equivnorm{\gamma_i}^2}{4 \pi} \log \frac{\sigma}{4 (k_{j} - k_{j - 1})}\biggr].
\end{split}
\end{equation}

We assume now that we have \(\limsup_{n \to \infty} \mathcal{E}_{g, \gamma_1, \dotsc, \gamma_k}^{\mathrm{top}} (a_1^n, \dotsc, a_k^n) < + \infty\).
By letting \(\sigma \to 0\) in \eqref{ineq_W_top_lower}, we deduce that
\[
 \sum_{j = 1}^{\ell}  \frac{\equivnorm{\Tilde{\gamma}_j}^2}{\limsup_{\sigma \to 0} \nu_{\tau, \sigma} (\Tilde{a}_j)} \le\sum_{i = 1}^{k} \equivnorm{\gamma_i}^2.
\]
Since \((\gamma_1, \dotsc, \gamma_k)\) and \((\Tilde{\gamma}_1, \dotsc, \Tilde{\gamma}_\ell)\) are two topological resolutions of \(g\), with \((\gamma_1, \dotsc, \gamma_k)\) which is minimal, and since \(\nu_{\tau, \sigma} (\Tilde{a}_j) \le 1\), we necessarily have, in view of \eqref{eq_aCheef4bohtoh0jeuraif1ei},
\(
 \sum_{j = 1}^{\ell} \equivnorm{\Tilde{\gamma}_j}^2 = \sum_{i = 1}^{k} \equivnorm{\gamma_i}^2\)
and for all \(j\in\{1,\dotsc,l\}\) such that \(\equivnorm{\Tilde{\gamma}_j}>0\),
\(\limsup_{\sigma \to 0} \nu_{\tau, \sigma} (\Tilde{a}_j) = 1\).
Since the boundary \(\partial \Omega\) is Lipschitz, the latter condition implies that \(\Tilde{a}_j \in \Omega\).
\end{proof}

\section{Renormalisable maps and renormalisable singular harmonic maps}
\label{section_ren2}

\subsection{Renormalised energy of renormalisable maps}
We define renormalisable maps as maps that have square-integrable derivative away from a set of singularities and whose derivative remains controlled near the singularity.

\begin{definition}
\label{def_renormalisable}
Let \(\Omega\subset\Rset^2\) be an open set. A map \(u:\Omega\to\manifold{N}\) is \emph{renormalisable} whenever there exists a finite set \(\mathcal{S}\subset \Omega\) such that if \(\rho > 0\) is small enough,
  \(u \in W^{1, 2} (\Omega \setminus \bigcup_{a\in\mathcal{S}} \Bar{B}_\rho (a) , \manifold{N})\) and its \emph{renormalised energy} \(\mathcal{E}^{\mathrm{ren}}(u)\) is finite, where
  \begin{equation*}
    \mathcal{E}^{\mathrm{ren}}(u)
    \defeq \liminf_{\rho \to 0} \int_{\Omega\setminus\bigcup_{a\in\mathcal{S}} \Bar{B}_\rho(a)}\frac{\abs{\Deriv u}^2}{2}-\sum_{a\in\mathcal{S}}
    \frac{\equivnorm{\tr_{\partial B_\rho (a)} u}^2}{4\pi}\log\frac{1}{\rho}<+\infty.
  \end{equation*}
We denote by \(W^{1, 2}_{\mathrm{ren}} (\Omega, \manifold{N})\) the set of renormalisable mappings.
\end{definition}
Let \(u\in W^{1, 2}_{\mathrm{ren}} (\Omega, \manifold{N})\) be a renormalisable map and \(\mathcal{S}_0\subset\Omega\) be the minimal set such that \(u \in W^{1, 2} (\Omega \setminus \bigcup_{a\in\mathcal{S}_0} \Bar{B}_\rho (a) , \manifold{N})\). In other words, \(\mathcal{S}_0\) is the set of points \(a\in\Omega\) such that \(\Deriv u\) is not square-integrable near \(a\). Hence, if \(a\in\Omega\setminus\mathcal{S}_0\), then \(\equivnorm{\tr_{\partial B_\rho (a)} u}=0\) for \(\rho>0\) small enough. In particular, the \(\liminf\) defining \(\mathcal{E}^{\mathrm{ren}}(u)\) does not depend on the set \(\mathcal{S}\) and \(\mathcal{E}^{\mathrm{ren}}(u)\) is well-defined.

If \(\mathcal{S}=\emptyset\), then \(\mathcal{E}^{\mathrm{ren}}(u)=\int_\Omega \abs{\Deriv u}^2\). If \(\mathcal{S}=\{a_1,\dotsc,a_k\}\) with \(k\in\Nset_\ast\) and \((a_1,\dotsc,a_k)\in\Conf{k}\Omega\), then we have by \cref{lemma_energy_growth_map_radii} that the map
\[
\rho\in (0, \Bar{\rho} (a_1, \dotsc, a_k))\mapsto
\int_{\Omega\setminus\bigcup_{i=1}^k \Bar{B}_\rho(a_i)}\frac{\abs{\Deriv u}^2}{2}-\sum_{i = 1}^k \frac{\equivnorm{\tr_{\partial B_\rho (a_i)} u}^2}{4\pi} \log\frac{1}{\rho}
\]
is non-increasing. We have thus
  \begin{equation}
\label{renormalised_energy_u}
   \begin{split}
    \mathcal{E}^{\mathrm{ren}}(u)
    &=
      \lim_{\rho \to 0}
        \int_{\Omega\setminus\bigcup_{i=1}^k \Bar{B}_\rho(a_i)}
          \frac{\abs{\Deriv u}^2}{2}
          -\sum_{i=1}^k \frac{\equivnorm{\tr_{\partial B_\rho (a_i)} u}^2}{4\pi} \log\frac{1}{\rho}\\
    &= \sup_{\rho\in (0, \Bar{\rho} (a_1, \dotsc, a_k))}
    \int_{\Omega\setminus\bigcup_{i=1}^k \Bar{B}_\rho(a_i)}\frac{\abs{\Deriv u}^2}{2}
    -\sum_{i=1}^k \frac{\equivnorm{\tr_{\partial B_\rho (a_i)} u}^2}{4\pi} \log\frac{1}{\rho}.
    \end{split}
  \end{equation}

\begin{proposition}
\label{admissibleMaps}
Let \(\Omega\subset\Rset^2\) be an open set. If \(u\in W^{1,2}_{\mathrm{ren}}(\Omega,\manifold{N})\), then either one has \(u\in W^{1,2}(\Omega,\manifold{N})\) or there exist \(k\in\Nset_\ast\), \((a_1, \dotsc, a_k)\in\Conf{k}\Omega\) and \(\gamma_1, \dotsc,\gamma_k\in \mathcal{C}^1 (\Sset^1, \mathcal{N})\) such that
\begin{enumerate}[(i)]
  \item \label{it_Joh2iexa0shoh8no5xoXuoli}
for every \(\rho \in (0,\Bar{\rho} (a_1, \dotsc, a_k))\), we have
  \[
  \int_{\Omega \setminus \bigcup_{i = 1}^k B_\rho (a_i)}
  \frac{\abs{\Deriv u}^2}{2}
  +\sum_{i=1}^k\int_{B_\rho(a_i)}
  \frac{1}{2}\Bigl(\abs{\Deriv u(x)} - \frac{\equivnorm{\gamma_i}}{2 \pi \abs{x-a_i}}\Bigr)^2\dif x
  \le \mathcal{E}^{\mathrm{ren}} (u)
  + \sum_{i = 1}^k \frac{\equivnorm{\gamma_i}^2}{4 \pi}
  \log \frac{1}{\rho},
  \]
  in particular, \(u\in W^{1,1}(\Omega,\manifold{N})\), 
  \item for each \(i \in \{1, \dotsc, k\}\), \(\gamma_i\) is a non-trivial minimising closed geodesic,
  \item for each \(i \in \{1, \dotsc, k\}\), there exists a sequence \((\rho_\ell)_{\ell \in \Nset}\) converging to \(0\) such that the sequence \((\tr_{\Sset^1} u (a_i + \rho_\ell\,\cdot))_{\ell \in \Nset}\) converges strongly to \(\gamma_i\) in \(W^{1, 2} (\Sset^1, \manifold{N})\),
  \item\label{synhamonicConvergenceSing} for each \(i \in \{1, \dotsc, k\}\),
  \(
  \lim_{\rho \to 0}\synhar{\tr_{\Sset^1}u(a_i+\rho\,\cdot)}{\gamma_i}=0,
  \)
    \item \label{it_bohp7shaephei0eeT}
if \(\Omega\) is a bounded Lipschitz domain, then \((\gamma_1, \dotsc, \gamma_k)\) is a topological resolution of \(g=\tr_{\partial \Omega} u\) and
  \(
  \displaystyle
    \mathcal{E}^{\mathrm{ren}}(u)\ge \mathcal{E}^{\mathrm{geom}}_{g, \gamma_1, \dotsc, \gamma_k}(a_1,\dotsc,a_k).
    \)
\end{enumerate}
\end{proposition}

We denote the set of singularities of \(u \in W^{1, 2}_{\mathrm{ren}} (\Omega, \manifold{N})\) by 
\[
\operatorname{sing}(u)=\{(a_1,\gamma_1), \dotsc, (a_k,\gamma_k)\}
\]
whenever 
\(u \in W^{1, 2}_{\mathrm{loc}} (\Bar{\Omega} \setminus \{a_1, \dotsc, a_k\}, \manifold{N})\) and for each \(i \in \{1, \dotsc, k\}\), \(\lim_{\rho \to 0}\synhar{\tr_{\Sset^1}u(a_i+\rho\,\cdot)}{\gamma_i}=0\) and \(\gamma_i\) is not homotopic to a constant map. Since the synharmony is not positive-definite, \((\gamma_1, \dotsc, \gamma_k)\) are only defined up to a synharmony class and so strictly speaking \(\operatorname{sing} (u)\) is well-defined up to synharmony of \(\gamma_1, \dotsc, \gamma_k\). 

In the particular case where \(u \in W^{1, 2} (\Omega, \manifold{N})\subset W^{1, 2}_{\mathrm{ren}} (\Omega, \manifold{N})\), we have \(\operatorname{sing} (u) = \emptyset\).
In general, we have  \((a, \gamma)\in\operatorname{sing}(u)\) if and only if \(\abs{\Deriv u}^2\) is not integrable near \(a\) and \(\lim_{\rho \to 0}\synhar{\tr_{\Sset^1}u(a_i+\rho\,\cdot)}{\gamma}=0\).

\begin{proof}%
  [Proof of \cref{admissibleMaps}]%
  \resetconstant
  Let \(\mathcal{S}\subset\Omega\) be given by the definition of renormalisable mappings (\cref{def_renormalisable}). Without loss of generality, we assume that \(\mathcal{S}\neq\emptyset\) so that there exist \(k\in \Nset_\ast\) and \((a_1,\dotsc,a_k)\in\Conf{k}\Omega\) such that \(\mathcal{S}=\{a_1,\dotsc,a_k\}\).

For every \(i \in \{1, \dotsc, k\}\), we let \(\Hat{\gamma}_i \in \mathcal{C} (\Sset^1, \manifold{N})\) be a map homotopic in \(\VMO (\Sset^1, \manifold{N})\) to the map \(\tr_{\Sset^1} u(a_i+r\,\cdot)\) for every \(r \in (0,\Bar{\rho} (a_1, \dotsc, a_k))\).
  
\medbreak
\noindent \emph{Step 1. Upper bound near singularities.} 
For every \( i\in\{1,\dotsc,k\}\), \(\rho\in (0,\Bar{\rho}(a_1,\dots,a_k)]\) and for almost every \(r \in (0,\rho)\), we have \(\int_{\partial B_r (a_i)} \frac{\abs{\Deriv u}^2}{2}\ge \frac{\equivnorm{\Hat{\gamma}_i}^2}{4\pi r}\). Hence, by definition of the renormalised energy and by monotone convergence, we have
\begin{equation}
\label{firstLowerRenorm}
\begin{split}
&\mathcal{E}^{\mathrm{ren}}(u)
  + \sum_{i = 1}^k \frac{\equivnorm{\Hat{\gamma}_i}^2}{4 \pi}
  \log \frac{1}{\rho}\\
  &=\lim_{\sigma\to 0}\biggl(
   \int_{\Omega \setminus \bigcup_{i = 1}^k B_\rho (a_i)}
  \frac{\abs{\Deriv u}^2}{2}
  + \sum_{i = 1}^k\int_{ B_\rho (a_i) \setminus B_\sigma (a_i)}
  \frac{\abs{\Deriv u}^2}{2}
   - \sum_{i = 1}^k \frac{\equivnorm{\Hat{\gamma}_i}^2}{4 \pi}
  \log \frac{\rho}{\sigma}\biggr)\\
 & = \int_{\Omega \setminus \bigcup_{i = 1}^k B_\rho (a_i)}
  \frac{\abs{\Deriv u}^2}{2}
  +
   \sum_{i = 1}^k\int_0^\rho\biggl(\int_{\partial B_r (a_i)}
\frac{\abs{\Deriv u}^2}{2} -\frac{\equivnorm{\Hat{\gamma}_i}^2}{4\pi r}\biggr)\dif r.
  \end{split}
\end{equation}
On the other hand, for every \( i\in\{1,\dotsc,k\}\) and for almost every \(r \in (0,\rho)\), we have that \(\int_{\partial B_r (a_i)} \abs{\Deriv u}\ge \equivnorm{\Hat{\gamma}_i}\) and thus
  \[
  \begin{split}
    \int_{\partial B_r (a_i)}
    \Bigl(\abs{\Deriv u} - \frac{\equivnorm{\Hat{\gamma}_i}}{2 \pi r}\Bigr)^2
    & = \int_{\partial B_r (a_i)} \abs{\Deriv u}^2 
    - \int_{\partial B_r (a_i)} \frac{\abs{\Deriv u} \equivnorm{\Hat{\gamma}_i}}{\pi r} 
    + \frac{\equivnorm{\Hat{\gamma}_i}^2}{2 \pi r}\\
    & \le \int_{\partial B_r (a_i)} \abs{\Deriv u}^2
    - \frac{\equivnorm{\Hat{\gamma}_i}^2}{2 \pi r}.
  \end{split}
  \]
 Hence,
\begin{equation}
\label{wtopLwr}
\begin{split}
\mathcal{E}^{\mathrm{ren}} (u)
  + \sum_{i = 1}^k \frac{\equivnorm{\Hat{\gamma}_i}^2}{4 \pi}
  \log \frac{1}{\rho}
  \ge
  \int_{\Omega \setminus \bigcup_{i = 1}^k B_\rho (a_i)}
  \frac{\abs{\Deriv u}^2}{2}
  +\sum_{i=1}^k\int_{B_\rho(a_i)}
  \frac{1}{2}\Bigl(\abs{\Deriv u(x)} - \frac{\equivnorm{\Hat{\gamma}_i}}{2 \pi \abs{x-a_i}}\Bigr)^2\dif x.
  \end{split}
\end{equation}
In particular, if \(\Hat{\gamma}_i\) is homotopic to a constant map, then we have \(\equivnorm{\Hat{\gamma}_i} = 0\) and
\(u \in W^{1, 2}_{\mathrm{loc}} (\Omega \setminus \{a_1, \dotsc, a_{i - 1}, a_{i + 1}, \dotsc, a_k\}, \manifold{N})\) since the right-hand side of \eqref{wtopLwr} is finite.
By redefining the set \(\{a_1, \dotsc, a_k\}\), we can thus assume that for every \(i \in \{1, \dotsc, k\}\), the map \(\Hat{\gamma}_i\) is not homotopic to a constant map.
\medbreak
\noindent\emph{Step 2.
Construction of the singularities \(\gamma_i\).} For almost every \(r \in (0,\Bar\rho (a_1,\dotsc,a_k))\), we have
\begin{equation}
\label{secondRenormLower}
\sum_{i = 1}^k\Big(\int_{\partial B_r (a_i)}
\frac{\abs{\Deriv u}^2}{2} -\frac{\equivnorm{\Hat{\gamma}_i}^2}{4\pi r}
\Big)
\ge 
\frac 1r \sum_{i = 1}^k\Big(\int_{\Sset^1}\frac{\abs{u(a_i+r\,\cdot)'}^2}{2}-\frac{\equivnorm{\Hat{\gamma}_i}^2}{4\pi }\Big)\ge 0.
\end{equation}
By \eqref{firstLowerRenorm}, the left-hand side of \eqref{secondRenormLower} is integrable with respect to \(r\) near \(0\). In particular, there exists a sequence \((\rho_\ell)_{\ell\in\Nset}\) in \((0,\Bar{\rho}(a_1,\dots,a_k))\) converging to \(0\) with for every \(i\in\{1,\dotsc,k\}\),
\begin{equation}
\label{wtopCvNorm}
\lim_{\ell \to \infty}
\int_{\Sset^1}\frac{\abs{u(a_i+\rho_\ell\,\cdot)'}^2}{2}
=
\frac{\equivnorm{\Hat{\gamma}_i}^2}{4\pi}.
\end{equation}
This implies that for each \(i\in\{1,\dotsc,k\}\), the sequence \((u(a_i+\rho_\ell\,\cdot))_{\ell \in \Nset}\) is bounded in \(W^{1,2}(\Sset^1,\mathcal{N})\) and  has a subsequence which converges weakly in \(W^{1,2}(\Sset^1,\Rset^\nu)\) to some map \(\gamma_i\in W^{1,2}(\Sset^1,\manifold{N})\). By Morrey's embedding theorem, the convergence is uniform.
Since the homotopy classes are closed under uniform convergence, \(\gamma_i\) is homotopic to \(\Hat{\gamma}_i\).
Moreover, we have by \eqref{wtopCvNorm}
\begin{equation}
\label{eq_omahz1ied6sheeNahpheej2c}
\int_{\Sset^1}\frac{\abs{\gamma_i'}^2}{2}
\le \frac{\equivnorm{\Hat{\gamma}_i}^2}{4\pi} = \frac{\equivnorm{\gamma_i}^2}{4\pi},
\end{equation}
and thus  \(\gamma_i\) is a minimising geodesic.
By \eqref{wtopCvNorm} and \eqref{eq_omahz1ied6sheeNahpheej2c},
the sequence \((u(a_i+\rho_\ell\,\cdot))_{\ell\in\Nset}\) converges strongly to \(\gamma_i\) in \(W^{1, 2}(\Sset^1,\manifold{N})\).

\medskip
\noindent\emph{Step 3.
Synharmonic convergence.}
We set \(T = - \log \Bar{\rho} (a_1, \dotsc, a_k)\) and for \(i \in \{1, \dotsc, k\}\), we define the function 
\(v_i : \Sset^1 \times (T, +\infty)\rightarrow \mathcal{N}\) by
\(
v_i(x,t)=u(a_i+e^{-t} x)
\) for every \((x,t)\in \Sset^1 \times (T, +\infty)\).
By \eqref{firstLowerRenorm}, we have 
\begin{equation*}
  \int_{T}^{+\infty} \Bigl( \int_{\Sset^1}\frac{\abs{\Deriv v_i}^2}{2}-\frac{\equivnorm{\gamma_{i}}^2}{4\pi}\Bigr)\dif t
 =\int_0^{\Bar\rho(a_1,\dotsc,a_k)}\biggl(\int_{\partial B_r (a_i)}
\frac{\abs{\Deriv u}^2}{2} -\frac{\equivnorm{\gamma_i}^2}{4\pi r}\biggr)\dif r
  < +\infty.
\end{equation*}
In particular, since the pseudometric \(\synhar{\cdot}{\cdot}\) satisfies the triangle inequality (\cref{proposition_synharmonic_pseudometric} \eqref{item_synhar_triangle}),  we have
\begin{equation}
\label{wtopCauchy}
\begin{split}
\limsup_{\rho \to 0} 
\synhar{u(a+\rho\,\cdot)}{\gamma_i}
&=
\limsup_{\rho \to 0}
\limsup_{\ell \to \infty}
\synhar{u(a+\rho\,\cdot)}{u(a+\rho_\ell\,\cdot)}\\
&\le
\limsup_{\rho \to 0}
\limsup_{\ell \to \infty}
\int_{-\log\rho}^{-\log\rho_\ell} \Bigl( \int_{\Sset^1}\frac{\abs{\Deriv v_i(\cdot, t)}^2}{2}-\frac{\equivnorm{\gamma_{i}}^2}{4\pi}\Bigr)\dif t = 0.
\end{split}
\end{equation}

\medskip
\noindent\emph{Step 4.
Geometric renormalised energy.}
By Fubini's theorem, we have \(\tr_{\Sset^1}u(a_i+\sigma\,\cdot)\in W^{1,2}(\Sset^1,\manifold{N})\) for almost every \(\sigma\in (0,\Bar{\rho}(a_1,\dotsc,a_k))\).
Given such a radius \(\sigma\), we take \(T > 0\) and \(u_i \in W^{1, 2} (\Sset^1 \times [0, T], \manifold{N})\) such that \(\tr_{\Sset^1\times\{0\}}u_i = \tr_{\Sset^1}u(a_i+\sigma\,\cdot)\) and \(\tr_{\Sset^1\times\{T\}}u_i  = {\gamma}_i\), we set \(\rho \defeq  e^{-T} \sigma\) and we define the map \(v \in W^{1, 2} (\Omega \setminus \bigcup_{i = 1}^k \Bar{B}_{\rho} (a_i), \manifold{N})\) by
  \[
  v (x)
  =
  \begin{cases}
    u (x) & \text{if \( x \in \Omega \setminus \bigcup_{i = 1}^k \Bar{B}_{\sigma} (a_i)\)},\\
    u_i \Bigl(\frac{x  - a_i}{\abs{x - a_i}},  \log \frac{\sigma}{\abs{x - a_i}}\Bigr)
    & \text{if \(x \in B_{\sigma} (a_i) \setminus \Bar{B}_{\rho} (a_i)\) for some \(i \in \{1, \dotsc, k\}\)}.
  \end{cases}
  \]
  We have then
  \[
  \begin{split}
    \int_{\Omega \setminus \bigcup_{i = 1}^k \Bar{B}_{\rho} (a_i)}
    \frac{\abs{\Deriv v}^2}{2} - \sum_{i =1}^k \frac{\equivnorm{\gamma_i}^2}{4 \pi} \log \frac{1}{\rho}
    = & \int_{\Omega \setminus \bigcup_{i = 1}^k \Bar{B}_{\sigma} (a_i)}
    \frac{\abs{\Deriv u}^2}{2} - \sum_{i =1}^k  \frac{\equivnorm{\gamma_i}^2}{4 \pi} \log \frac{1}{\sigma}\\
    &+ \sum_{i = 1}^k \Biggl(\int_{\Sset^1 \times [0, T]}
    \frac{\abs{\Deriv u_i}^2}{2} - T \frac{\equivnorm{\gamma_i}^2}{4 \pi}\Biggr).
  \end{split}
  \]
  It follows by \eqref{eq_def_renorm_geom} and \eqref{renormalised_energy_u} that
  \[
  \mathcal{E}_{g, \gamma_1, \dotsc, {\gamma}_k}^{\mathrm{geom}}  (a_1, \dotsc, a_k)
  \le \mathcal{E}^{\mathrm{ren}}  (u)
  +\sum_{i = 1}^k \liminf_{\sigma\to 0} \synhar{\gamma_i}{\tr_{\Sset^1} u(a_i+\sigma\,\cdot)}
  = \mathcal{E}^{\mathrm{ren}}  (u),
  \]
  since the infimum in \cref{definition_synharmonic} stays the same under a lower bound on \(T\) when either \(\beta\) or \(\gamma\) is a geodesic.
\end{proof}

\begin{remark}
The assertion \eqref{it_Joh2iexa0shoh8no5xoXuoli} in \cref{admissibleMaps} implies a weak \(L^2\) estimate on \(\bigcup_{i=1}^k B_\rho(a_i)\) for every \(\rho\in(0,\Bar\rho(a_1,\dotsc,a_k))\), Indeed, if \(\frac{\equivnorm{\gamma_i}}{\pi t}<r<\rho\), then
 \[
 \begin{split}
   \mathcal{H}^1
   (\{ x \in \partial B_r (a_i) \st \abs{\Deriv u (x)} > t\})
   &\le \mathcal{H}^1
   \biggl(\biggl\{ x \in \partial B_r (a_i) \st \Bigabs{\abs{\Deriv u (x)}- \frac{\equivnorm{\gamma_i}}{2 \pi r}} >
   \frac{t}{2} \biggr\}\biggr)\\
   & \le \frac{4}{t^2} \int_{\partial B_r (a_i)} \Bigl(\abs{\Deriv u} - \frac{\equivnorm{\gamma_i}}{2 \pi r}\Bigr)^2\dif \mathcal{H}^1.
 \end{split}
 \]
 It follows thus that for every \(t>0\), setting \(r_t\defeq\inf\{\frac{\equivnorm{\gamma_i}}{\pi t},\rho\}\),
 \[
 \begin{split}
 \mathcal{L}^2(\{ x \in B_\rho (a_i) \st \abs{\Deriv u (x)} > t\})
&\le  \mathcal{L}^2(\{ x \in B_\rho (a_i)\setminus B_{r_t} (a_i)\st \abs{\Deriv u (x)} > t\})
+ \mathcal{L}^2(B_{r_t} (a_i))
\\ 
 &\le \frac{4}{t^2} \int_0^\rho  \int_{\partial B_r (a_i)} \Bigl(\abs{\Deriv u}- \frac{\equivnorm{\gamma_i}}{2 \pi r}\Bigr)^2 \dif r 
 +\frac{\equivnorm{\gamma_i}^2}{t^2 \pi}.
 \end{split}
 \]
Hence, by the estimate \eqref{it_Joh2iexa0shoh8no5xoXuoli} in \cref{admissibleMaps}, we have the weak \(L^2\) estimate:
\begin{equation}
  \label{eq_uNuoy5iCah9eseijoochoote}
  \begin{split}
 \abs{\Deriv u}_{L^{2,\infty}(\bigcup_{i = 1}^k B_\rho (a_i),\Rset^{\nu\times 2})}^2
 =\sup_{t>0}t^2\mathcal{L}^2
  \bigl(\{ x \in {\textstyle \bigcup_{i = 1}^k B_\rho (a_i)} \st \abs{\Deriv u (x)} > t\}\bigr)\\
 \le 
 8\biggl( \mathcal{E}^{\mathrm{ren}} (u) - \int_{\Omega \setminus \bigcup_{i = 1}^k B_\rho (a_i)}  \frac{\abs{\Deriv u}^2}{2}
  + \Bigl(\log\frac 1\rho+\frac{1}{2}\Bigr)\sum_{i=1}^k\frac{\equivnorm{\gamma_i}^2}{4\pi}\biggr).
  \end{split}
\end{equation}
In particular, \eqref{it_Joh2iexa0shoh8no5xoXuoli} implies an \(L^p\) estimate for \( 1\leq p<2\) since
\[
\begin{split}
\norm{\Deriv u}_{L^p \bigl(\textstyle\bigcup_{i = 1}^k B_\rho (a_i),\Rset^{\nu\times 2}\bigr)}
&\le 
C\mathcal{L}^2\bigl(\textstyle\bigcup_{i = 1}^k B_\rho (a_i)\bigr)^{\frac 1p -\frac{1}{2}}
\abs{\Deriv u}_{L^{2,\infty}\bigl(\bigcup_{i = 1}^k B_\rho (a_i),\Rset^{\nu\times 2}\bigr)}\\
&\le 
C\rho^{\frac 2p - 1}
\abs{\Deriv u}_{L^{2,\infty}(\bigcup_{i = 1}^k B_\rho (a_i),\Rset^{\nu\times 2})}.
\end{split}
\]
\end{remark}

\subsection{Renormalisable singular harmonic map}
Renormalisable singular harmonic maps are defined as renormalisable maps which are harmonic outside their singular set.

\begin{definition}
\label{definition_rensingharmap}
Let \(\Omega\subset\Rset^2\) be an open set.
The map \(u \in W^{1, 2}_{\mathrm{ren}} (\Omega, \mathcal{N})\) is a \emph{renormalisable singular harmonic map} whenever \(u\) is harmonic on \(\Omega \setminus \{a_1, \dotsc, a_k\}\),
with \(\{(a_1, \gamma_1), \dotsc, (a_k, \gamma_k)\} = \operatorname{sing}(u)\).
\end{definition}

As previously, the charges \(\gamma_1, \dotsc, \gamma_k\) of the singularities are only defined up to synharmony.

\begin{lemma}\label{lemmaWeakequation}
Let \(\Omega\subset\Rset^2\) be an open set, let \(u \in W^{1, 2}_{\mathrm{ren}} (\Omega, \mathcal{N})\) be a renormalisable singular harmonic map, let \(\varphi \in \mathcal{C}^2_c (\Omega, \Rset^\nu)\) and let \(\{(a_1, \gamma_1), \dotsc, (a_k, \gamma_k)\}= \operatorname{sing}(u)\). 
If for every \(i \in \{1, \dotsc, k\}\) one has \(\varphi (a_i) = 0\),
then we have
\begin{equation}
\label{eq_oodooTai0phoabae7joS7aih}
  \int_{\Omega} \Deriv u \cdot \Deriv \varphi = \int_{\Omega} \varphi \cdot A (u)[\Deriv u, \Deriv u].
\end{equation}
\end{lemma}

Here and in the sequel,
\[ A(u)[Du,Du]= \sum_{m+1\leq i \leq \nu} \sum_{1\leq j, k \leq 2} A^i(u) (\partial_j u, \partial_k u) \mathbf{n}_i(u)
\]
with \( \dim \mathcal{N}=m\) and where \( (\mathbf{n}_{m+1}(u), \dotsc, \mathbf{n}_{\nu}(u))\) is a local orthonormal frame of \( T_u \manifold{N}^\perp \subset \Rset^\nu\) and
\(A^i(u)=\nabla \mathbf{n}_i (u)\) is the second fundamental form of \(\manifold{N}\) in the normal direction \(\mathbf{n}_i(u)\).

If the map \(u\) is harmonic in \(\Omega \setminus \{a_1, \dotsc, a_k\}\), then \eqref{eq_oodooTai0phoabae7joS7aih} holds for \(\varphi\) vanishing in a neighbourhood of each \(\{a_1, \dotsc, a_k\}\); the point of \cref{lemmaWeakequation} is that if \(u\) is renormalisable, then one can merely assume that \(\varphi\) vanishes on \(\{a_1, \dotsc, a_k\}\).

\begin{proof}[Proof of \cref{lemmaWeakequation}]Since the map \(u\) is harmonic on \(\Omega\setminus\{a_1, \dotsc, a_k\}\), we have \(u \in \mathcal{C}^\infty (\Omega\setminus\{a_1, \dotsc, a_k\})\) \cite{Helein_1991}, and
\[
 -\Delta u = A (u)[\Deriv u, \Deriv u]
 \qquad
 \text{in \(\Omega\setminus\{a_1, \dotsc, a_k\}\)}.
\]
For every  \(\rho >0\) such that the disks \(\Bar B_\rho(a_1), \dotsc, B_\rho (a_k)\) are disjoint and contained in \(\Omega\), we have by integration by parts
\[
  \int_{\Omega \setminus\bigcup_{i = 1}^k B_\rho (a_i)}
    \Deriv u \cdot \Deriv \varphi
 = \int_{\Omega \setminus \bigcup_{i = 1}^{k} B_\rho (a_i)}
   \varphi \cdot A (u)[\Deriv u, \Deriv u]
   - \sum_{i = 1}^k\int_{\partial B_\rho (a_i)} \varphi \partial_\nu u .
\]
Since the map \(u\) is renormalisable,
we have by \eqref{it_Joh2iexa0shoh8no5xoXuoli} in \cref{admissibleMaps} that \(\abs{\Deriv u} \in L^1(\Omega)\) and so \(\Deriv u \cdot \Deriv \varphi\in L^1(\Omega)\) as \(\Deriv \varphi\) is bounded. Hence,
\[
\lim_{\rho \to 0} \int_{\Omega \setminus \bigcup_{i = 1}^k B_\rho (a_i)}
\Deriv u \cdot \Deriv \varphi
= \int_{\Omega} \Deriv u \cdot \Deriv \varphi.
\]
This also implies that for every \(i \in \{1, \dotsc, k\}\),
\[
\liminf_{\rho \to 0} \rho \int_{\partial B_\rho (a_i)} \abs{\Deriv u}  = 0
\]
which in turn yields, as \(\varphi (a_i) = 0\) and \(\varphi\) is Lipschitz-continuous,
\[
\liminf_{\rho \to 0} \int_{\partial B_\rho (a)} \varphi \partial_\nu u = 0.
\]

Finally, using again \eqref{it_Joh2iexa0shoh8no5xoXuoli} in \cref{admissibleMaps}, we obtain \(\abs{\cdot-a_i}\abs{\Deriv u}^2\in L^1(B_\rho(a_i))\). Hence, since \(A(u)\) is bounded by compactness and smoothness of the manifold \(\manifold{N}\), and using again the fact that \(\varphi (a_i) = 0\) with \(\varphi\) is Lipschitz-continuous, we obtain that \(\varphi \cdot A (u)[\Deriv u, \Deriv u]\in L^1(\Omega)\) and thus
\[
\lim_{\rho\to 0}\int_{\Omega\setminus\bigcup_{i = 1}^k B_\rho(a_i)}\varphi \cdot A (u)[\Deriv u, \Deriv u]=
\int_\Omega \varphi \cdot A (u)[\Deriv u, \Deriv u].\qedhere
\]

\end{proof}

\begin{proposition}
If \(u \in W^{1, 2}_{\mathrm{ren}} (\Omega, \manifold{N})\) is a renormalisable singular harmonic map on an open set \(\Omega\subset\Rset^2\) and \(\operatorname{sing} (u) = \{(a_1,\gamma_1), \dotsc, (a_k,\gamma_k)\}\), then \(u \in \mathcal{C}^\infty (\Omega \setminus \{a_1, \dotsc, a_k\}, \manifold{N})\)
and for every \(i \in \{1, \dotsc, k\}\) and \(\rho\in(0,\Bar{\rho} (a_1, \dotsc, a_k))\),
\begin{equation}\label{eq:estimate_blow_up_near_sing}
\sup_{x \in B_\rho(a_i)\setminus\{a_i\}}
  \abs{x - a_i} \abs{\Deriv u (x)} < +\infty.
  \end{equation}
\end{proposition}
\begin{proof}
\resetconstant
 The regularity follows from the classical regularity theory for harmonic maps on planar domains \cite{Helein_1991}.
 
For the estimate, there exist \(\varepsilon > 0 \) and \(\Cl{cst_kuCheo4yohwoh6ig5ahni5Sh} > 0\) such that if \(u \in W^{1, 2} (B_r (x),\mathcal{N})\) is a smooth harmonic map on \(B_{r} (x)\) and \(\norm{D u}_{L^2 (B_r(x))}\leq \varepsilon\), then \(\abs{\Deriv u(x)} \le \Cr{cst_kuCheo4yohwoh6ig5ahni5Sh}/r\) \cite{Schoen_1984}*{Theorem 2.2}. 
For every \(i\in \{1,\dotsc k\}\), if \(B_r (x) \subseteq B_\rho (a_i)\setminus \{a_i\}\) we have
\[
\int_{B_r(x)} |Du|^2 \leq 2\int_{B_r(x)}\left( |Du(y)|-\frac{\lambda(\gamma_i)}{2\pi|y-a_i|} \right)^2\dif y+ 2\int_{B_r(x)} \frac{\lambda(\gamma_i)^2}{4\pi^2|y-a_i|^2} \dif y.
\]
By \cref{admissibleMaps} \eqref{it_Joh2iexa0shoh8no5xoXuoli},  we deduce that there exists \(\delta \in (0, 1)\) such that if \(i \in \{1, \dotsc, k\}\) and \(x \in \Omega \setminus \{a_1,\dotsc,a_k\}\) satisfy \(\abs{x - a_i} \le \delta\) and if \(r=\delta \abs{x - a_i}\), then we have \(B_{r} (x) \subseteq B_\rho (a_i)\setminus \{a_i\} \subseteq \Omega \setminus \{a_1, \dotsc ,a_k\}\) and \( \|Du\|_{L^2(B_r(x))} \leq \varepsilon\) and hence \( \abs{\Deriv u (x)} \le  \Cr{cst_kuCheo4yohwoh6ig5ahni5Sh}/(\delta \abs{x - a_i})\). 
\end{proof}

\begin{definition}
  \label{definition_stationary}
  Let \(\Omega\subset\Rset^2\) be an open set. A map \(u \in W^{1, 2}_{\mathrm{ren}} (\Omega, \mathcal{N})\) is a \emph{stationary renormalisable singular harmonic map}
whenever for every map \(\psi \in \mathcal{C}^2_c (\Omega, \Rset^2)\) which is constant in a neighbourhood of \(a\) for each \( (a, \gamma) \in \operatorname{sing} (u)\), we have
\[
 \int_{\Omega} \Deriv u {\,:\,} (\Deriv u \Deriv \psi) = \frac{1}{2} \int_{\Omega} \abs{\Deriv u}^2 \operatorname{div} \psi.
\]
\end{definition}

In view of the estimate \eqref{it_Joh2iexa0shoh8no5xoXuoli} in \cref{admissibleMaps}, namely, the fact that \(\abs{\cdot-a}\abs{\Deriv u}^2\) is integrable near \(a\), it is equivalent in \cref{definition_stationary} to assume the condition to hold for any \(\psi \in \mathcal{C}^2_c (\Omega, \Rset^2)\) such that
\(\Deriv \psi (a) = 0\) for every \((a,\gamma)\in \operatorname{sing}(u)\).

The following provides a criterion for a renormalisable singular harmonic map to be stationary.

\begin{lemma}
\label{lemma_Stress_energy_tensor}
  Let \(u \in W^{1, 2}_{\mathrm{ren}} (\Omega, \mathcal{N})\) be a renormalisable singular harmonic map defined on an open set \(\Omega\subset\Rset^2\) and define the stress-energy tensor by
  \begin{equation}
  \label{eq_Yo4ooJeeS4Lapaetao9Iawih}
  T\defeq \frac{\abs{\Deriv u}^2}{2}  - \Deriv u^*\, \Deriv u
  =\Big(\frac{\abs{\Deriv u}^2}{2}\delta_{ij}-\partial_i u\cdot\partial_ju\Big)_{i,j\in\{1,2\}}.
  \end{equation}
  and let \(\operatorname{sing} (u) = \{(a_1, \gamma_1), \dotsc, (a_k, \gamma_k)\}\).
  Then \(T\) is smooth in \(\Omega \setminus \{a_1, \dotsc, a_k\}\)
  and
  \[
  \operatorname{div} T = 0
  \qquad \text{in \(\Omega \setminus \{a_1, \dotsc, a_k\}\)}.
  \]
  Moreover, if \(0 < \rho < \Bar{\rho} (a_1, \dotsc, a_k)\) and if \(\psi \in \mathcal{C}^2_c (\Omega, \Rset^2)\) is constant in \(B_\rho (a_i)\) for each \( (a,\gamma) \in \operatorname{sing} (u)\), then
  \[
  \int_{\Omega\setminus\bigcup_{i=1}^k B_{\rho} (a_i)} \Deriv u {\,:\,} (\Deriv u\, \Deriv \psi) - \frac{\abs{\Deriv u}^2}{2}   \operatorname{div} \psi
  = \int_{\partial B_\rho (a_i)} \psi \cdot T \cdot \nu\dif \mathcal{H}^1.
  \]
  In particular, the map \(u\) is a stationary renormalisable singular harmonic map if and only if the flux of \(T\) around any small circle containing a singularity vanishes.
\end{lemma}
\begin{proof}
  The fact that the stress-energy tensor \(T\) is divergence-free in \(\Omega \setminus \{a_1, \dotsc, a_k\}\) follows from the fact that \(u\) is harmonic and thus smooth in \(\Omega \setminus \{a_1, \dotsc, a_k\}\). Hence, we can see that \((\operatorname{div} T)_i= -\partial_i u\cdot \Delta u=0\) for \(i\in \{1,2\}\), because the Laplacian of \(u\) is orthogonal to the tangent space of \(\mathcal{N}\) at \(u\). The last identity follows by integration by parts.
\end{proof}
\cref{lemma_Stress_energy_tensor} can be reinterpreted in terms of the \emph{Hopf differential} defined by \( \omega_u(z)\defeq|\partial_xu|^2-|\partial_yu|^2-2i\partial_xu \cdot \partial_yu\). The stress-energy tensor \(T\) is divergence-free in an open set if and only if the Hopf differential is holomorphic in this open set. Furthermore \(u\) is a stationary singular harmonic map if and only if the residues of the Hopf differential in the points \(  (a_1,\dotsc,a_k)\) vanish, see \cite{Bethuel_Brezis_Helein_1994}*{Theorem 8.4}.
\begin{definition}
Let \(\Omega\subset\Rset^2\) be an open set. A map \(u \in W^{1, 2}_{\mathrm{ren}} (\Omega, \manifold{N})\) such that
\(\operatorname{sing} (u) = \{(a_1,\gamma_1), \dotsc, (a_k, \gamma_k)\}\) is a \emph{minimal renormalisable singular harmonic map} if for every map \(v \in W^{1, 2}_{\mathrm{ren}} (\Omega, \manifold{N})\) such that \(\operatorname{sing} (v) = \{(b_1, \gamma_1), \dotsc, (b_k,\gamma_k)\}\), one has
\[
 \mathcal{E}_{\mathrm{ren}} (u) \le \mathcal{E}_{\mathrm{ren}} (v).
\]
\end{definition}

Since \(\operatorname{sing} (u)\) and \(\operatorname{sing} (v)\) are only defined up to synharmony, \(\operatorname{sing} (v) = \{(b_1, \gamma_1), \dotsc, (b_k,\gamma_k)\}\) has to be understood in the sense that one can find \(\gamma_1, \dotsc, \gamma_k\) satisfying the definition of both \(\operatorname{sing} (u)\) and \(\operatorname{sing} (v)\). In other words, a minimal renormalisable singular harmonic map is minimal among all renormalisable maps with fixed synharmony classes of charges \(\gamma_i\).

By \cref{admissibleMaps}, this will be the case if \(\mathcal{E}_{\mathrm{ren}} (u) = \mathcal{E}^{\mathrm{geom}}_{g, \gamma_1, \dotsc, \gamma_k}(a_1,\dotsc,a_k)\) and if for every \((b_1,\dots,b_k)\in\Conf{k}\Omega\), one has \(\mathcal{E}^{\mathrm{geom}}_{g, \gamma_1, \dotsc, \gamma_k}(a_1,\dotsc,a_k)
\le \mathcal{E}^{\mathrm{geom}}_{g, \gamma_1, \dotsc, \gamma_k}(b_1, \dotsc, b_k)\).

\begin{proposition}\label{proposition_renormalisable_stationary}
If \(u \in W^{1, 2}_{\mathrm{ren}} (\Omega, \manifold{N})\) is a minimal renormalisable singular harmonic map on an open set \(\Omega\subset\Rset^2\),
then \(u\) is a stationary renormalisable singular harmonic map.
\end{proposition}

\Cref{proposition_renormalisable_stationary} follows from the next proposition and the fact that a minimal renormalisable singular harmonic map is in particular a minimising harmonic map away from its singularities:

\begin{proposition}
\label{proposition_perturbation_stationary_identity}
If \(u \in W^{1, 2}_{\mathrm{ren}} (\Omega, \manifold{N})\) is a renormalisable singular harmonic map on an open set \(\Omega\subset\Rset^2\), which is minimising away from the singularities \(a_1,\dots,a_k\in\Omega\), where \(\operatorname{sing} (u) = \{(a_1, \gamma_1), \dotsc, (a_k,\gamma_k)\}\), then for every \(\psi \in \mathcal{C}^2_c (\Omega,\Rset^2)\) which is constant on the disks \(B_\rho (a_i)\) for some \(\rho>0\), one has
\[
 \lim_{t \to 0}  \frac{\mathcal{E}_{\mathrm{ren}} (u \compose \Psi_t) - \mathcal{E}_{\mathrm{ren}} (u)}{t}
 =   \int_{\Omega\setminus\bigcup_{i=1}^k B_\rho(a_i)} \Deriv u {\,:\,} (\Deriv u \, \Deriv \psi) - \frac{\abs{\Deriv u}^2}{2}   \operatorname{div} \psi,
\]
where \(\Psi_t : \Omega \to \Omega\) is defined for \(t \in \Rset\) when \(\abs{t}\) is small enough and \(x \in \Omega\) by \(\Psi_t (x) \defeq x +t \psi (x)\).
\end{proposition}

\begin{proof}
First, we observe that as a minimising harmonic map, \(u\) is smooth outside these singularities and \(u\) is a renormalisable singular harmonic map. We have \(\Psi_t (x) = x + t \psi (a_i)\) for \(x \in B_\rho (a_i)\) and \(\Psi_t\) is a diffeomorphism onto \(\Omega\) provided \(\abs{t}\) is small enough. If follows then that
\[
\begin{split}
  \mathcal{E}_{\mathrm{ren}} (u \compose \Psi_t) - \mathcal{E}_{\mathrm{ren}} (u)
  &= \int_{\Omega \setminus \bigcup_{i = 1}^k \psi_t^{-1}(B_\rho (a_i))} \frac{\abs{\Deriv  (u \compose \Psi_t)}^2}{2}
- \int_{\Omega \setminus \bigcup_{i = 1}^k B_\rho (a_i)} \frac{\abs{\Deriv u}^2}{2}\\
&=\int_{\Omega \setminus \bigcup_{i = 1}^k B_\rho (a_i)}
\frac{\abs{(\Deriv u)((D \Psi_t) \compose \Psi_t^{-1})}^2}{2} \det (D \Psi_t^{-1}) - \frac{\abs{\Deriv u}^2}{2}.
\end{split}
\]
The conclusion follows then by computing the derivative with respect to \(t\) at \(0\).
\end{proof}

\section{Renormalisable harmonic maps with prescribed singularities}

\subsection{Topologically prescribed singularities}

The topological renormalised energy \(\mathcal{E}_{g, \gamma_1, \dotsc, \gamma_k}^{\mathrm{top}}\) is defined as a supremum of infima over classes of maps.
For given points \(a_1, \dotsc, a_k \in \Omega\), it can be computed in terms of a limit of energies of a minimal harmonic map with prescribed singularities.

\begin{proposition}
\label{proposition_renormalised_singular}
Let \(\Omega\subset\Rset^2\) be a Lipschitz bounded domain, \(g \in W^{1/2,2}(\partial \Omega,\manifold{N})\), \(k\in\Nset_\ast\), \((a_1, \dotsc, a_k)\in\Conf{k}\Omega\) and \((\gamma_1, \dotsc, \gamma_k)\in \mathcal{C}^1 (\Sset^1, \manifold{N})^k\) be a topological resolution of \(g\) such that \(\gamma_i\) is not homotopic to constant map for every \(i\in\{1,\dotsc,k\}\).
Then, there exists a map \(u \in W^{1, 2}_{\mathrm{ren}} (\Omega, \manifold{N})\) such that
\begin{enumerate}[(i)]
\item \(\tr_{\partial \Omega}u = g\),
\item \(\operatorname{sing}(u)=\{(a_1,\Tilde{\gamma}_1), \dotsc,(a_k,\Tilde{\gamma}_k)\}\) for some  minimising geodesics \(\Tilde{\gamma}_i\) homotopic to \(\gamma_i\),
 \item
for every \(\rho > 0\), \(u\) is an \(\manifold{N}\)-valued minimising harmonic map inside \(\Omega \setminus \bigcup_{i = 1}^k \Bar{B}_{\rho} (a_i)\) with respect to its own boundary conditions,
\item  \( \mathcal{E}^{\mathrm{ren}}(u) = \mathcal{E}_{g, \gamma_1, \dotsc, \gamma_k}^{\mathrm{top}}  (a_1, \dotsc, a_k)\).
\end{enumerate}
\end{proposition}

For the classical Ginzburg--Landau problem with \(\manifold{N} = \Sset^1\) and \(\Omega\) simply-connected, the map \(u\), known as the \emph{canonical harmonic map}, is unique \cite{Bethuel_Brezis_Helein_1994}*{Corollary I.1}. In this case, the energy of \(u\) outside small disks was already known to give asymptotically the renormalised energy \cite{Bethuel_Brezis_Helein_1994}*{Theorem I.8} and \(u\) is the limit of harmonic maps outside small disks \cite{Bethuel_Brezis_Helein_1994}*{Theorem I.6}.

\begin{proof}%
[Proof of \cref{proposition_renormalised_singular}]%
\resetconstant%
By \eqref{eq_def_renorm_top_rho}, for every \(\rho\in(0, \Bar{\rho} (a_1, \dotsc, a_k))\), there exists a map  \(u_\rho \in W^{1, 2} (\Omega \setminus \bigcup_{i = 1}^k \Bar{B}_\rho (a_i), \manifold{N})\) such that \(\tr_{\partial \Omega}u_\rho  = g\) and for every \(i \in \{1, \dotsc, k\}\) the maps \(\tr_{\Sset^1} u_\rho (a_i + \rho\,\cdot)\) and \(\gamma_i\) are homotopic in \(\VMO (\Sset^1, \manifold{N})\) and
\[
\int_{\Omega \setminus \bigcup_{i = 1}^k \Bar{B}_\rho (a_i)} \frac{\abs{\Deriv u_\rho}^2}{2}
  = \mathcal{E}_{g, \gamma_1, \dotsc, \gamma_k}^{\mathrm{top},\rho} (a_1, \dotsc, a_k).
\]
By \cref{lemma_energy_growth_map_radii}, if \(\rho < \sigma < \Bar{\rho} (a_1, \dotsc, a_k)\), we have
\[
\int_{\Omega \setminus \bigcup_{i = 1}^k \Bar{B}_\sigma (a_i)} \frac{\abs{\Deriv u_\rho}^2 }{2}
  \le \mathcal{E}_{g, \gamma_1, \dotsc, \gamma_k}^{\mathrm{top},\rho} (a_1, \dotsc, a_k)
   - \sum_{i = 1}^k \frac{\equivnorm{\gamma_i}^2}{4 \pi}  \log \frac{\sigma}{\rho},
\]
and thus, by definition \eqref{eq_def_renorm_top} of the topological renormalised energy \(\mathcal{E}^{\mathrm{top}}_{g, \gamma_1, \dotsc, \gamma_k} (a_1, \dotsc, a_k)\), we have
\begin{equation}
\label{ineq_u_rho_W}
\limsup_{\rho \to 0}  \int_{\Omega \setminus \bigcup_{i = 1}^k \Bar{B}_\sigma (a_i)} \frac{\abs{\Deriv u_\rho}^2 }{2}
 - \sum_{i = 1}^k \frac{\equivnorm{\gamma_i}^2}{4 \pi}  \log \frac{1}{\sigma}
 \le \mathcal{E}^{\mathrm{top}}_{g, \gamma_1, \dotsc, \gamma_k} (a_1, \dotsc, a_k).
\end{equation}
By the boundedness condition \eqref{ineq_u_rho_W} and by a diagonal argument, there exists a sequence \((\rho_n)_{n \in \Nset}\to 0\) and a map \(u : \Omega \setminus \{a_1, \dotsc, a_k\} \to \manifold{N}\) such that for every \(\sigma \in (0, \Bar{\rho} (a_1, \dotsc, a_k))\), the sequence \((u_{\rho_n}\vert_{\Omega \setminus \bigcup_{i = 1}^k \Bar{B}_\sigma (a_i)})_{n \in \Nset}\) converges weakly to \(u \vert_{\Omega \setminus \bigcup_{i = 1}^k \Bar{B}_\sigma (a_i)}\) in \(W^{1, 2} (\Omega \setminus \bigcup_{i = 1}^k \Bar{B}_\sigma (a_i), \manifold{N})\).
By weak lower semi-continuity of the Dirichlet integral, we deduce from \eqref{ineq_u_rho_W} that for every \(\sigma \in (0, \Bar{\rho} (a_1, \dotsc, a_k))\),
\begin{equation}
\label{eq_u_star_renormalised_upper}
\int_{\Omega \setminus \bigcup_{i = 1}^k \Bar{B}_\sigma (a_i)} \frac{\abs{\Deriv u}^2}{2}
 - \sum_{i = 1}^k \frac{\equivnorm{\gamma_i}^2}{4 \pi}  \log \frac{1}{\sigma}
 \le \mathcal{E}_{g, \gamma_1, \dotsc, \gamma_k}^{\mathrm{top}} (a_1, \dotsc, a_k).
\end{equation}
By \eqref{ineq_u_rho_W} and Fatou's lemma we have for every \(i \in \{1, \dotsc, k\}\), if \(0 < \sigma < \tau < \Bar{\rho} (a_1, \dotsc, a_k)\),
\[
 \int_\sigma^\tau
 \biggl(\liminf_{n \to \infty} \int_{\partial B_r (a_i)} \frac{\abs{\Deriv u_{\rho_n}}^2}{2}\biggr) \dif r
 \le \mathcal{E}_{g, \gamma_1, \dotsc, \gamma_k}^{\mathrm{top}} (a_1, \dotsc, a_k) + \sum_{i = 1}^k \frac{\equivnorm{\gamma_i}^2}{4 \pi}  \log \frac{1}{\sigma},
\]
and thus for almost every \(r \in (0,\Bar{\rho} (a_1, \dotsc, a_k))\),
\[
  \liminf_{n \to \infty} \int_{\partial B_r (a_i)} \abs{\Deriv u_{\rho_n}}^2  < + \infty.
\]
By Morrey's inequality and the Ascoli criterion for compactness, there exists a sequence of integers \((n_\ell^r)_{\ell \in \Nset}\) divenging to \(\infty\) such that \((\tr_{\Sset^1}u_{\rho_{n_\ell^r}} (a_i + r\,\cdot))_{\ell \in \Nset}\) converges in \(L^\infty (\Sset^1, \manifold{N})\) to \(\tr_{\Sset^1}u(a_i + r\,\cdot)\).
Therefore, the maps \(\tr_{\Sset^1}u(a_i + r\,\cdot)\) and \(\gamma_i\) are homotopic in \(\VMO (\Sset^1, \manifold{N})\).
By \eqref{eq_u_star_renormalised_upper} and by definition of renormalised maps and their remormalised energies (see \cref{def_renormalisable}), we have that \(u \in W^{1, 2}_{\mathrm{ren}} (\Omega, \manifold{N})\) and \( \mathcal{E}^\mathrm{ren} (u) \le \mathcal{E}^\mathrm{top}_{g, \gamma_1, \dotsc, \gamma_k} (a_1, \dotsc, a_k)
\). By \cref{admissibleMaps}, we have \(\operatorname{sing}(u)=\{(a_1,\Tilde{\gamma}_1), \dotsc,(a_k,\Tilde{\gamma}_k)\}\); \eqref{synhamonicConvergenceSing} in \cref{admissibleMaps} together with \eqref{item_synhar_finite} in \cref{proposition_synharmonic_pseudometric} ensures that
for each \(i\), the map \(\Tilde{\gamma}_i\) is homotopic to \(\gamma_i\) and \eqref{it_bohp7shaephei0eeT} in \cref{admissibleMaps} yields
\[
 \mathcal{E}^\mathrm{ren} (u) = \mathcal{E}^\mathrm{top}_{g, \gamma_1, \dotsc, \gamma_k} (a_1, \dotsc, a_k).
\]
Since \(u \vert_{\Omega \setminus \bigcup_{i = 1}^k \Bar{B}_{\sigma} (a_i)}\) is a weak limit of a sequence of minimising harmonic maps, it is a minimising harmonic map \cite{Luckhaus_1993}.
\end{proof}

\subsection{Geometrically prescribed singularities}
We now state the analogous of \cref{proposition_renormalised_singular} for the geometrical renormalised energy \(\mathcal{E}_{g, \gamma_1, \dotsc, \gamma_k}^{\mathrm{geom}}\).

\begin{proposition}
\label{proposition_geometric_renormalised_singular}
Let \(\Omega\subset\Rset^2\) be a Lipschitz bounded domain, \(g \in W^{1/2,2}(\partial \Omega,\manifold{N})\), \(k\in\Nset_\ast\), \((a_1, \dotsc, a_k)\in \Conf{k}\Omega\) and \((\gamma_1, \dotsc, \gamma_k) \in \mathcal{C}^1 (\Sset^1, \manifold{N})^k\) be a minimal topological resolution of \(g\) such that each \(\gamma_i\) is a minimising geodesic. Then, there exists a map \(u \in W^{1, 2}_{\mathrm{ren}} (\Omega, \manifold{N})\) such that
\begin{enumerate}[(i)]
\item \(\tr_{\partial \Omega} u = g\),
\item
\(\operatorname{sing}(u)=\{(a_1, \Tilde{\gamma}_1), \dotsc, (a_k, \Tilde{\gamma}_k)\}\), for some minimising geodesics \(\Tilde{\gamma}_i\) homotopic to \(\gamma_i\),

 \item
for every \(\rho > 0\), \(u\) is a \(\manifold{N}\)-valued minimising harmonic map inside \(\Omega \setminus \bigcup_{i = 1}^k \Bar{B}_{\rho} (a_i)\) with respect to its own boundary conditions,
\item \label{synharmonicGeometricEstimate} \(\mathcal{E}_{g, \gamma_1, \dotsc, \gamma_k}^{\mathrm{geom}}  (a_1, \dotsc, a_k)= \mathcal{E}^\mathrm{ren}(u)+\sum_{i=1}^k \synhar{\gamma_i}{\Tilde{\gamma}_i}\).
\end{enumerate}
\end{proposition}

\begin{proof}%
[Proof of \cref{proposition_geometric_renormalised_singular}]%
\resetconstant%
For every \(\rho \in (0, \Bar{\rho} (a_1, \dotsc, a_k))\), by definition of \(\mathcal{E}_{g, \gamma_1, \dotsc, \gamma_k}^{\mathrm{geom},\rho}\), there exists a map  \(u_\rho\in W^{1, 2} (\Omega \setminus \bigcup_{i = 1}^k \Bar{B}_\rho (a_i), \manifold{N})\) such that \(\tr_{\partial \Omega}u_\rho  = g\), \(\tr_{\Sset^1}u_\rho (a_i + \rho\,\cdot)=\gamma_i\) for every \(i \in \{1, \dotsc, k\}\), and
\[
\int_{\Omega \setminus \bigcup_{i = 1}^k \Bar{B}_\rho (a_i)} \frac{\abs{\Deriv u_\rho}^2}{2}
  = \mathcal{E}_{g, \gamma_1, \dotsc, \gamma_k}^{\mathrm{geom},\rho} (a_1, \dotsc, a_k).
\]
We continue as in the proof of \cref{proposition_renormalised_singular} and we obtain in view of \cref{admissibleMaps} a map \(u \in W^{1, 2}_\mathrm{ren} (\Omega, \manifold{N})\), which is the weak limit in \(W^{1,2}\) of \((u_{\rho_n})_{n\in\Nset}\) away from the singularities with \((\rho_n)_{n\in\Nset}\to 0\), and minimising geodesics \(\Tilde{\gamma}_1,\dotsc,\Tilde{\gamma}_k\)
such that \(\operatorname{sing} (u) = \{(a_1, \Tilde{\gamma}_1), \dotsc, (a_k, \Tilde{\gamma}_k)\}\), where for every \(i \in \{1, \dotsc, k\}\), the map \(\gamma_i\) is homotopic to \(\Tilde{\gamma}_i\), and \(u\) is a minimising harmonic map away from the singularities.

 Moreover, there exists a sequence \((\sigma_\ell)_{\ell \in \Nset}\) converging to \(0\)
such that for every \(\ell \in \Nset\) and every \(i \in \{1, \dotsc, k\}\),
the sequence \((\tr_{\Sset^1} u(a_i + \sigma_\ell\, \cdot))_{\ell \in \Nset}\) converges strongly to 
\(\Tilde{\gamma}_i\) in \(W^{1,2}(\Sset^1,\manifold{N})\).

Now, for every \(n,\ell\in\Nset\) such that \(\rho_n<\sigma_\ell<\Bar{\rho}(a_1,\dotsc,a_k)\), we observe that by definition of the synharmony (\cref{definition_synharmonic}) and by a change of variable, we have for every \(i\in\{1,\dotsc,k\}\),
\begin{equation}
\label{synharmonicity_est}
\int_{B_{\sigma_\ell}(a_i)\setminus B_{\rho_n}(a_i)}\frac{\abs{\Deriv u_{\rho_n}}^2}{2}
- \frac{\equivnorm{\gamma_i}^2}{4\pi} \log\frac{\sigma_\ell}{\rho_n}
\ge \synhar{\tr_{\Sset^1}u_{\rho_n}(a_i+\sigma_\ell\,\cdot)}{\gamma_i}.
\end{equation}
Hence, by \eqref{synharmonicity_est}, by weak lower semi-continuity of the Dirichlet integral and by \cref{proposition_synharmonic_pseudometric},
\begin{multline*}
 \mathcal{E}^{\mathrm{geom}}_{g, \gamma_1, \dotsc, \gamma_k} (a_1, \dotsc, a_k)
 = \lim_{n\to\infty}
 \int_{\Omega \setminus \bigcup_{i = 1}^k \Bar{B}_{\rho_n} (a_i)} \frac{\abs{\Deriv u_{\rho_n}}^2}{2}-\sum_{i = 1}^k \frac{\equivnorm{\gamma_i}^2}{4 \pi}  \log \frac{1}{\rho_n}\\
 \ge \lim_{n\to\infty}
\int_{\Omega \setminus \bigcup_{i = 1}^k \Bar{B}_{\sigma_\ell} (a_i)} \frac{\abs{\Deriv u_{\rho_n}}^2}{2} -\sum_{i = 1}^k \frac{\equivnorm{\gamma_i}^2}{4 \pi}  \log \frac{1}{\sigma_\ell}
+\sum_{i=1}^k \synhar{\tr_{\Sset^1} u_{\rho_n}(a_i+\sigma_\ell\,\cdot)}{\gamma_i}\\
\ge \int_{\Omega \setminus \bigcup_{i = 1}^k \Bar{B}_{\sigma_\ell} (a_i)} \frac{\abs{\Deriv u}^2}{2}
-\sum_{i = 1}^k \frac{\equivnorm{\gamma_i}^2}{4 \pi}  \log \frac{1}{\sigma_\ell}
+\sum_{i=1}^k \synhar{\tr_{\Sset^1} u(a_i+\sigma_\ell\,\cdot)}{\gamma_i}.
\end{multline*}
By definition of \(\mathcal{E}^\mathrm{ren}\) (\cref{def_renormalisable}) and by \cref{proposition_synharmonic_pseudometric}, we obtain in the limit 
\(\ell\to \infty\),
\[
 \mathcal{E}^{\mathrm{geom}}_{g, \gamma_1, \dotsc, \gamma_k} (a_1, \dotsc, a_k)
 \ge\mathcal{E}^\mathrm{ren}(u)+\sum_{i=1}^k \synhar{\Tilde{\gamma}_i}{\gamma_i}.
\]
For the reverse inequality, we have by \cref{admissibleMaps} \eqref{it_bohp7shaephei0eeT} and \cref{proposition_renorm_geom_dep_gamma} 
\begin{equation*}
\mathcal{E}^\mathrm{ren}(u)
\ge\mathcal{E}^{\mathrm{geom}}_{g, \Tilde\gamma_1, \dotsc, \Tilde\gamma_k} (a_1, \dotsc, a_k)
\ge \mathcal{E}^{\mathrm{geom}}_{g, \gamma_1, \dotsc, \gamma_k} (a_1, \dotsc, a_k) - \sum_{i=1}^k \synhar{\Tilde{\gamma}_i}{\gamma_i}.
\qedhere
\end{equation*}
\end{proof}

\subsection{Relationship between renormalised energies}

By the definitions \eqref{eq_def_renorm_top_rho} and \eqref{eq_def_renorm_geom_rho} we have immediately the lower bound on the geometric energy \eqref{eq_comparison_top_geom}.
Under the condition that \(\gamma_1, \dotsc, \gamma_k\) are minimising geodesics which are synharmonic to all geodesics that are homotopic to it, we prove that the geometrical and topological renormalised energies coincide.

\begin{proposition}
\label{proposition_renormalised_prescribed}
Let \(\Omega\subset\Rset^2\) be a Lipschitz bounded domain, \(g \in  W^{1/2,2}(\partial \Omega, \manifold{N}) \), \( k\in \Nset_\ast\), \((a_1,\dotsc,a_k) \in \Conf{k} \Omega\) and \( (\gamma_1,\dotsc,\gamma_k) \in \mathcal{C}^1(\mathbb{S}^1,\manifold{N})^k\) be a topological resolution of \(g\). Assume that for each \(i\), \(\gamma_i\) is a non-trivial minimising geodesic. 
Then, there exists \(u \in W^{1,2}_{\mathrm{ren}}(\Omega,\manifold{N}) \cap \mathcal{C}^\infty(\Omega \setminus \{a_1,\dotsc,a_k \},\manifold{N})\) and minimising geodesics \(\Tilde\gamma_i\) homotopic to \(\gamma_i\) such that 
\begin{align*}
&\operatorname{sing} (u) = \{(a_1, \Tilde{\gamma}_1), \dotsc, (a_k, \Tilde{\gamma}_k)\},\\
&\mathcal{E}^{\mathrm{ren}}(u) =  \mathcal{E}^{\mathrm{top}}_{g,\gamma_1,\dotsc,\gamma_k}(a_1,\dotsc,a_k)=\mathcal{E}^{\mathrm{geom}}_{g,\tilde{\gamma}_1,\dotsc,\tilde{\gamma}_k}(a_1,\dotsc,a_k).
\end{align*}
\end{proposition}

In view of \eqref{eq_comparison_top_geom}, \cref{proposition_renormalised_prescribed} means that
\begin{multline}\label{linkRenormalisedEnergies}
\mathcal{E}_{g, \gamma_1, \dotsc, \gamma_k}^{\mathrm{top}} (a_1, \dotsc, a_k)
= \inf\, \bigl\{\mathcal{E}_{g, \Tilde{\gamma}_1, \dotsc, \Tilde{\gamma}_k}^{\mathrm{geom}} (a_1, \dotsc, a_k)
\st \text{for each \(i \in \{1, \dotsc, k\}\),}\\
\text{\(\Tilde{\gamma}_i\) is a minimising geodesic homotopic to \(\gamma_i\)}\bigr\}.
\end{multline}
\begin{proof}[Proof of \cref{proposition_renormalised_prescribed}]
  We let \(u \in W^{1,2}_{\mathrm{ren}}(\Omega,\manifold{N}) \cap \mathcal{C}^\infty(\Omega \setminus \{a_1,\dotsc,a_k \},\manifold{N})\) be a singular minimising harmonic map given by \cref{proposition_renormalised_singular}, so that in particular,
  \[
  \mathcal{E}^{\mathrm{ren}}(u) =  \mathcal{E}^{\mathrm{top}}_{g,\gamma_1,\dotsc,\gamma_k}(a_1,\dotsc,a_k).
  \]
Let  \(\operatorname{sing} (u) = \{(a_1, \Tilde{\gamma}_1), \dotsc, (a_k, \Tilde{\gamma}_k)\}\). By construction, \(\Tilde\gamma_i\) and \(\gamma_i\) are homotopic. Hence, by \cref{admissibleMaps} and \eqref{eq_comparison_top_geom}, we have 
  \[
   \mathcal{E}^{\mathrm{ren}}(u) 
   \ge \mathcal{E}^{\mathrm{geom}}_{g,\tilde{\gamma}_1,\dotsc,\tilde{\gamma}_k}(a_1,\dotsc,a_k)
   \ge \mathcal{E}^{\mathrm{top}}_{g,\Tilde\gamma_1,\dotsc,\Tilde\gamma_k}(a_1,\dotsc,a_k)
   = \mathcal{E}^{\mathrm{top}}_{g,\gamma_1,\dotsc,\gamma_k}(a_1,\dotsc,a_k),
  \]
  and the conclusion follows.
\end{proof}

In particular, if for every \(i \in \{1, \dotsc, k\}\), all minimising geodesics that are homotopic to \(\gamma_i\) are synharmonic to \(\gamma_i\), then by \cref{proposition_renormalised_prescribed} and \cref{proposition_renorm_geom_dep_gamma}, \(\mathcal{E}_{g, \gamma_1, \dotsc, \gamma_k}^{\mathrm{geom}}  (a_1, \dotsc, a_k)
  = \mathcal{E}_{g, \gamma_1 \dotsc, \gamma_k}^{\mathrm{top}}  (a_1, \dotsc, a_k)\).
This is the case in particular for the classical Ginzburg--Landau problem for \(\manifold{N} = \Sset^1\) for which the equivalence between topological and geometric renormalised energies was proved by Bethuel, Brezis and H\'elein \cite{Bethuel_Brezis_Helein_1994}*{Theorem I.9 and Remark I.5}.

\subsection{Superdifferentiability}
We have seen in \cref{theorem_continuity_renormalised}, that the geometric renormalised energy is Lipschitz-continuous. The definition as an infimum, does not give hope for much more regularity. 
Using the expression of the geometrical renormalised energy as the renormalised energy of some renormalisable singular harmonic map \cref{proposition_geometric_renormalised_singular}, we obtain the next proposition.

\begin{proposition}
\label{proposition_superdifferentiability}
Let \(\Omega\subset\Rset^2\) be a Lipschitz bounded domain, \(g \in W^{1/2,2}(\partial\Omega,\manifold{N})\), \(k\in\Nset_\ast\) and \((\gamma_1, \dotsc, \gamma_k) \in \mathcal{C}^1 (\Sset^1, \manifold{N})^k\) be a minimal topological resolution of \(g\) such that each \(\gamma_i\) is a minimising geodesic. 

For every \((a_1, \dotsc, a_k)\in\Conf{k}\Omega\), one has
\[
 \limsup_{(b_1, \dotsc, b_k) \to (a_1, \dotsc, a_k)}
 \frac{\mathcal{E}^\mathrm{geom}_{g, \gamma_1, \dotsc, \gamma_k} (b_1, \dotsc, b_k) - \mathcal{E}^\mathrm{geom}_{g, \gamma_1, \dotsc, \gamma_k} (a_1, \dotsc, a_k) - \tau_i \cdot (b_i - a_i)}{\max_{1 \le i \le k} \abs{b_i - a_i}} \le 0,
\]
where for each \(i \in \{1, \dotsc, k\}\),
\[
\tau_i \defeq - \int_{\partial B_\rho(a_i)} T \cdot \nu,
\]
with\footnote{Note that the the flux of the stress-energy tensor through a small circle centered at the singularity \(a_i\) is independant of the (small) radius since the stress-energy tensor is divergence-free away from singularities.} \(\rho\in(0,\Bar\rho(a_1,\dotsc,a_k))\), and where \(T\) is the stress-energy tensor -- defined in \eqref{eq_Yo4ooJeeS4Lapaetao9Iawih} -- associated to some map \(u \in W^{1, 2}_{\mathrm{ren}} (\Omega, \manifold{N})\) -- given by \Cref{proposition_geometric_renormalised_singular} -- such that 
\begin{align*}
&\operatorname{sing}(u)=\{(a_1, \Tilde{\gamma}_1), \dotsc, (a_k, \Tilde{\gamma}_k)\},\\
\mathcal{E}_{g, \gamma_1, \dotsc, \gamma_k}^{\mathrm{geom}}  &(a_1, \dotsc, a_k)= \mathcal{E}^\mathrm{ren}(u)+\sum_{i=1}^k \synhar{\gamma_i}{\Tilde{\gamma}_i}.
\end{align*}
\end{proposition}

In summary, \cref{proposition_superdifferentiability} states that the derivative of the geometric renormalised energy, when it exists, is given by the flux of the stress-energy tensor of the associated renormalisable singular harmonic map. In particular, critical points of the geometric renormalised energy are characterised by points at which this flux vanishes so that, in this sense, the stress energy-tensor is divergence-free on the whole domain \(\Omega\).

In view of \cref{proposition_renormalised_prescribed}, the topological renormalised energy has the same property.

\begin{proof}[Proof of \cref{proposition_superdifferentiability}]
We choose \(\varphi \in \mathcal{C}^\infty (\Rset^2)\) such that \(\varphi = 1\) on \(B_{\rho/2}\) and \(\varphi = 0\) on \(\Rset^2\setminus B_\rho\). We then define \(\Phi_{b_1, \dotsc, b_k}: \Omega \to \Omega\) for \((b_1, \dotsc, b_k)\) close enough to \((a_1, \dotsc, a_k)\) and \(x \in \Omega\) by 
\[
\Phi_{b_1, \dotsc, b_k} (x) 
\defeq x + \sum_{i = 1}^k (b_i - a_i) \varphi (x - a_i).
\]
By \cref{proposition_renorm_geom_dep_gamma} and \cref{admissibleMaps}, we have 
\[
 \mathcal{E}_{g, \gamma_1, \dotsc, \gamma_k}^{\mathrm{geom}}  (b_1, \dotsc, b_k)
 \le \mathcal{E}^\mathrm{ren}(u \compose \Phi_{b_1, \dotsc, b_k})+\sum_{i=1}^k \synhar{\gamma_i}{\Tilde{\gamma}_i}.
\]
The conclusion then follows from a variant of \cref{proposition_perturbation_stationary_identity} (where \(t \in \Rset^{2 k}\) becomes a vector parameter and \(\psi\) is at each point a linear mapping from \(\Rset^{2k}\) to \(\Rset^2\)) 
and from \cref{lemma_Stress_energy_tensor}.
\end{proof}

\section{Computing some singular energies}
\label{section_examples}

In the examples we have in mind from condensed matter physics and computer graphics (meshing and cross-fields theory), the
manifolds can be obtained as quotients of \(SU(2)\) by some of its subgroups.
We first review briefly the common properties of such spaces. We indicate how
to compute their fundamental group and the conjugacy classes in the non-abelian case. Then we describe the closed minimising geodesics in free homotopy classes and describe the minimal topological resolution
\eqref{def_loose_equiv_norm} of some maps from \(\partial \Omega \rightarrow \manifold{N}\) when \(\partial \Omega\) is connected.

\subsection{Fundamental group and geodesics on homogeneous spaces}

Given a Lie group \(G\) and a closed subgroup \(H \subseteq G\), we consider the manifold \(G/H = \{gH \st g \in G\}\) obtained by taking the quotient of \(G\) by its subgroup \(H\). The group \(G\) acts transitively on \(G/H\) by left multiplication, making it a \emph{homogeneous space}, see for example \cite{Lee_2013}*{Theorem 21.17}. 

The fundamental group of \(G/H\) can be computed by using the following:
\begin{proposition}[Fundamental group on homogeneous spaces \cite{Mostow_1950}*{\S 8}]
\label{proposition_fund_group_homog_space}
Let \(G\) be a simply connected topological group, let \(H\) be a closed subgroup of \(G\)
and let \(H_0\) be the connected component of the identity of \(H\),
then \(H_0\) is a normal subgroup of \(H\) and \(\pi_1(G/H) \simeq H/H_0\). 
\end{proposition}
The group \(H/H_0\) is referred to as the group of connected components of \(H\) in the literature.

The restriction that \(G\) is simply connected can be overcome by replacing \(G\) by its universal covering.

When the Lie group \(G\) is endowed with a doubly invariant Riemannian metric we can endow \(G/H\) with a Riemannian structure thanks to the following:

\begin{proposition}[Riemannian structure on homogeneous spaces \cite{Cheeger_Ebin_1975}*{Propositions 3.14 and 3.16}]
  \label{proposition_riem_struct}
Assume that \(G\) is endowed with a Riemannian metric which is left invariant by the action of \(G\) and right invariant by the action of \(H\).
Then, the left coset space \(G/H\) is a Riemannian manifold
of dimension \(\dim G - \dim H\) , and has a unique Riemannian structure such that the canonical projection \(\pi: G \rightarrow G/H\) is a Riemannian submersion and \(G\) acts isometrically on \(G/H\) by left multiplication.
\end{proposition}

The fact that \(\pi\) is a Riemannian submersion is equivalent to the fact \(d\pi (e)\) is an isometry from the orthogonal \(\mathfrak{h}^\perp\) in the Lie algebra \(\mathfrak{g}\) of \(G\) of the Lie algebra \(\mathfrak{h}\) of \(H\) into \(T_{eH} G/H\), where \(e\) denotes the identity element of \(G\). We recall that the geodesics in compact Lie groups endowed with a bi-invariant metric are described by the following

\begin{proposition}\label{prop:geodesics_compact_Lie_groups}
Let \(G\) be a compact Lie group with a bi-invariant Riemannian
metric. Then the geodesics have the form \(\gamma(t) = \exp(t \xi)g\) for \(g \in G\) and
\(\xi \in \mathfrak{g}\).
\end{proposition}

The exponential is understood in the sense of Lie groups; we refer to \citelist{\cite{Helgason_2001}*{Ch. 2 \S 1.3}} for the proof of \cref{prop:geodesics_compact_Lie_groups}.


We now turn  to the classification of geodesics in order to describe their synharmony. We first have a lifting property of geodesics under the Riemannian submersion \(\pi:G\to G/H\) (see \cite{gallotRiemannianGeometry2004}*{\S 2.C.6}, or \cite{Michor_2008}*{\S 26}). We say that a map \(g \in \mathcal{C}^1 ([0, 2\pi], G)\) is horizontal whenever for every \(t \in [0, 2\pi]\), \(g'(t) \in (\ker d \pi (g (t)))^\perp\), or equivalently, \(g(t)^{-1}g'(t)\in\mathfrak{h}^\perp\). We say that \(g\) is a horizontal lift of a map \(\gamma\in\mathcal{C}^1 ([0, 2\pi],G/H)\) if \(g\) is horizontal and \(\pi\compose g=\gamma\). The map \(\gamma\) has a unique horizontal lift \(g\) such that \(g(0)=g_0\) whenever \(\pi(g_0)=\gamma(0)\). The horizontal lifting preserves the length; in particular, the map \(\gamma\) is a geodesic in \(G/H\) if and only if its horizontal lifts are geodesics in \(G\).

We also have a horizontal lifting property for homotopies:

\begin{proposition}
\label{prop_homot_geodesics_lifting}
If \(\gamma \in \mathcal{C}^1 (\Sset^1 \times [0, 1], G/H)\),
then there exist \(\sigma\in\mathcal{C}^1([0,1],G)\) and \(g \in \mathcal{C}^1 ([0, 2\pi] \times [0, 1], G)\) such that the following holds:
\begin{enumerate}[(i)]
 \item for every \((s,t) \in [0, 2\pi]\times [0, 1]\), one has
 \(\gamma (s, t) = \pi (\sigma(t)g(s,t))\) under the identification \(\Sset^1 \simeq [0, 2\pi]/\{0, 2 \pi\}\),
 \item for every  \(t \in [0, 1]\), the map \(g(\cdot, t)\) is horizontal,
 \item \(g (0, \cdot) = e\) and \(g (2 \pi, \cdot) \in H\),
 \item the map \(\sigma\) is horizontal.
\end{enumerate}
\end{proposition}

\begin{remark}
In the particular case where the group \(H\) is discrete (\(\dim H = 0\)), then \(g (2 \pi, \cdot)\) is constant in \(H\); otherwise it remains in the same connected component of \(H\).
\end{remark}

\begin{proof}[Proof of \cref{prop_homot_geodesics_lifting}]
We let \(\sigma\in\mathcal{C}^1([0,1],G)\) be a horizontal map such that \(\gamma(0,\cdot)=\pi\compose\sigma\), and, for every \(t \in [0, 1]\), we let \(g (\cdot, t) \in \mathcal{C}^1 ([0, 2\pi], G)\) be the (unique) horizontal map such that \(\pi \compose g (\cdot, t) = \sigma(t)^{-1}\gamma (\cdot, t)\) and \(g(0,t)=e\). This way, for every \(t\in[0,1]\), we have \(\gamma(\cdot,t)=\pi(\sigma(t)g(\cdot,t))\). Since \(\gamma(0,\cdot)=\gamma(2\pi,\cdot)\), we have \(g(2\pi,t)\in\pi^{-1}(\{0\})=H\). Moreover, by uniqueness of the horizontal lifting, \(g \in \mathcal{C}^1 ([0, 2\pi]\times[0,1],G)\).
\end{proof}

By \Cref{prop_homot_geodesics_lifting}, since the horizontal lifting preserves the length, classifying homotopy classes of minimising closed geodesics in \(G/H\) reduces to the question of classifying homotopy classes of geodesics in \(G\) that minimise the distance between the identity \(e\) and a given connected component of \(H\). Our main tool to study this on examples will be the following proposition that classifies geodesics on \(SU(d)\), \(d \geq 2\).

\begin{proposition}
  \label{proposition_geodesicsSUd}
Let \(SU (d)\) be endowed with its bi-invariant metric.
If \(\gamma_0, \gamma_1:[0,1]\to SU(d)\) are two minimising geodesics between \(\operatorname{id}\) and some point \(g\in SU(d)\), then there exists \(\xi \in \mathfrak{su} (d)=\{ M \in M_d(\mathbb{C})\st M^*=-M \text{ and }\operatorname{tr}(M)=0\}\) such that for every \(t \in [0, 1]\), \(g = \exp (t \xi ) g \exp (-t \xi )\) and for every \(s \in [0, 1]\),
\[
 \gamma_1(s) = \exp (\xi) \gamma_0(s) \exp (-\xi).
\]
In particular, \(\gamma_0\) is homotopied to \(\gamma_1\) by the geodesics \(\gamma_t(s)=\exp (t \xi )\gamma_0(s)\exp (-t \xi )\), \(t\in[0,1]\).
\end{proposition}

\begin{proof}
Since \(\gamma_k\) is a geodesic, there exists \(\sigma_k \in \mathfrak{su} (d)\) such that for every \(s\) in \([0,1]\), \(\gamma_k (s) = \exp (s \sigma_k)\). Since \(\gamma_k\) is minimising from \(\operatorname{id}\) to \(g\), we have \(\abs{\sigma_k} = \dist_{SU (d)} (\operatorname{id}, g)\) and the spectrum of \(\sigma_k\) is contained in the interval \([-i\pi,i\pi]\) of the imaginary axis in the complex plane. Let \(V \defeq  \ker (g + \operatorname{id}) \subset \Cset^d\). Since \(g=\exp ({\sigma_k})\) and \(\operatorname{spec}(\sigma_k)\subset[-i\pi,i\pi]\), we have
\[
 V = \ker (\sigma_k + i \pi \operatorname{id})
 \oplus \ker (\sigma_k - i \pi \operatorname{id}),\quad k=0,1.
\]
Therefore, the spaces \(V\) and \(V^\perp\) are invariant subspaces of the linear operators \(g\), \(\sigma_0\) and \(\sigma_1\) on \(\Cset^d\).
Since the complex exponential is injective on the segment \((-i\pi,i\pi)\), we have \(\ker (\sigma_k - \lambda \operatorname{id}) = \ker (g - e^{\lambda} \operatorname{id})\) for every \(\lambda\in (-i\pi,i\pi)\), and therefore \(\sigma_0 \vert_{V^\perp}
= \sigma_1 \vert_{V^\perp}\).
On the other hand, since \(\tr\sigma_k=0\), we have
\[
 \tr \sigma_0 \vert_V
 = - \tr \sigma_0 \vert_{V^\perp}
 = - \tr \sigma_1 \vert_{V^\perp}
 = \tr \sigma_1 \vert_V.
\]
Since the spectra of both \(\sigma_0 \vert_V\) and \(\sigma_1 \vert_V\) are contained in \(\{-i\pi,i\pi\}\), \(\sigma_0\) and \(\sigma_1\) have thus the same eigenvalues with the same multiplicity. There exists thus \(U\in SU(d)\) such that \(U\vert_{V^\perp}=\operatorname{id}_{V^\perp}\), \(U(V)\subset V\) and \(\sigma_1 = U \sigma_0 U^\ast\). We write \(U=\exp(\xi)\) with \(\xi\in\mathfrak{su}(d)\) such that \(\xi(V)\subset V\) and \(\xi=0\) on \(V^\perp\). It follows by exponentiation that \(\gamma_1(s)=\exp(\xi)\gamma_0(s)\exp(-\xi)\) for every \(s\). Now, using that \(V\) and \(V^\perp\) are invariant subspaces of \(g\) and \(\xi\) with \(g\vert_V=-\operatorname{id}_V\) and \(\xi\vert_{V^\perp}=0\), one sees that \(g\xi=\xi g\). It follows by exponentiation that \(g \exp (t \xi )= \exp (t \xi ) g \) for every \(t\).
\end{proof}

\begin{corollary}
\label{corollary_discreteSUdQuotient}
Let \(SU (d)\) be endowed with its bi-invariant metric and \(H\) be a discrete subgroup of \(SU(d)\). Then any two homotopic closed geodesics \(\gamma_0, \gamma_1 \in \mathcal{C}^1 (\Sset^1, SU(d)/H)\) that are minimising in their homotopy class are synharmonic.
\end{corollary}
In view of \eqref{linkRenormalisedEnergies}, \cref{corollary_discreteSUdQuotient} tells us in particular that the topological and geometrical renormalised energies coincide in the case of \(\manifold{N}=SU(d)/H\).
\begin{proof}
By \cref{prop_homot_geodesics_lifting}, since \(\gamma_0\) and \(\gamma_1\) are homotopic, there exist a map \(\sigma\in\mathcal{C}^1([0,1],SU(d))\) and a homotopy \(\mu \in \mathcal{C}^1 ([0, 2\pi] \times [0, 1], SU(d))\) such that if we set \(\mu_0(\cdot)\defeq\mu(\cdot,0)\) and \(\mu_1(\cdot)\defeq\mu(\cdot,1)\), then we have \(\gamma_0(\cdot)=\pi (\sigma(0)\mu_0(\cdot))\), \(\gamma_1(\cdot)=\pi (\sigma(1)\mu_1(\cdot))\), \(\mu (0, \cdot) = \operatorname{id}\), \(\mu (2 \pi, \cdot) \in H\). Since \(H\) is discrete, there exists \(h\in H\) such that \(\mu (2 \pi, \cdot)\equiv h\). In particular, \(\mu_0(2\pi)=\mu_1(2\pi)=h\). Since the horizontal lifting preserves the length, the maps \(\mu_0\) and \(\mu_1\) are minimising geodesics in \(SU(d)\). Hence, by \cref{proposition_geodesicsSUd}, there exists \(\xi \in \mathfrak{su} (d)\) such that for every \(t \in [0, 1]\), \(h = \exp (t\xi ) h \exp (-t\xi )\) and for every \(s \in [0, 2\pi]\), \(\mu_1(s) = \exp (\xi) \mu_0(s) \exp (-\xi)\). We now set
\[
\gamma(s,t)=\pi(\sigma(t)\exp (t \xi )\mu_0(s)\exp (-t \xi )).
\]
Then, \(\gamma\in\mathcal{C}^1([0,2\pi]\times[0,1],SU(d)/H)\) is a homotopy between \(\gamma_0=\gamma(\cdot,0)\) and \(\gamma_1=\gamma(\cdot,1)\) such that \(\gamma(0,\cdot)=\gamma(2\pi,\cdot)\), i.e. it is a homotopy of loops. Moreover, the maps \((\gamma(\cdot,t))_{t\in[0,1]}\) have all the same length -- by bi-invariance of the metric -- and are thus minimising in their homotopy class since this is the case of \(\gamma_0\) and \(\gamma_1\). Hence, by \cref{proposition_homotopy_synharmonic}, \(\gamma_0\) and \(\gamma_1\) are synharmonic.
\end{proof}

\subsection{Working with \texorpdfstring{\(SO(3)\)}{SO(3)} and \texorpdfstring{\(SU(2)\)}{SU(2)}}
The group of rotations of \(\Rset^3\), denoted by
\[ SO(3)=\Bigl\{ g\in \Rset^{3 \times 3} \st g^* g=\id \text{ and } \det(g)=1 \Bigr\} \]
is a compact Lie group whose Lie algebra is given by
\[ \mathfrak{so}(3)= \Bigl\{ X \in \Rset^{3 \times 3}\st X    ^*=-X \Bigr\}.\]
We endow \( \mathfrak{so}(3)\) with the bi-invariant metric defined for \(X,Y \in \mathfrak{so}(3)\) by 
\[
 \frac{\tr (X^*Y)}{2};
\]
this choice ensures that the distance of a rotation of amplitude \(\theta \in [-\pi, \pi]\) from the identity is \(\abs{\theta}\).

The universal covering of \(SO(3)\) is the unitary group \(SU (2)\), which is a compact Lie group defined as
\begin{equation*}
\begin{split}
 SU(2)&= \Bigl\{ g \in \Cset^{2 \times 2}\st g^*g= \id \text{ and } \det(g)=1 \Bigr\} \\
&= \Bigl\{ x_0\mathbf{1} +x_1\mathbf{i}+x_2 \mathbf{j}+x_3\mathbf{k}\st x_0, \dotsc, x_3 \in \Rset \text{ and } \sum_{\ell = 0}^3 x_\ell^2 = 1\Bigr\},
\end{split}
\end{equation*} 
where we have used the matrices
\[ 
\mathbf{1}= \begin{pmatrix}
1 & 0 \\
0& 1
\end{pmatrix}, \quad 
\mathbf{i}= \begin{pmatrix}
i & 0 \\
0& -i
\end{pmatrix}, \quad \mathbf{j}=\begin{pmatrix}
0&1 \\
-1 & 0
\end{pmatrix}, \quad 
\mathbf{k}= \begin{pmatrix}
0& i \\
i& 0
\end{pmatrix}.
\]
Since \( \mathbf{i}, \mathbf{j}, \mathbf{k}\) satisfy \( \mathbf{i}^2=\mathbf{j}^2=\mathbf{k}^2=-\mathbf{1}\) and \(\mathbf{i}\mathbf{j}\mathbf{k}=-1\),
they generate in \(\Cset^{2 \times 2}\) an algebra that can be identified with the quaternions \(\Hset\); \(SU(2)\) can be identified as the unit sphere in \(\Hset\). The matrices \(\mathbf{i}\), \(\mathbf{j}\) and \(\mathbf{k}\), also known as the Pauli matrices, form a basis of the Lie algebra of \(SU (2)\), given by
\begin{equation}
T_{\id} SU(2)= \mathfrak{su}(2)= \Bigl\{ X \in \Cset^{2 \times 2}; X^*=-X \text{ and } \tr(X)=0\Bigr\}.
\end{equation}
Any \(g \in SU (2)\) defines a rotation \(p(g)\) on \(\Rset^3 \simeq \{x \in \Hset \st \operatorname{Re}x = 0\}\) defined by
\[
 p (g) x = g x g^{-1};
\]
the Lie group homomorphism \(p: SU(2)\to SO(3)\) is a double cover: \(p(g) = p (h)\) if and only if \(g = \pm h\); \(p\) is an isometry provided we endow \(SU(2)\) with the bi-invariant metric defined for \(X, Y \in  \mathfrak{su}(2) \subset \Hset\) by
\[ 
  2 \tr(X^*Y) = 4 \Re (\Bar{X} Y),
\]
that is, the distance on \(SU(2)\) is twice the distance of the points on the corresponding unit sphere \(\Sset^3 \subset \Hset\equiv\Rset^4\).

\subsection{Concrete examples}
We review several examples of vacuum manifold \(\manifold{N}\) that we consider either for their relevance in applications in physics or computer graphics or for the illustration of the wide spectrum of behaviours that can arise.

In all what follows, we will say that a finite collection of free homotopy classes, or equivalently a finite collection of conjugacy classes in \(\pi_1(\manifold{N})\), is a \emph{free decomposition} of a given free homotopy class if corresponding loops in \(\manifold{N}\) are topologically compatible. We will also say that the decomposition is minimal when the corresponding topological resolutions are minimal.

\subsubsection{The circle \(\Sset^1\)}
The case \(\manifold{N}=\mathbb{S}^1\) arises in superconductivity models
and in two-dimensional cross-fields generation.
In view of \(\Sset^1 \simeq \Rset/\Zset\), one has \(\pi_1(\mathbb{S}^1) =\Zset\).
Moreover, since the closed geodesics on \(\mathbb{S}^1\) can be written as \(t \in [0, 2\pi] \simeq \Sset^1 \mapsto e^{i \alpha}e^{i n t}\), for some \(\alpha \in \Rset\) and for some \(n \in \mathbb{Z}\) that corresponds to the degree of the geodesic, any two homotopic geodesics are automatically synharmonic.
Moreover, the systolic geodesics are the geodesics of degree \(n \in \{-1, 1\}\). We finally have for a map \(\gamma\) of degree \(n\) that \(\Esing (\gamma) = \pi \abs{n}\), with the only minimal free decomposition of a map of degree \(n\) being into \(\abs{n}\) maps of degree \(\operatorname{sgn} (n)\).

\subsubsection{Flat torus}
For the flat torus \(\Sset^1 \times \Sset^1\), we have \(\pi_1 (\Sset^1 \times \Sset^1) \simeq  \pi_1 (\Sset^1)  \times \pi_1(\Sset^1) \simeq \Zset \times \Zset\).
Each closed geodesic in \(\Sset^1 \times \Sset^1\) is a pair of closed geodesic into \(\Sset^1\) and thus all closed geodesics in a given homotopy class are synharmonic.
For every \((n, m) \in \pi_1(\Sset^1 \times \Sset^1)\), \(\equivnorm{(n, m)} = 2 \pi \sqrt{n^2 + m^2}\) while  \(\Esing{(n, m)} = (\abs{n} + \abs{m}) \pi\).
The systoles correspond to \((1, 0)\), \((-1, 0)\), \((0, 1)\) and \((0, -1)\).
Each of the elements \((1, 1)\), \((-1, 1)\), \((1, -1)\) and \((-1, -1)\) have two minimal free decompositions in either two systolic elements or in itself.
In particular, the orthogonal singularities \((1, 0)\) and \((0, 1)\) do not repulse each other in the renormalised energy (see \cref{proposition_coercivity}).

\subsubsection{Equilateral torus}
The equilateral torus is \(\manifold{N} = \Cset / H\) with \(H = \{ n + m e^{2 \pi i/3} \st n, m \in \Zset\}\).
One has then \(\pi_1 (\manifold{N}) = H\)
and for every \(h \in \pi_1 (\manifold{N})\), \(\equivnorm{h} = \abs{h}\) (the Euclidean norm).
Closed geodesics in the homotopy class of some \(h\in H\) lift into line segments whose endpoints differ by the vector \(h\).
In particular, all closed geodesics in a given homotopy class are synharmonic and there are six atomic minimising closed geodesics, corresponding to the neighbouring vectors \(e^{2 k i \pi/3} \in \pi_1 (\manifold{N})\) with \(k \in \{0, \dotsc, 5\}\) and achieving the length of the systole.
Other homotopy classes have minimal free decomposition into these six classes; each minimal free decomposition contains only two neighbouring vectors.
This example shows that the geometry of \(\manifold{N}\) matters in the decompositions of homotopy classes.

\subsubsection{Projective spaces and orthogonal group}
We consider the case of real projective spaces \(\Rset \mathbb{P}^n\) with \(n \ge 2\).
In particular \(\Rset \mathbb{P}^2\) arises in nematic liquid crystals models, \(\Rset \mathbb{P}^3 \simeq SO(3)\)  appears in models of superfluid Helium 3 (\(\prescript{3}{}{\mathrm{He}}\)) in the dipole-locked A phase.

The projective space \(\Rset\mathbb{P}^n\) can be realized as the set of orthogonal projections of \(\Rset^n\) whose trace is \(1\), or equivalently as polynomials
\begin{multline*}
  \{P : \Rset^n \to \Rset \st \text{\(P\) is a homogeneous polynomial of degree \(2\), } \Delta P = 1 \text{ and } \abs{\Deriv  P}^2 = P \}.
\end{multline*}
This latter space is a subset of an affine space of dimension \(\frac{(n - 1)(n + 2)}{2}\).
The \(\mathcal{Q}\)-tensor representation \cite{Ball_Zarnescu_2011} corresponds to the corresponding trace-less tensors.

More usually, the projective space is obtained from the sphere by identifying antipodal points, that is \(\Rset \mathbb{P}^n = \Sset^n / \{+\id,-\id\}\). Hence, when \(n \ge 2\), \(\pi_1 (\Rset \mathbb{P}^n) = \Zset_2 = \Zset/2\Zset\).
Closed geodesics in \(\Rset \mathbb{P}^n\) are the images of geodesics between antipodal points in \(\mathbb{S}^n\) by the canonical projection \(\pi: \mathbb{S}^n \rightarrow \Rset \mathbb{P}^n\).
Minimising closed geodesic that are topologically non-trivial are all homotopic, synharmonic and systolic, since they all have the same length.
The minimal free decompositions of the different homotopy classes is summarized in \cref{talbe_RPn}.

\begin{table}
\caption{Decomposition of closed geodesics of the real projective space \(\Rset \mathbb{P}^n = \Sset^n/\{\id, -\id\}\)}
\label{talbe_RPn}

\renewcommand{\arraystretch}{1.2}
\begin{tabular}{ccc c c c}
\toprule
\(\gamma\) & Description & Conjugates & \(\equivnorm{\gamma}\) & Decompositions & \(\Esing(\gamma)\)\\
\midrule
\(\gamma_{\mathrm{c}}\) & constant & 1 & 0 &  & 0\\
\(\gamma_\mathrm{a}\) & geodesic between antipodal points& 1 & \(\pi\) & \(\gamma_\mathrm{a}\) & \(\frac{\pi}{4}\)\\
\bottomrule
\end{tabular}

\end{table}

\subsubsection{Superfluid Helium 3 in the dipole-free A phase}
\label{sect_gooWae7aNu2roatoozahb0al}

In superfluid \ce{^3He} in the dipole-free A phase, one considers the manifold \(\manifold{N} = (SU (2)\times SU(2))/H\), obtained by quotienting the Lie group \(SU (2)\times SU(2)\), endowed with the product \((h_1,h_2)(g_1,g_2)=(h_1g_1,h_2g_2)\), by its closed subgroup
\[H\defeq  \bigcup_{k=0}^3 (\mathbf{k}, \mathbf{i})^k H_0, \quad \text{ with } H_0=\{(\cos \theta  \mathbf{1}+ \sin \theta\,\mathbf{i},\mathbf{1})\st \theta \in [0,2\pi] \}, \]
under the identification of  \(SU (2)\) with the unit sphere in the field of quaternions \(\Hset \).
The connected component of the identity in \(H\) is \(H_0\).
By \cref{proposition_fund_group_homog_space}, \(\pi_{1} (\manifold{N}) \simeq H/H_0\simeq \mathbb{Z}_4 \), as \((\mathbf{k},\mathbf{i})\) is an element of order \(4\). Since \(\pi_1 (\manifold{N})\) is abelian, we have four homotopy classes of maps from \(\Sset^1\) to \(\manifold{N}\).
By \cref{prop_homot_geodesics_lifting}, we can reduce our study of minimal lengths in homotopy classes to geodesics in \(SU (2) \times SU (2)\) from \((\mathbf{1},\mathbf{1})\) to points of \(H\) that minimise the distance from \( (\mathbf{1},\mathbf{1})\) to a connected component of \(H\).
For the first homotopy class, corresponding to the element \(1 \in \Zset_4 \simeq \pi_1(\manifold{N})\), we consider a geodesic to the connected component \((\mathbf{k},\mathbf{i})H_0\).
For points in \((\mathbf{k},\mathbf{i})H_0\), the second component is fixed to \( \mathbf{i}\), which is at distance \(\pi\) from \(\mathbf{1}\). The first components can be written \(\cos \theta\,\mathbf{k}+\sin\theta\,\mathbf{j}\), which are all at equal distance \(\pi\) from \(\mathbf{1}\). This means that geodesics from \((\mathbf{1},\mathbf{1})\) to any point of \((\mathbf{k},\mathbf{i})H_0\) all minimise the distance, and their lengths are equal to
\(\sqrt{\pi^2 + \pi^2} = \sqrt{2}\,\pi\).
The same applies to \(-1 \in \Zset_4 \simeq \pi_1(\manifold{N})\), since it is obtained by reversing the orientation of curves.
For the element \(2 \in \Zset_4\simeq\pi_1(\manifold{N})\), we have \( (\mathbf{k}^2, \mathbf{i}^2)H_0=(-\mathbf{1},-\mathbf{1})H_0=\{( -\cos \theta\,\mathbf{1} -\sin \theta\,\mathbf{i},-\mathbf{1}); \theta \in [0,2\pi)\}\). The geodesic connecting \( (\mathbf{1},\mathbf{1})\) and \( (\mathbf{1},-\mathbf{1})\), corresponding to \(\theta = \pi\), has length \(2\pi\).
The length of the systole corresponds to \(\sqrt{2} \pi\). All the corresponding geodesics can be obtained from each other by conjugation with elements of \(H_0\), and are thus all synharmonic by \cref{proposition_homotopy_synharmonic}.
The properties of the minimal free decomposition are summarized in \cref{talbe_3He_free}.
Interestingly, every homotopy class is atomic and there is an atomic class that has two alternate minimal decompositions.

\begin{table}
\caption{Decomposition of closed geodesics of \((SU (2) \times SU (2)) / H\)}
\label{talbe_3He_free}

\renewcommand{\arraystretch}{1.2}
\begin{tabular}{ccc c c c}
\toprule
\(\gamma\) & Description & Conjugates & \(\equivnorm{\gamma}\) & Decompositions & \(\Esing(\gamma)\)\\
\midrule
\(\gamma_{0}\) & constant & 1 & 0 &  & 0\\
\(\gamma_{\pm 1}\) & \(180\degree\) rotation & 1 & \(\sqrt{2}\pi\) & \(\gamma_{\pm 1}\) & \(\frac{ \pi}{2}\)\\[1ex]
\(\gamma_{2}\) & \(360\degree\) rotation & 1 & \(2\pi\) & 
 \parbox[c]{2cm}{\begin{center}\(\gamma_2\)\\
 \(\gamma_{+1} \;\gamma_{+1}\)\\
 \(\gamma_{-1} \;\gamma_{-1}\)\end{center}}
 & \(\pi\)\\
\bottomrule
\end{tabular}
\end{table}

\subsubsection{Orthorombic space}
We consider the case where \(\manifold{N}= SO(3)/{D_2}\) where \(D_2 \subset SO (3)\) is the
group with four elements corresponding to the identity and \(3\) rotations of \(180 \degree\) with respect to three mutually perpendicular axes.
This manifold arises in biaxial nematics models.
Under the double covering of \(SO (3)\) by \(SU (2) \simeq \Sset^3 \subset \Hset\), we have
\(\manifold{N}= SU(2)/Q\), where \(Q = \{\pm 1, \pm \mathbf{i}, \pm \mathbf{j}, \pm \mathbf{k}\}\) is the \emph{quaternion group}.
By \cref{proposition_fund_group_homog_space}, we have \(\pi_1(SO(3)/D_2)=\pi_1(SU(2)/Q)=Q\), which is non-abelian.
By the anticommutativity of the fundamental quaternion units \(\mathbf{i}\), \(\mathbf{j}\) and \(\mathbf{k}\), the quaternion group \(Q\) contains five conjugacy classes
\(\{1\}\), \(\{-1\}\), \(\{\pm \mathbf{i}\}\), \(\{\pm \mathbf{j}\}\) and  \(\{\pm \mathbf{k}\}\),
corresponding geometrically to the trivial geodesic, the \(360\degree\) rotation corresponding to the generator of \(\pi_1 (SO(3))\) and the \(180\degree\) rotations with respect to the coordinate axes. Minimising closed geodesics in \(\manifold{N}=SU(2)/Q\) lift to arcs of great circles in \(\Sset^3\simeq SU(2)\).
We deduce therefrom, or by \cref{corollary_discreteSUdQuotient}, that homotopic geodesics are  synharmonic.
The geodesics corresponding to the \(180\degree\) rotations have length \( \pi\) and are systolic;
since \((2\pi)^2 > 2 \pi^2\), the homotopy class of the \(360\degree\) rotation is not atomic and has \emph{three minimal free decompositions} in twice a \(180\degree\) rotation.
Interestingly, in this case we have a homotopy class that has multiple minimal free decompositions.
\begin{table}
\caption{Decomposition of closed geodesics of \(SO(3)/D_2 \simeq SU (2)/Q\).}

\renewcommand{\arraystretch}{1.2} \
\begin{tabular}{ccc c c c}
\toprule
\(\gamma\) & Description & Conjugates in \(Q\) & \(\equivnorm{\gamma}\) & Decompositions & \(\Esing(\gamma)\)\\
\midrule
\(\gamma_{\mathrm{c}}\) & constant & 1 & 0 &  & 0\\
\(\gamma_{\mathrm{x}}\) & \(180\degree\) rotation around the \(x\)-axis& 2 & \(\pi\) & \(\gamma_{\mathrm{x}}\) & \(\frac{\pi}{4}\)\\
\(\gamma_{\mathrm{y}}\) & \(180\degree\) rotation around the \(y\)-axis& 2 & \(\pi\) & \(\gamma_{\mathrm{y}}\) & \(\frac{\pi}{4}\)\\
\(\gamma_\mathrm{z}\) & \(180\degree\) rotation around the \(z\)-axis& 2 & \(\pi\) & \(\gamma_\mathrm{z}\) & \(\frac{\pi}{4}\)\\[.5em]
\(\gamma_{\mathrm{w}}\) & \(360\degree\) rotation & 1 & \(2 \pi\) &
\parbox[c]{2cm}{\begin{center}\(\gamma_\mathrm{x} \; \gamma_\mathrm{x}\)\\
\(\gamma_\mathrm{y} \;\gamma_\mathrm{y}\)\\
\(\gamma_\mathrm{z} \;\gamma_\mathrm{z}\)\end{center}}
& \(\frac{\pi}{2}\)\\
\bottomrule
\end{tabular}

\end{table}
\subsubsection{Tetrahedral space}

The configurations of a regular tetrahedron in \(\Rset^3\) are parametrized by the manifold \(SO (3)/T\), where \(T\) is the tetrahedral group of direct isometries that preserve a tetrahedron.
Under the double covering of \(SO (3)\) by \(SU (2)\), \(T\) is the image of the \emph{binary tetrahedral group} \(2T\), a group of order \(24\) which is generated by \(\frac{1}{2} (1 +  \mathbf{i} + \mathbf{j} + \mathbf{k})\) and
\(\frac{1}{2} (1 +  \mathbf{i} + \mathbf{j} - \mathbf{k})\).
By \cref{proposition_fund_group_homog_space}, we have \(\pi_1 (\manifold{N}) \simeq 2T\).
The conjugacy classes corresponding to the possible rotations of a face (or equivalently a vertex) and of an edge are described on \cref{table_tetrahedral}.
The systole corresponds to rotations of faces by \(\pm 120\degree\).
These two homotopy classes are atomic.
The \(240\degree\) rotations can be decomposed into two \( 120\degree\) rotations, and are not atomic because \((\frac{4\pi}{3})^2 > 2 (\frac{2\pi}{3})^2\).
The \(180\degree\) rotations of an edge decompose into \(120\degree\) and \(-120\degree\) rotation,
this decomposition is minimal since \((\pi)^2 > 2 (\frac{2\pi}{3})^2\).
Similarly, since \((2\pi)^2 > 3 (\frac{2\pi}{3})^2\), the \(360\degree\) decomposes in either three \(120\degree\) rotations or three \(-120\degree\) rotations.
By \cref{corollary_discreteSUdQuotient}, all homotopic geodesics are synharmonic.
Interestingly, this example features two atomic systolic homotopy classes.

\begin{table}
\caption{Decomposition of closed geodesics of \(SO(3)/ T \simeq SU (2)/2 T\).}
\label{table_tetrahedral}
\renewcommand{\arraystretch}{1.2}
\begin{tabular}{ccc c c c}
\toprule
\(\gamma\) & Description & Conjugates in \(2T\) & \(\equivnorm{\gamma}\) & Decompositions & \(\Esing(\gamma)\)\\
\midrule
\(\gamma_{\mathrm{c}}\) & constant & 1 & 0 &  & 0\\
\(\gamma_{+}\) & \(120\degree\) rotation of a face& 4 & \(\frac{2\pi}{3}\) & \(\gamma_{+}\) & \(\frac{\pi}{9}\)\\
\(\gamma_{-}\) & \(-120\degree\) rotation of a face & 4 & \(\frac{2\pi}{3}\) & \(\gamma_{-}\) & \(\frac{\pi}{9}\)\\
\(\gamma_{+}^2\) & \(240\degree\) rotation of a face& 4 & \(\frac{4\pi}{3}\) & \(\gamma_{+} \; \gamma_{+}\) & \(\frac{2\pi}{9}\)\\
\(\gamma_{-}^2\) & \(-240\degree\) rotation of a face& 4 & \(\frac{4\pi}{3}\) & \(\gamma_{-} \; \gamma_{-}\) & \(\frac{2\pi}{9}\)\\
\(\gamma_{e}\) & \(180\degree\) rotation of an edge & 6 & \(\pi\) & \(\gamma_{+} \; \gamma_{-}\) & \(\frac{2\pi}{9}\)\\[1ex]
\(\gamma_{\mathrm{w}}\) & \(360\degree\) rotation & 1 & \(2 \pi\) &
\parbox[c]{2cm}{\begin{center}\(\gamma_{+} \; \gamma_{+} \; \gamma_{+}\)\\
\(\gamma_{-} \; \gamma_{-} \; \gamma_{-}\)
\end{center}}
& \(\frac{\pi}{3}\)\\
\bottomrule
\end{tabular}

\end{table}
\subsubsection{Octahedral space}

The octahedral space describes the configurations of a regular octahedron centered at the origin in the Euclidean space \(\Rset^3\) up to the symmetries of the octahedron, or equivalently of a cube up to its symmetries.
It is parametrised by \(\manifold{N} = SO(3)/O\), where \(O\) is the octahedral group of direct isometries that preserve an octahedron or, equivalently, a cube.
The action of \(O\) on the four pairs of antipodal faces of the octahedron, identifies \(O\) with the symmetric group \(S_4\).
The space \(\manifold{N}\) can be obtained as a subset of an affine \(9\)--dimensional subspace of fourth order tensors satisfying a quadratic condition \cite{Chemin_Henrotte_Remacle_VanSchaftingen}:
\begin{multline*}
\{ P : \Rset^3 \to \Rset \st
P \text{ is a homogeneous polynomial of degree \(4\)},\\
  \text{for every \(x \in \Rset^3\) }
\Delta P (x)= \tfrac{1}{12} \abs{x}^2
\text{ and } \abs{\Deriv ^2 P (x)}^2 = P(x)
\}.
\end{multline*}
Under the universal covering of \(SO (3)\) by \(SU (2)\), the octahedral group \(O\) lifts to the \emph{binary octahedral} group \(2O\), which is generated by the quaternions \(\frac{1}{2} (1 +  \mathbf{i} + \mathbf{j} + \mathbf{k})\)
and \(\frac{1}{\sqrt{2}} (1 +  \mathbf{i})\).
There are 8 conjugacy classes described in \cref{table_octahedron}.
There is one systolic homotopy class corresponding to the shortest rotation of \(90\degree\) of a vertex of the octahedron.
The second shortest is the \(120\degree\) rotation around a face.
It can be decomposed into two \(90\degree\) rotations of vertices, but this decomposition is not minimal since \(2(\frac{\pi}{2})^2 > (\frac{2\pi}{3})^2\); hence the \(120\degree\) rotation is also atomic. The \(180\degree\) rotation around an edge is not atomic since it can be decomposed into \(90\degree\) rotation of a vertex and \(120\degree\) rotation of a face, and \((\frac{\pi}{2})^2+(\frac{2\pi}{3})^2<\pi^2\).
The other homotopy classes can be treated by a combinatorial analysis of the different possible decompositions whose results are summarized in \cref{table_octahedron}. 
By \cref{corollary_discreteSUdQuotient}, all homotopic geodesics are synharmonic. An interesting feature of the octahedral space is that there are two atomic homotopy classes.
\begin{table}
\caption{Decomposition of closed geodesics in the octahedral space \(SO(3)/O \simeq SU (2)/2O\).}
\label{table_octahedron}

\renewcommand{\arraystretch}{1.2}
\begin{tabular}{ccc c c c}
\toprule
\(\gamma\) & Description & Conjugates in \(2 O\) & \(\equivnorm{\gamma}\) & Decomposition & \(\Esing(\gamma)\)\\
\midrule
\(\gamma_{\mathrm{c}}\) & constant & 1 & 0 &  & 0\\
\(\gamma_\mathrm{v}\) & \(90\degree\) rotation of a vertex & 6 & \(\frac{\pi}{2}\) & \(\gamma_\mathrm{v}\) & \(\frac{\pi}{16}\)\\
\(\gamma_\mathrm{f}\) & \(120\degree\) rotation of a face & 8 & \(\frac{2 \pi}{3}\) & \(\gamma_\mathrm{f}\) & \(\frac{\pi}{9}\)\\
\(\gamma_\mathrm{v}{}^2\) & \(180\degree\) rotation of a vertex & 6 & \(\pi\) & \(\gamma_\mathrm{v} \, \gamma_\mathrm{v}\) & \(\frac{\pi}{8}\)\\
\(\gamma_\mathrm{e}\) & \(180\degree\) rotation of an edge & 12 & \(\pi\) & \(\gamma_\mathrm{v} \, \gamma_\mathrm{f}\) & \(\frac{25 \pi}{144}\)\\
\(\gamma_\mathrm{v}{}^3\) & \(270\degree\) rotation of a vertex & 6 & \(\frac{3\pi}{2}\) & \(\gamma_\mathrm{v} \, \gamma_\mathrm{v} \, \gamma_\mathrm{v}\) & \(\frac{3\pi}{16}\)\\
\(\gamma_\mathrm{f}{}^2\) & \(240\degree\) rotation of a face & 8 & \(\frac{4\pi}{3}\) & \(\gamma_\mathrm{f} \, \gamma_\mathrm{f}\) & \(\frac{2\pi}{9}\)\\
\(\gamma_{\mathrm{w}}\) & \(360\degree\) rotation & 1 & \(2 \pi\) & \(\gamma_\mathrm{v} \,\gamma_\mathrm{v} \, \gamma_\mathrm{v} \, \gamma_\mathrm{v}\) & \(\frac{\pi}{4}\)\\
\bottomrule
\end{tabular}

\end{table}
\subsubsection{Icosahedral space}
\begin{table}
\caption{Decomposition of closed in the icosahedral space  \(SO(3)/I\simeq SU (2)/2I\).}
\label{table_dodecahedral_space}

\renewcommand{\arraystretch}{1.2}
\begin{tabular}{ccc c c c}
\toprule
\(\gamma\) & Description & Conjugates in \(2I\) & \(\equivnorm{\gamma}\) & Decomposition & \(\Esing(\gamma)\)\\
\midrule
\(\gamma_{\mathrm{c}}\) & constant & 1 & 0 &  & 0\\
\(\gamma_\mathrm{v}\) & \(72\degree\) rotation of a vertex & 12 & \(\frac{2\pi}{5}\) & \(\gamma_\mathrm{v}\) & \(\frac{\pi}{25}\)\\
\(\gamma_\mathrm{f}\) & \(120\degree\) rotation of a face & 20 & \(\frac{2 \pi}{3}\) & \(\gamma_\mathrm{v} \, \gamma_\mathrm{v}\) & \(\frac{2\pi}{25}\)\\
\(\gamma_\mathrm{v}{}^2\) & \(144\degree\) rotation of a vertex & 12 & \(\frac{4\pi}{5}\) & \(\gamma_\mathrm{v} \, \gamma_\mathrm{v}\) & \(\frac{2\pi}{25}\)\\
\(\gamma_\mathrm{e}\) & \(180\degree\) rotation of an edge & 30 & \(\pi\) & \(\gamma_\mathrm{v} \, \gamma_\mathrm{v} \, \gamma_\mathrm{v}\) & \(\frac{3\pi}{25}\)\\
\(\gamma_\mathrm{v}{}^3\) & \(216\degree\) rotation of a vertex & 12 & \(\frac{6\pi}{5}\) & \(\gamma_\mathrm{v} \, \gamma_\mathrm{v} \, \gamma_\mathrm{v}\) & \(\frac{3\pi}{25}\)\\
\(\gamma_\mathrm{f}{}^2\) & \(240\degree\) rotation of a face & 20 & \(\frac{4\pi}{3}\) & \(\gamma_\mathrm{v} \, \gamma_\mathrm{v} \, \gamma_\mathrm{v} \, \gamma_\mathrm{v}\) & \(\frac{4\pi}{25}\)\\
\(\gamma_\mathrm{v}{}^4\) & \(288\degree\) rotation of a vertex & 12 & \(\frac{8\pi}{5}\) & \(\gamma_\mathrm{v} \, \gamma_\mathrm{v} \, \gamma_\mathrm{v} \, \gamma_\mathrm{v}\) & \(\frac{4\pi}{25}\)\\
\(\gamma_{\mathrm{w}}\) & \(360\degree\) rotation & 1 & \(2 \pi\) & \(\gamma_\mathrm{v} \,\gamma_\mathrm{v} \, \gamma_\mathrm{v} \, \gamma_\mathrm{v} \, \gamma_\mathrm{v}\) & \(\frac{\pi}{5}\)\\
\bottomrule
\end{tabular}

\end{table}

This space \(\manifold{N} = SO (3)/I\), also known as the \emph{Poincar\'e homology sphere}, is obtained by quotienting the rotation group \(SO (3)\) by the group \(I\) of direct isometries preserving the icasahedron or, equivalently, the dodecahedron.
The group \(I\) is isomorphic to the alternating group \(A_5\) of even permutations of \(5\) elements.
Under the double covering by \(SU(2)\) of \(SO(3)\),
we have \(SO(3)/I \simeq SU (2) / 2I\), where \(2I\) is the binary icosahedral group, generated by the quaternions \(\frac{1}{2} (1 +  \mathbf{i} + \mathbf{j} + \mathbf{k})\) and \(\tfrac{1}{2} (\frac{\sqrt{5} + 1}{2} + \frac{\sqrt{5} - 1}{2} \mathbf{i} + \mathbf{j})\).
The geodesic corresponding to \(72\degree\) rotation around the vertex of the icosahedron or, equivalently, of a face of the dodecahedron, is systolic and is the only atomic conjugacy class.
The conjugacy classes and their minimal decompositions are described in \cref{table_dodecahedral_space}.
By \cref{corollary_discreteSUdQuotient}, all homotopic geodesics are synharmonic.

\section{Computing some renormalised energies}
\label{section_simple_boundary_conditions}
The next result shows that the renormalised energy of simple boundary conditions  can be computed in terms of the renormalised energy of Bethuel--Brezis--Hélein for the circle \(\mathbb{S}^1\).

\begin{theorem}%
\label{proposition_explicit_renormalised_energy}
Let \(\Omega\subset \Rset^2\) be a simply connected bounded Lipschitz domain. If \(\sigma : \Sset^1 \to \manifold{N}\) is a minimising geodesic, 
\(g \in (BV \cap W^{1/2, 2}) (\partial \Omega, \Sset^1)\) is length minimising, and \(\gamma_1 \in \mathcal{C}^1 (\Sset^1, \Sset^1)\) is a minimising geodesic, then for every \(a \in \Omega\),
\[
    \mathcal{E}^{\mathrm{geom}}_{\sigma \compose g, \sigma \compose \gamma_1} (a)
  =
    \biggl(\frac{\equivnorm{\sigma}}{2 \pi}\biggr)^2
    \mathcal{E}^{\mathrm{geom}}_{g, \gamma_1} (a).
\]
\end{theorem}

When \(\Omega=B_1(0)\) and \(\gamma_1(x)=g(x)=x\) for every \(x\in\Sset^1\), \(\mathcal{E}^{\mathrm{geom}}_{g, \gamma_1}\) has a unique minimiser at \(a=0\) \cite{Bethuel_Brezis_Helein_1994}*{\S VIII.4.} (see also \cite{Lefter_Radulescu_1996}). Thus, the same holds for the renormalised energy in \(\manifold{N}\): \(\mathcal{E}^{\mathrm{geom}}_{\sigma , \sigma} (a)\) has a unique minimiser at \(a=0\).

\begin{lemma}
\label{lemma_lower_bound_H}
Let \(\Omega\subset \Rset^2\) be a simply connected bounded open set with smooth boundary and let \(v \in \mathcal{C}^1 (\Bar{\Omega} \setminus B_\rho (a),\manifold{N})\) with \(a\in\Omega\) and \(0<\rho<\dist(a,\partial\Omega)\).
Assume that
\[
    \int_{\partial \Omega}
      \abs{\partial_t v}
  =
    \int_{\partial B_{\rho} (a)}
      \abs{\partial_t v}
  =
    \equivnorm{v \vert_{\partial \Omega}},
\]
where \(\abs{\partial_tv}\) denotes the norm of the tangential derivative. If \(H \in \mathcal{C}^1 (\Bar{\Omega}\setminus B_\rho(a), \Rset) \cap \mathcal{C}^2 (\Omega\setminus \bar{B}_\rho(a), \Rset)\) satisfies
\begin{equation}
\label{eq_Choozeu6ooZ8Mih9soh1iePe}
 \left\{
 \begin{aligned}
   \Delta H &= 0 &&\text{in \(\Omega \setminus B_{\rho} (a)\)},\\
   \partial_{n} H &= \abs{\partial_t v} &&\text{on \(\partial \Omega\)},\\
   \partial_{n} H &= \abs{\partial_t v} &&\text{on \(\partial B_{\rho} (a)\)},
 \end{aligned}
 \right.
 \end{equation}
where \(\partial_n\) denotes the exterior normal derivative to \(\Omega\) and to \(B_{\rho} (a)\), then,
\[
  \int_{\Omega\setminus B_\rho(a)} \abs{\Deriv  H \wedge dv} \ge \int_{\Omega\setminus B_\rho(a)} \abs{\nabla H}^2,
\]
where \(\Deriv H \wedge dv = \partial_1 H \partial_2 v - \partial_2 H \partial_1 v\).
\end{lemma}

\begin{proof}[Proof of \cref{lemma_lower_bound_H}]
Without loss of generality we can assume that \( \equivnorm{v \vert_{\partial \Omega}}>0\). By Sard's lemma, for almost every \(s \in \Rset\), we have 
\begin{equation}
\label{sard}
\begin{split}
&\Deriv H \ne 0 \quad\text{on } \Gamma_s\defeq \{x\in\Omega\setminus\Bar{B}_\rho(a)\st H(x)=s\},\\
&\partial_t H\ne 0\quad \text{on }\{x\in\partial(\Omega\setminus\Bar{B}_\rho(a))\st H(x)=s\},
\end{split}
\end{equation}
where \(\partial_t\) denotes the tangential derivative. We have thus by the coarea formula:
\begin{equation}\label{coareaFirst}
  \int_{\Omega\setminus\Bar{B}_\rho(a)} |dH \wedge dv| =\int_\Rset \int_{\Gamma_s}\frac{|dH \wedge d v|}{|d H|}\dif s
  =
  \int_{\Rset} \biggl(\int_{\Gamma_s} \abs{\partial_t v} \biggr) \dif s,
\end{equation}
where \(\partial_t v\) denotes the tangential derivative to the locally smooth curve \(\Gamma_s\) for almost every \(s\).

Fix \(s\in\Rset\) satisfying \eqref{sard} and define the following subsets of \(\Omega\setminus\Bar{B}_\rho(a)\) by \(A_s\defeq H^{-1}((-\infty, s))\) and \(A^s\defeq H^{-1}([s, +\infty))\). By \eqref{sard}, the intersections of \(\partial A_s\) and \(\partial A^s\) with \(\Omega \setminus \Bar{B}_{\rho} (a)\) coincide with \(\Gamma_s\). Moreover, by \eqref{sard} and by the implicit function theorem, \(\Gamma_s\) is a locally finite union of smooth curves that cross \(\partial (\Omega\setminus \Bar{B}_\rho(a))\)transversally; hence, \(\Gamma_s\) is a finite union of smooth curves in \(\Omega\setminus \Bar{B}_\rho(a)\) which, by harmonicity, converge at their endpoints to points on \(\partial (\Omega\setminus \Bar{B}_\rho(a))\). In particular, we obtain that the open set \(A_s\cup\Bar{B}_\rho (a)\) has Lipschitz boundary.

Now, we claim that  the boundary of every nonempty connected component of \(A_s\) touches \(\partial B_\rho (a)\) and the boundary of every nonempty connected component of \(A^s\) touches \(\partial \Omega\). Indeed, let \(C\subset A_s\) be a non-empty connected component of \(A_s\), since \(A_s\) is open, then \(C\) is also open. We also observe from what precedes that \(C\) has Lipschitz boundary. If \(C\) does not meet \(\partial (\Omega \setminus \bar{B}_\rho(a))\) then we have that \(H\) is harmonic in \(C\) and \(H =s\) on \(\partial C\). We then deduce from the maximum principle that \(H\) is constant in \(C\) and, by the isolated zeros property, \(H\) is constant in all \(\Omega \setminus \bar{B}_\rho(a)\). This is a contradiction with our assumption that \(  \equivnorm{v \vert_{\partial \Omega}}>0\) and hence \( C\) meets \(\partial (\Omega \setminus \bar{B}_\rho(a))\). Now, by contradiction we assume that \(C\) does not meet \(\partial B_\rho(a)\). Then we can decompose \(\partial C = (\partial C)_1 \cup (\partial C)_2\) where the union is disjoint and \((\partial C)_1 \subset \Gamma_s\) and \((\partial C)_2 \subset \partial \Omega\). From the divergence theorem we find that
\begin{align*}
0&= \int_C \Delta H =\int_{\partial C} \partial_n H=  \int_{(\partial C)_1} \partial_n H +\int_{(\partial C)_2} \partial_n H \\
&= \int_{ (\partial C)_1} \nabla H \cdot \frac{\nabla H}{|\nabla H|}+\int_{(\partial C)_2} |\partial_t v|.
\end{align*}
We have used that \(\nabla H \) is normal to \( \Gamma_s\) and points in the outer direction of \(A_s\). We again arrive at the contradiction since \( \nabla H \neq 0\) on \(\Gamma_s\). Hence every nonempty connected component of \(A_s\) touches \(\partial B_\rho (a)\). In a similar way we can prove that every nonempty connected component of \( A^s\) touches \(\partial \Omega\).
This implies that \(A_s\cup\Bar{B}_\rho (a)\) and its complement \(A^s\cup(\Rset^2\setminus\Omega)\) are connected. Since the open set \(A_s\cup B_\rho (a)\) has Lipschitz boundary, we deduce that \(A_s\cup B_\rho (a)\) is a simply connected domain \cite{ahlforsComplexAnalysisIntroduction1978}*{\S 4.2}. Since \(\Omega\) is also simply connected, we get that the parametrizations of \(\partial\Omega\) and of \(\partial (A_s\cup\Bar{B}_\rho (a))\) are homotopic in \(\Bar{\Omega}\setminus(A_s\cup\Bar{B}_\rho (a))\). Hence, \(\tr_{\partial (\Bar{B}_\rho (a) \cup A_s)}v\) and \(\tr_{\partial \Omega}v\) are the image of free-homotopic Lipschitz-continuous curves, and thus we have, since \(H\) is harmonic,
\begin{equation}
\label{eq_nePhupheiL2eiquinoom2uon}
\begin{split}
    \int_{\partial (B_\rho (a)\cup A_s)}
        \abs{\partial_t v}
  &\ge
   \equivnorm{v \vert_{\partial (B_\rho (a) \cup A_s)}}\\
   &= 
      \equivnorm{v \vert_{\partial \Omega}}
=\int_{\partial \Omega} \abs{\partial_t v}       
      = \int_{\partial \Omega} \partial_n H 
      = \int_{\partial (B_\rho (a) \cup A_s)} \partial_n H.
    \end{split}
\end{equation}
We deduce from \eqref{eq_Choozeu6ooZ8Mih9soh1iePe} and \eqref{eq_nePhupheiL2eiquinoom2uon} that for almost every \(s \in \Rset\),
\begin{equation}
\label{eq_nie5iephex7Koolahw5aequo}
\begin{split}
\int_{\Gamma_s}\abs{\partial_t v}&=\int_{\partial (B_\rho (a) \cup A_s)}\abs{\partial_t v}-\int_{\bigl(\partial B_\rho(a)\cup \,\partial \Omega\bigr)\cap\,\partial\bigl(B_\rho (a) \cup A_s\bigr)}\abs{\partial_t v}\\
&\ge\int_{\partial (B_\rho (a) \cup A_s)}
\partial_n H-\int_{\bigl(\partial B_\rho(a)\cup\, \partial \Omega\bigr)\cap\,\partial\bigl(B_\rho (a) \cup A_s\bigr)}\partial_n H=\int_{\Gamma_s} \partial_n H=\int_{\Gamma_s}\abs{\nabla H},
\end{split}
\end{equation}
where \(\partial_n H\) is the normal derivative with respect to the normal \(\nabla H/\abs{\nabla H}\).
Therefore, by \eqref{coareaFirst}, in view of the coarea formula again,
\[
    \int_{\Omega \setminus B_\rho(a) }
      \abs{\Deriv  H \wedge d v}
  \ge
    \int_{\Rset}
      \biggl(\int_{H^{-1} (\{s\})} \abs{\nabla H} \biggr)\dif s
  =
    \int_{\Omega \setminus B_\rho(a)}
      \abs{\nabla H}^2.\qedhere
\]
\end{proof}

\begin{lemma}
\label{lemma_explicit_lower_bound_free_sphere}
Let \(\Omega\subset \Rset^2\) be a simply connected bounded Lipschitz domain, and let \(a\in\Omega\) with \(0<\rho<\dist(a,\partial\Omega)\). If \(g \in W^{1/2, 2} (\partial \Omega, \Sset^1)\) is length-minimising and \(\gamma_1 : \Sset^1 \to \Sset^1\) is a minimising geodesic, then for every minimising geodesic \(\sigma: \Sset^1 \rightarrow \mathcal{N}\) and \(w \in W^{1, 2} (\Omega\setminus \bar{B}_\rho(a), \manifold{N})\) such that \(\operatorname{tr}_{\partial\Omega} w = \sigma \compose g\) and \(\operatorname{tr}_{\partial B_\rho(a)} w = \sigma \compose \gamma_1\), we have
\begin{multline}\label{defMinv}
    \int_{\Omega \setminus \Bar{B}_\rho (a)}
    \frac{\abs{\Deriv  w}^2}{2}
  \ge
     \biggl(\frac{\equivnorm{\sigma}}{2 \pi}\biggr)^2
     \inf \biggl\{
        \int_{\Omega \setminus \Bar{B}_\rho (a)}
        \frac{\abs{\Deriv v}^2}{2}
        \st
          v \in W^{1, 2} (\Omega  \setminus \Bar{B}_\rho (a), \Sset^1), \\                   
         \operatorname{tr}_{\partial \Omega} v = g
          \text{ and }
        \operatorname{tr}_{\Sset^1}  v(a+\rho\,\cdot) = \tau \compose \gamma_1, \  \tau \in SO (2)
     \biggr\}
     .
\end{multline}
\end{lemma}
\begin{proof}
By approximation it suffices to prove the lemma when \(w:\Omega\setminus \Bar{B}_\rho(a)\to\manifold{N}\) is smooth. Similarly by an approximation and extension argument, considering a harmonic extension of \(g:\partial\Omega\to\Sset^1\) to a neighborhood of \(\partial\Omega\) in \(\Rset^2\), we can also assume that \(\Omega\) has smooth boundary, and that \(g:\partial\Omega\to\Sset^1\) is smooth.

Let \(v\) be a minimiser for the variational problem on the right-hand side. In particular, \(v\) is harmonic, hence smooth in \(\Omega\setminus \Bar{B}_\rho (a)\) up to the boundary. Assume that \(\varphi \in \mathcal{C}^1 (\Bar{\Omega}, \Rset)\), \(\varphi = 0\) on \(\partial \Omega\) and \(\varphi\) is constant on \( \partial B_\rho(a) \).
We have 
\begin{equation}
  \label{eq_ochaish0woh5xaimu5aeJooM}
\begin{split}
    0
  &=
    \frac{d}{d \lambda}
      \int_{\Omega \setminus \Bar{B}_\rho (a)}
      \frac{\abs{\Deriv  (e^{i \lambda \varphi} v)^2}}{2}
      \bigg \vert_{\lambda = 0}
  =
  \int_{\Omega \setminus \Bar{B}_\rho (a)}\sum_{k=1,2}
  \operatorname{Re} (\partial_k v\,\overline{\partial_k(i\varphi v)})\\
 &=\int_{\Omega \setminus \Bar{B}_\rho (a)}\sum_{k=1,2}
  \operatorname{Im}(\partial_k v\,\partial_k\varphi\,\overline{v})
   = - i \int_{\Omega \setminus \Bar{B}_{\rho} (a)}
  \nabla \varphi \cdot v^{-1} \nabla v,
  \end{split}
\end{equation}
where the product corresponds to the complex product in \(\Cset \supset \Sset^1\), and we have used that \(v\in\Sset^1\) so that \(v^{-1}\partial_kv\in i\Rset\).
It follows then that the map \(v\) is a harmonic map and satisfies the equation
\begin{equation}
\label{eq_iChaich9Eongohchootheivu}
\dive(v^{-1} \nabla v) = 0 \text{ in \(\Omega \setminus \Bar{B}_{\rho} (a)\)}.
\end{equation}
By integrating by parts in \eqref{eq_ochaish0woh5xaimu5aeJooM} with \(\varphi\) vanishing on \(\partial \Omega\) and non-zero constant on \(\partial \Omega\), we have 
\begin{equation}
\label{eq_harmonic_free_boundary}
    \int_{\partial \Bar{B}_\rho (a)}
    (v^{-1} \partial_n v) \varphi = 0.
\end{equation}

We define the potential \(H : \Omega \setminus \Bar{B}_{\rho} (a)\to \Rset\) by setting,
\(
\nabla^\perp H = \frac{\equivnorm{\sigma}}{2 \pi} v^{-1} \nabla v\).
In view of the Poincaré lemma and \eqref{eq_harmonic_free_boundary}, it follows that \(H\) is well-defined. Moreover, \(H\) is harmonic and it is smooth in \(\Omega \setminus \Bar{B}_{\rho} (a)\) up to the boundary. Now, we compute,
\[
\left\{
\begin{aligned} 
&\abs{\nabla H} = \frac{\equivnorm{\sigma}}{2 \pi}\abs{\nabla v} &&\text{in }\Omega \setminus B_{\rho} (a),\\
&\partial_n H = \frac{\equivnorm{\sigma}}{2 \pi}v^{-1} \partial_t v = \frac{\equivnorm{\sigma}}{2 \pi}\abs{\partial_t v} =  \abs{\partial_t w} && \text{on }\partial \Omega,\\
&\partial_n H = - \frac{\equivnorm{\sigma}}{2 \pi} v^{-1} \partial_t v = - \frac{\equivnorm{\sigma}}{2 \pi} \abs{\partial_t v} = -\abs{\partial_t w}&&\text{on }\partial B_\rho(a).
\end{aligned}
\right.
\]
Here \(n\) denotes the outward unit normal to \( \Omega \setminus \bar{B}_\rho(a)\). By \cref{lemma_lower_bound_H} and the Cauchy--Schwarz inequality, we deduce that
\begin{align*}
  \int_{\Omega \setminus \Bar{B}_\rho (a)}
  \frac{\abs{\Deriv  w}^2}{2}
 &\ge
  \left(\int_{\Omega \setminus \Bar{B}_\rho (a)} |dH \wedge dw| \right)^2 \left( \int_{\Omega \setminus \Bar{B}_\rho (a)} |DH|^2\right)^{-1}
   \\
&   
   \ge
   \int_{\Omega \setminus \Bar{B}_\rho (a)}
   \frac{\abs{\Deriv  H}^2}{2}
 =
   \biggl(\frac{\equivnorm{\sigma}}{2 \pi}\biggr)^2
   \int_{\Omega \setminus \Bar{B}_\rho (a)}
   \frac{\abs{\Deriv v}^2}{2}.\qedhere
\end{align*}
\end{proof}
\begin{proof}%
[Proof of \cref{proposition_explicit_renormalised_energy}]%
We first observe that if \(B_\rho (a) \subset \Omega\) and if \(v \in W^{1, 2} (\Omega \setminus \Bar{B}_{\rho} (a), \Sset^1)\), then by the chain rule and the fact that \(\sigma : \Sset^1 \to \manifold{N}\) is a minimising geodesic, we have
\[
\begin{split}
    \int_{\Omega \setminus B_\rho (a)}
    \frac{\abs{\Deriv  (\sigma \compose v)}^2}{2}
  &
  =
    \int_{\Omega \setminus B_\rho (a)}
    \frac{\abs{\sigma' \compose v}^2
      \abs{\Deriv v }^2}{2}
  =
    \biggl(\frac{\equivnorm{\sigma}}{2 \pi}\biggr)^2
    \int_{\Omega \setminus B_\rho (a)}
    \frac{\abs{\Deriv v }^2}{2}.
\end{split}
\]
It follows then that
\[
    \mathcal{E}^{\mathrm{geom}, \rho}_{\sigma \compose g, \sigma \compose \gamma_1} (a)
  \le
    \biggl(\frac{\equivnorm{\sigma}}{2 \pi}\biggr)^2
    \mathcal{E}^{\mathrm{geom}, \rho}_{g, \gamma_1} (a),
\]
and thus, by \eqref{eq_def_renorm_geom} and the identity \(\equivnorm{\sigma\compose\gamma_1}=\equivnorm{\sigma}\equivnorm{\gamma_1}\),
\begin{equation}
     \mathcal{E}^{\mathrm{geom}}_{\sigma \compose g, \sigma \compose \gamma_1} (a)
  \le
    \biggl(\frac{\equivnorm{\sigma}}{2 \pi}\biggr)^2
    \mathcal{E}^{\mathrm{geom}}_{g, \gamma_1} (a).
\end{equation}

Conversely, assume that \(w \in W^{1, 2} (\Omega \setminus \Bar{B}_\rho (a), \manifold{N})\) satisfies \(\operatorname{tr}_{\partial \Omega} w = \sigma \compose g\) and \(\operatorname{tr}_{\Sset^1}w(a+\rho\,\cdot) = \sigma \compose \gamma_1\). We have by \cref{lemma_explicit_lower_bound_free_sphere}, and with \(v\) a minimiser in the right-hand side of \eqref{defMinv}:
\[
    \int_{\Omega \setminus B_\rho (a)}
    \frac{\abs{\Deriv  w}^2}{2}
  \ge
    \biggl(\frac{\equivnorm{\sigma}}{2 \pi}\biggr)^2
    \int_{\Omega \setminus B_\rho (a)}
    \frac{\abs{\Deriv v}^2}{2},
\]
with \(\operatorname{tr}_{\partial \Omega} v = g\) and \(\operatorname{tr}_{\Sset^1}  v(a+\rho\,\cdot) = \tau \compose \gamma_1\), \(\tau\in SO(2)\). It follows thus, in view of \eqref{eq_def_renorm_geom},
\[
    \mathcal{E}^{\mathrm{geom}}_{\sigma \compose g, \sigma \compose \gamma_1} (a)
  \ge
    \biggl(\frac{\equivnorm{\sigma}}{2 \pi}\biggr)^2
    \mathcal{E}^{\mathrm{geom}}_{g, \tau \compose \gamma_1} (a).
\]
Moreover, by \cref{proposition_homotopy_synharmonic}, we have \(\synhar{\gamma_1}{\tau \compose \gamma_1} = 0\), and we deduce by \cref{proposition_renorm_geom_dep_gamma} that
\[
    \mathcal{E}^{\mathrm{geom}}_{\sigma \compose g, \sigma \compose \gamma_1} (a)
  \ge
     \biggl(\frac{\equivnorm{\sigma}}{2 \pi}\biggr)^2 \mathcal{E}^{\mathrm{geom}}_{g, \gamma_1} (a).\qedhere
\]
\end{proof}

\section*{Conflicts of interest}
On behalf of all authors, the corresponding author states that there is no conflict of interest.

\end{document}